\documentclass[12pt]{amsart}

\pdfoutput=1
\def\inmain{1}
\def\usehyperref{1}

\usepackage{amssymb, amsmath, amsthm}
\usepackage{calligra, mathrsfs}
\usepackage{upgreek}
\usepackage{graphicx}
\usepackage{url}
\usepackage{mathtools}
\usepackage{enumerate}
\usepackage{verbatim}
\usepackage[retainorgcmds]{IEEEtrantools}
\usepackage{tikz-cd}
\usetikzlibrary{positioning}
\usepackage{microtype}

\usepackage[margin=1in,marginparwidth=0.8in, marginparsep=0.1in]{geometry}

\ifx\usehyperref\undefined
\usepackage{xr}
\newcommand\texorpdfstring[2]{#1}
\newcommand\nolinkurl[1]{\url{#1}}
\else
\usepackage{xr-hyper}
\usepackage[linktocpage=true, unicode]{hyperref}
\usepackage{xcolor}
\hypersetup{
    colorlinks,
    linkcolor={red!65!black},
    citecolor={blue!78!black}
    }
\fi

\newcommand{\Acal}{{\mathcal A}}

\newcommand{\Ccal}{{\mathcal C}}
\newcommand{\Dcal}{{\mathcal D}}
\newcommand{\Ecal}{{\mathcal E}}
\newcommand{\Fcal}{{\mathcal F}}

\newcommand{\Ical}{{\mathcal I}}
\newcommand{\Jcal}{{\mathcal J}}

\newcommand{\Pcal}{{\mathcal P}}

\newcommand{\Tcal}{{\mathcal T}}

\newcommand{\acal}{\Acal}
\newcommand{\ccal}{\Ccal}
\newcommand{\dcal}{\Dcal}
\newcommand{\ecal}{\Ecal}

\renewcommand{\SS}{\mathbb{S}}

\usepackage{scalerel,stackengine}

\makeatletter
\g@addto@macro\bfseries{\boldmath}
\makeatother

\newcommand{\kr}{\kern -2pt}

\DeclareMathOperator{\id}{id}
\DeclareMathOperator{\Hom}{Hom}

\DeclareMathOperator{\Fun}{Fun}
\DeclareMathOperator{\ev}{ev}

\DeclareMathOperator{\der}{D}

\DeclareMathOperator{\fib}{fib}
\DeclareMathOperator{\cofib}{cofib}

\newcommand{\enh}{{\text{\normalfont enh}}}
\newcommand{\op}{{\text{\normalfont op}}}

\newcommand{\pre}{{\normalfont \text{pre}}}

\DeclareMathOperator{\Spc}{Spc}
\DeclareMathOperator{\Set}{Set}
\DeclareMathOperator{\Cat}{Cat}

\DeclareMathOperator{\colim}{colim}
\DeclareMathOperator{\Tot}{Tot}
\newcommand{\lex}{{\text{\normalfont lex}}}

\DeclareMathOperator{\Ind}{Ind}
\DeclareMathOperator{\Pro}{Pro}

\newcommand{\acc}{{\normalfont \text{acc}}}

\DeclareMathOperator{\Sp}{Sp}

\newcommand{\cn}{{\normalfont\text{cn}}}

\DeclareMathOperator{\LMod}{LMod}

\let \Bar \relax
\DeclareMathOperator{\Bar}{Bar}

\newcommand{\dk}{\Gamma}
\newcommand{\seq}{{\normalfont \text{seq}}}

\DeclareMathOperator{\Top}{Top}
\DeclareMathOperator{\Sh}{Sh}

\let \Im \relax
\DeclareMathOperator{\Im}{Im}

\newcommand{\reg}{{\text{\normalfont reg}}}
\newcommand{\ex}{{\normalfont \text{ex}}} 
\newcommand{\hyp}{{\normalfont \text{hyp}}}

\newcommand{\comp}{{\normalfont\text{comp}}}
\newcommand{\bounded}{{\normalfont \text{b}}}

\newcommand{\wproj}{{\normalfont \text{wproj}}}

\newcommand{\shchat}{\smash{{\vstretch{.8}{\widehat{\vstretch{1.25}{\Sh(\ccal)}}}}}}

\newtheorem{proposition}[subsubsection]{Proposition}
\newtheorem{lemma}[subsubsection]{Lemma}
\newtheorem{theorem}[subsubsection]{Theorem}
\newtheorem{corollary}[subsubsection]{Corollary}

\newtheoremstyle{note}{8.0pt plus 2.0pt minus 4.0pt}{8.0pt plus 2.0pt minus 4.0pt}{}{}{\bfseries}{.}{.5em}{} 
\theoremstyle{note}
\newtheorem{example}[subsubsection]{Example}
\newtheorem{remark}[subsubsection]{Remark}
\newtheorem{notation}[subsubsection]{Notation}
\newtheorem{construction}[subsubsection]{Construction}
\newtheorem{warning}[subsubsection]{Warning}
\newtheorem{definition}[subsubsection]{Definition}

\AtBeginDocument{\def\MR#1{}}

\usepackage{standalone}

\title{Derived $\infty$-categories as exact completions}

\author{G. Stefanich}

\date{}

\begin{document}

%%%%%%%%%%%%%%%%%%%%%%%%%%%%%%%%%%%%%%%%%%%%%%%%%%%%%%%%%%%%%%%%%%%%%%%%
%%%%%%%%%%%%%%%%%%%%%%%%%%%%%%%%%%%%%%%%%%%%%%%%%%%%%%%%%%%%%%%%%%%%%%%%
%%%%%%%%%%%%%%%%%%%%%%%%%%%%%%%%%%%%%%%%%%%%%%%%%%%%%%%%%%%%%%%%%%%%%%%%
%%%%%%%%%%%%%%%%%%%%%%%%%%%%%%%%%%%%%%%%%%%%%%%%%%%%%%%%%%%%%%%%%%%%%%%%
%%%%%%%%%%%%%%%%%%%%%%%%%%%%%%%%%%%%%%%%%%%%%%%%%%%%%%%%%%%%%%%%%%%%%%%%
%%%%%%%%%%%%%%%%%%%%%%%%%%%%%%%%%%%%%%%%%%%%%%%%%%%%%%%%%%%%%%%%%%%%%%%%

\begin{abstract}
We develop the theory of exact completions of regular $\infty$-categories, and show that the $\infty$-categorical exact completion (resp. hypercompletion) of an abelian category recovers the connective half of its bounded (resp. unbounded) derived $\infty$-category. Along the way, we prove that a finitely complete $\infty$-category is exact and additive if and only if it is prestable, extending a classical characterization of abelian categories. We also establish $\infty$-categorical versions of Barr's embedding theorem and Makkai's image theorem.
\end{abstract}
 
\maketitle

\tableofcontents

\newpage

%%%%%%%%%%%%%%%%%%%%%%%%%%%%%%%%%%%%%%%%%%%%%%%%%%%%%%%%%%%%%%%%%%%%%%%%
%%%%%%%%%%%%%%%%%%%%%%%%%%%%%%%%%%%%%%%%%%%%%%%%%%%%%%%%%%%%%%%%%%%%%%%%
%%%%%%%%%%%%%%%%%%%%%%%%%%%%%%%%%%%%%%%%%%%%%%%%%%%%%%%%%%%%%%%%%%%%%%%%
%%%%%%%%%%%%%%%%%%%%%%%%%%%%%%%%%%%%%%%%%%%%%%%%%%%%%%%%%%%%%%%%%%%%%%%%
%%%%%%%%%%%%%%%%%%%%%%%%%%%%%%%%%%%%%%%%%%%%%%%%%%%%%%%%%%%%%%%%%%%%%%%%
%%%%%%%%%%%%%%%%%%%%%%%%%%%%%%%%%%%%%%%%%%%%%%%%%%%%%%%%%%%%%%%%%%%%%%%%
 
\section{Introduction}

Let $\Acal$ be an abelian category. Then $\Acal$ satisfies the following two conditions:
\begin{enumerate}[\normalfont(1)]
\item Let $X \rightarrow Y$ be a morphism in $\Acal$ and consider the equivalence relation $K = X \times_Y X$ on $X$. Then there is a quotient object $X/K$, which is stable under change of base.
\item Equivalence relations in $\Acal$ are effective: in other words, if $X$ is an object in $\Acal$ and $R \subseteq X \times X$ is an equivalence relation on $X$, then there is a quotient object $X/R$ and furthermore $R = X \times_{X/R} X$.
\end{enumerate}
A category with finite limits satisfying condition (1) is said to be regular, while a regular category satisfying condition (2) is said to be exact. The notions of regularity and exactness were introduced by Barr and Grillet \cite{Exact, Grillet}, and have received attention over the years due to the abundance of examples (every topos and every abelian category is exact, as is any category which is monadic over the category of sets)  and their role in topos theory and categorical logic \cite{Makkai, Carboni, BCRS, vdBMoerdijkAspects,  Ultracat}. 

The theory of exact categories may be thought of as a nonadditive version of the theory of abelian categories. This is justified by the following result from \cite{Exact}, which is attributed to Tierney:

\begin{theorem}[Tierney]\label{theorem classical}
A finitely complete category is abelian if and only if it is exact and additive.
\end{theorem}

 The first goal of this paper is to record a variant of the above result in the setting of higher category theory. Here the role of equivalence relations is played by groupoid objects, giving rise to the following variants of properties (1) and (2), which may be imposed on a finitely complete $\infty$-category $\Ccal$:
  
 \begin{enumerate}[($1_\infty$)]
 \item The $\infty$-category $\Ccal$ admits geometric realizations of \v{C}ech nerves, which are stable under base change.
 \item Groupoid objects in $\Ccal$ are effective.
  \end{enumerate}
  
  A finitely complete $\infty$-category satisfying property ($1_\infty$) is said to be regular, and if it also satisfies property ($2_\infty$) we say that it is exact. We note that this notion of exact $\infty$-category is different from the one studied by Barwick \cite{BarwickExact}, which is an $\infty$-categorical analogue of Quillen's (as opposed to Barr's) notion of exact category.
 
 To formulate an $\infty$-categorical version of theorem \ref{theorem classical}, one is tasked with finding an $\infty$-categorical analogue of the notion of abelian category. This is supplied by Lurie's theory of prestable $\infty$-categories \cite{SAG}. We will review the definition of prestability in section \ref{subsection prestables}; for now we mention that an $\infty$-category with finite limits $\ccal$ is prestable if and only if there is an equivalence between $\ccal$ and the connective half of a t-structure on some stable $\infty$-category. In the same way that categories of modules over rings are the most basic examples of abelian categories, $\infty$-categories of connective modules over connective ring spectra are the most basic examples of prestable $\infty$-categories.
 
 With these  notions at hand we are ready to formulate our first main result:
 
\begin{theorem}\label{theo infty}
A finitely complete $\infty$-category is prestable if and only if it is exact and additive.
\end{theorem}

One of the central features of exact $\infty$-categories is descent: if $\ccal$ is an exact $\infty$-category then the assignment $\ccal \mapsto \ccal_{/X}$ maps geometric realizations of groupoid objects to totalizations of $\infty$-categories. In addition to descent, in higher category theory one is sometimes interested in hyperdescent. In this case, the role of groupoids is played by a class of semisimplicial objects known as Kan semisimplicial objects (see section \ref{subsection hypercomplete exact} for the definition). An exact $\infty$-category is said to be hypercomplete if every Kan semisimplicial object $X_\bullet$ in $\ccal$ admits a geometric realization, and furthermore an augmentation $X_\bullet \rightarrow Y$ is a geometric realization if and only if $X_\bullet$ is a semisimplicial hypercover of $Y$.  Our next result is a variant of theorem \ref{theo infty} that includes the hypercompleteness condition:

\begin{theorem}\label{theo infinity unbounded}
A finitely complete prestable $\infty$-category is hypercomplete if and only if it is separated and admits geometric realizations.
\end{theorem}

The second goal of this paper is to reexamine the procedure of passage to derived $\infty$-categories from the point of view of theorems \ref{theo infty} and \ref{theo infinity unbounded}. In order to do so, we develop an $\infty$-categorical version of the theory of exact completions of regular categories. The fact that every regular category admits a universal embedding into an exact category was observed early on in the development of the subject by Lawvere and Succi Cruciani \cite{Lawvere, Cruciani}, and there is by now an extensive literature on exact completions fueled by the observation that many categories of interest admit a description as an exact completion, see for instance \cite{Carboni, Rosicky, Lack, Shulman, EPConstructive}.

We will in fact develop two different completion procedures: the exact completion, which is the universal embedding of a regular $\infty$-category into an exact $\infty$-category, and the hypercompletion, which is the universal way of turning a regular $\infty$-category into a hypercomplete exact $\infty$-category. We apply these constructions to exact $1$-categories: note that while an exact $1$-category is almost never exact when regarded as an $\infty$-category, it is nevertheless still regular, so it admits an exact completion and a hypercompletion.  Our next theorem identifies the output of these constructions when applied to an abelian category:

\begin{theorem}\label{theorem intro exact completion abelian}
Let $\acal$ be an abelian category. Then the $\infty$-categorical exact completion $\acal^\ex$ is the connective half of the bounded derived $\infty$-category of $\acal$, while the hypercompletion $\acal^\hyp$ is the connective  half of the unbounded derived $\infty$-category of $\acal$.
\end{theorem}

We emphasize that exact completions and hypercompletions are purely nonadditive constructions. We will deduce theorem \ref{theorem intro exact completion abelian} from a more general recognition principle for exact completions and hypercompletions of   exact $1$-categories, which also allows us to identify their value in other cases of interest:

\begin{itemize}
\item Let $\ccal$ be a small category equipped with a Grothendieck topology. Then the hypercompletion of the category of sheaves of sets $\Sh(\ccal, \Set)$   is the $\infty$-category of hypercomplete sheaves of spaces $\Sh(\ccal, \Spc)^\hyp$, while the exact completion of $\Sh(\ccal, \Set)$ is the $\infty$-category of truncated objects of $\Sh(\ccal, \Spc)^\hyp$.
\item Let $\ccal$ be a small category with finite products. Then the hypercompletion of the category $\Fun^{\times}(\ccal, \Set)$ of product preserving functors from $\ccal$ into sets is the $\infty$-category $\Fun^{\times}(\ccal, \Spc)$ of product preserving functors from $\ccal$ into spaces, while the exact completion of $\Fun^{\times}(\ccal, \Set)$ is the $\infty$-category of truncated objects in $\Fun^{\times}(\ccal, \Spc)$.
\end{itemize}

Our construction of $\infty$-categorical exact completions uses the tools of higher topos theory as developed in \cite{HTT}. We obtain the exact completion of a small regular $\infty$-category $\ccal$ as the closure of $\ccal$ under geometric realizations of groupoids inside the $\infty$-category of sheaves of spaces for the regular topology on $\ccal$. Similarly, the hypercompletion is obtained as the closure under geometric realizations of Kan semisimplicial objects of the image of $\ccal$ inside the $\infty$-category of hypercomplete sheaves.

 To prove that these constructions have the desired properties, we develop the theory of regular $\infty$-topoi, which is a variant of Lurie's theory of coherent $\infty$-topoi \cite{DAGVII, SAG} (which is itself a higher categorical version of the classical theory of coherent topoi). We may summarize the parallels between the regular and the coherent setting in the following table:

 \medskip
  \begin{center}
 \begin{tabular}{c | c | c | c }
Logic & Syntactic category & Grothendieck topology & Classifying topoi \\
\hline
Regular & Exact category & Singleton & Regular topoi \\
\hline
Coherent & Pretopos & Finitary & Coherent topoi
\end{tabular}
 \end{center}
 \medskip
 
In the case of regularity, the connections between the different columns are described by the following diagram:
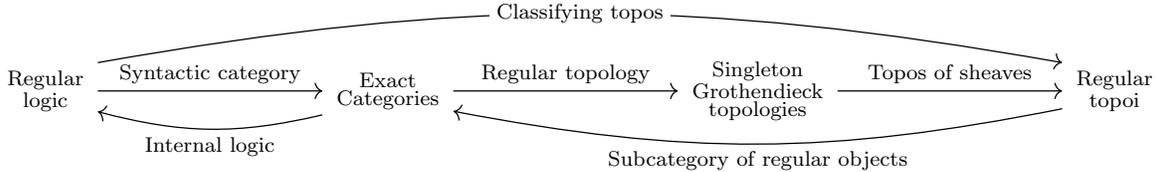
\[
\begin{tikzcd}[column sep = 1.5cm]
	{\substack{\text{Regular}\\ \text{logic}}} && {\substack{\text{Exact} \\ \text{Categories}}} && {\substack{\text{Singleton} \\ \text{Grothendieck} \\ \text{topologies}}} && {\substack{\text{Regular} \\ \text{topoi}}}
	\arrow["{\text{Syntactic category}}", from=1-1, to=1-3]
	\arrow["{\text{Internal logic}}", shift left = 1, bend left = 14pt, from=1-3, to=1-1]
	\arrow["{\text{Regular topology}}", from=1-3, to=1-5]
	\arrow["{\text{Topos of sheaves}}", from=1-5, to=1-7]
	\arrow["{\text{Subcategory of regular objects}}", shift left = 1, bend left = 11pt, from=1-7, to=1-3]
	\arrow["{\text{Classifying topos}}"{description}, shift left=2.5, bend left = 10pt, from=1-1, to=1-7]
\end{tikzcd}
\]

 In the $\infty$-categorical setting, the connection between $\infty$-pretopoi, finitary Grothendieck topologies, and coherent $\infty$-topoi was studied in \cite{SAG}, where it is proven that small hypercomplete $\infty$-pretopoi are in correspondence with hypercomplete  coherent and locally coherent $\infty$-topoi. In this paper we establish a regular version of this result:
 
 \begin{theorem}
 There is a one to one correspondence between small hypercomplete exact $\infty$-categories and hypercomplete regular and locally regular $\infty$-topoi.
 \end{theorem}

The notion of exactness for categories is a particular case of the notion of exactness for $(n,1)$-categories, which we also explore in this paper. Exact $(n,1)$-categories assemble into an $(n+1,1)$-category $(n,1)\kr\Cat_\ex$ whose morphisms are regular functors: that is, functors which preserve finite limits and geometric realizations of \v{C}ech nerves. In the limit $n \to \infty$ we have an $\infty$-category of exact $\infty$-categories $(\infty,1)\kr\Cat_\ex$, which has a full subcategory $(\infty,1)\kr\Cat_{\hyp}$ (resp.  $(\infty,1)\kr\Cat^\bounded_{\ex}$) whose objects are those exact $\infty$-categories which are hypercomplete (resp. have all objects truncated). We may arrange all of these into a sequence of $\infty$-categories
   \[
   (\infty,1)\kr\Cat_{\hyp} \rightarrow (\infty, 1)\kr\Cat^\bounded_{\ex} \rightarrow \ldots \rightarrow (3,1)\kr\Cat_{\ex} \rightarrow (2,1)\kr\Cat_{\ex} \rightarrow (1,1)\kr\Cat_{\ex}
   \] 
  where the transitions are given by passage to the full subcategories of ($n$-)truncated objects. As we shall see, the theory of exact completions and hypercompletions supplies fully faithful left adjoints to the functors in the above sequence. Our next theorem is concerned with the specialization of this picture to the additive case:
  
  \begin{theorem}\label{theorem intro abelian n}
  Let $n \geq 1$ and let $\Acal$ be an additive $(n,1)$-category.
  \begin{enumerate}[\normalfont (1)]
  \item The following are equivalent:
  \begin{enumerate}[\normalfont (i)]
  \item $\acal$ admits finite limits and colimits, every $(n-2)$-truncated morphism in $\acal$ is the fiber of its cofiber, and every $(n-2)$-cotruncated morphism in $\acal$ is the cofiber of its fiber.
  \item $\acal$ is an exact $(n,1)$-category.
  \end{enumerate}
  \item Assume that $\Acal$ satisfies the equivalent conditions from (1). Then the exact completion and the hypercompletion of $\acal$ are finitely complete prestable $\infty$-categories.
  \end{enumerate}
  \end{theorem}
  
 We say that an additive $(n,1)$-category is abelian if it satisfies condition (i) from theorem \ref{theorem intro abelian n}. Part (1) of the theorem is then a generalization to arbitrary values of $n$ of theorem \ref{theorem classical}.  In light of theorems   \ref{theorem intro exact completion abelian} and  \ref{theorem intro abelian n}, we may think about the exact completion (resp. hypercompletion) of an abelian $(n,1)$-category $\acal$ as providing a definition of the connective half of the bounded (resp. unbounded) derived $\infty$-category of $\acal$.
 
  The most basic example of an abelian $(n,1)$-category is given by the $(n,1)$-category $(\LMod_A^\cn)_{\leq n-1}$ of $(n-1)$-truncated connective modules over an $(n-1)$-truncated connective ring spectrum $A$. More generally, any full subcategory of $(\LMod_A^\cn)_{\leq n-1}$ which is closed under finite limits and finite colimits will be an abelian $(n,1)$-category. In the case $n = 1$, the Freyd-Mitchell embedding theorem states that every small abelian category arises in this way for some choice of $A$. 
  
  The Freyd-Mitchell embedding theorem was generalized to the nonadditive context by Barr, who proved that any small regular category admits a regular embedding into a presheaf category \cite{Exact}. There is in fact a canonical such embedding: every small regular category $\ccal$ embeds inside the category $\Fun(\Fun^\reg(\ccal, \Set), \Set)$ of copresheaves on the category $\Fun^\reg(\ccal, \Set)$ of regular functors from $\ccal$ into sets \cite{MakkaiFull}. It was shown by Makkai that in the case when $\ccal$ is exact then the image of this embedding consists precisely of those copresheaves on $\Fun^\reg(\ccal,\Set)$ which preserve small products and filtered colimits  \cite{Makkai}. Our next theorem generalizes these results to the $\infty$-categorical setting:

    \begin{theorem}\label{theorem barr embedding intro}
 Let $\ccal$ be a small regular $\infty$-category.
\begin{enumerate}[\normalfont (1)]
\item The $\infty$-category $\Fun^\reg(\ccal, \Spc)$ is accessible and admits small products and filtered colimits.
\item Assume that every object of $\ccal$ is truncated. Then the functor 
\[
\Upsilon_\ccal: \ccal \rightarrow \Fun(\Fun^\reg(\ccal, \Spc), \Spc)
\]
 induced by evaluation is fully faithful.
\item Assume that $\ccal$ is exact and that every object of $\ccal$ is truncated. Then the image of $\Upsilon_\ccal$ consists of those functors which are truncated and preserve small products and filtered colimits.
\end{enumerate}  
 \end{theorem}
 
 Specializing  theorem \ref{theorem barr embedding intro} to the additive setting we deduce the following generalization of the Freyd-Mitchell embedding theorem:
 
\begin{corollary}\hfill
\begin{enumerate}[\normalfont (1)]
\item  Let $n \geq 1$, and let $\acal$ be a small abelian $(n,1)$-category. Then there exists an $(n-1)$-truncated connective ring spectrum $A$ and a fully faithful exact functor $\acal \rightarrow (\LMod_A^\cn)_{\leq n-1}$.
\item Let $\ccal$ be a small finitely complete prestable $\infty$-category. Assume that every object of $\ccal$ is the limit of its Postnikov tower. Then there exists a connective ring spectrum $A$ and a fully faithful exact functor $\ccal \rightarrow \LMod_A^\cn$.
\end{enumerate}
\end{corollary}

\subsection{Conventions and notation}

In the main body of the paper we use the convention where the word category stands for $\infty$-category, and use the term $(1,1)$-category if we wish to refer to the classical notion.  Similarly, we use the word topos to refer to $\infty$-topoi in the sense of \cite{HTT}.

Each category $\ccal$ has a Hom bifunctor $\Hom_{\ccal}(-, -)$. If $\Ccal$ is a category and $k \geq -2$ is an integer we will denote by $\Ccal_{\leq k}$ the full subcategory of $\Ccal$ on the $k$-truncated objects. If the inclusion $\ccal_{\leq k} \rightarrow \ccal$ admits a left adjoint, this will be denoted by $\tau_{\leq k}$. If  $\ccal = \ccal_{\leq n-1}$ for some integer $n \geq -1$ we say that $\ccal$ is an $(n,1)$-category. Given a pair of categories $\ccal , \dcal$ we denote by $\Fun(\ccal, \dcal)$ the category of functors from $\ccal$ to $\dcal$.

For each topos $\ccal$ we denote by $\ccal^\hyp$ the full subcategory of $\ccal$ on the hypercomplete objects; in other words, this is the localization of $\ccal$ at the class of $\infty$-connective morphisms. A family of maps $f_\alpha: X_\alpha \rightarrow X$ in a topos is said to be jointly effectively epimorphic if the induced map $f: \coprod X_\alpha \rightarrow X$ is an effective epimorphism. A family of objects $\Fcal$ in $\ccal$ is said to be a generating family if for every $X$ in $\ccal$ there exists a jointly effectively epimorphic family of maps $f_\alpha: X_\alpha \rightarrow X$ such that $X_\alpha$ belongs to $\Fcal$ for all $\alpha$. A full subcategory of $\ccal$ is said to be a generating subcategory if its objects form a generating family. A functor between topoi $F: \ccal \rightarrow \dcal$ is said to be geometric (or a geometric morphism) if it preserves colimits and finite limits.   We will also make use of the notions of jointly effective epimorphic families, generating families and geometric morphisms in the setting of $(n,1)$-topoi, by which we mean $n$-topoi in the sense of \cite{HTT}.

We let $\Delta$ (resp. $\Delta_+$) be the category whose objects are nonempty (resp. possibly empty) totally ordered finite sets and whose morphisms are order preserving functions. For each $n \geq -1$ we let $[n]$ be the totally ordered set consisting of integers $i$ such that $0 \leq i \leq n$. We denote by $\Delta_{\normalfont\text{s}}$ (resp. $\Delta_{{\normalfont \text{s}}, +}$) the wide subcategory of $\Delta$ (resp. $\Delta_{+}$)  whose morphisms are the injective order preserving functions. For each $n \geq 0$ we let $\Delta_{  < n}$ (resp. $\Delta_{\normalfont\text{s}, < n}$) be the full subcategory of $\Delta$ (resp. $\Delta_{\normalfont\text{s}}$) on those nonempty totally ordered sets with at most $n$ elements.

If $\ccal$ is a category, a simplicial (resp. semisimplicial) object in $\ccal$ is a functor $X_\bullet: \Delta^\op \rightarrow \ccal$ (resp. $X_\bullet :\Delta^{\op}_{\normalfont \text{s}} \rightarrow \ccal$). For each $n \geq 0$ we denote the value of $X_\bullet$ on $[n]$ by $X_n$. An augmentation for $X_\bullet$ is an extension to $\Delta_{+}$ (resp. $\Delta_{\normalfont\text{s}, +}$). A colimit for $X_\bullet$ is called a geometric realization, and is denoted by $|X_\bullet|$. A (semi)cosimplicial object in $\ccal$ is a (semi)simplicial object in $\ccal^\op$. A limit for a (semi)cosimplicial object $X_\bullet$ in $\ccal$ is called a totalization, and is denoted by $\Tot X_\bullet$.

Let $\ccal$ be a category with finite limits. Then for each morphism $f: X_0 \rightarrow X$ in $\ccal$ one may define a simplicial object $X_\bullet$ in $\ccal$ called the \v{C}ech nerve of $f$, whose value on a nonempty totally ordered set $S$ is given by the image under the forgetful functor $\ccal_{/X} \rightarrow \ccal$ of the product $f^S$. This is an example of a groupoid object in $\ccal$: in other words, it has the property that for every  totally ordered finite set $S$ and every pair of subsets $S_0, S_1 \subseteq S$ such that $S_0 \cap S_1 = [0]$ and $S_0 \cup S_1 = S$ we have $X(S) = X(S_0) \times_{X([0])} X(S_1)$.

\subsection{Size conventions}\label{subsection size management}

We fix a sequence of nested universes. Objects belonging to the first two universes are called small and large, respectively. We denote by $\Cat$ the category of small categories and by $\Spc$ the category of small spaces. Given a small category $\ccal$ we will denote by $\Pcal(\ccal) = \Fun(\ccal^\op, \Spc)$ the category of presheaves on $\ccal$. If $\ccal$ is equipped with a Grothendieck topology we denote by $\Sh(\ccal)$ the resulting topos of sheaves.

Throughout the paper we will frequently be studying a small category $\ccal$ by embedding it inside the large category $\Sh(\ccal)$ of sheaves for some subcanonical Grothendieck topology on $\ccal$. To facilitate these manipulations, we adopt the following convention: when introducing a category $\ccal$ we will implicitly assume $\ccal$ to be small unless it is otherwise clear from the context (for instance by the assumption that $\ccal$ is a topos, or by the fact that $\ccal$ arises as a category of modules over a ring spectrum).  That being said, our results and constructions apply to non-necessarily small categories as well, after a change in universe.

\ifx\inmain\undefined
\bibliographystyle{myamsalpha2}
\bibliography{References}
\fi

\newpage

%%%%%%%%%%%%%%%%%%%%%%%%%%%%%%%%%%%%%%%%%%%%%%%%%%%%%%%%%%%%%%%%%%%%%%%%
%%%%%%%%%%%%%%%%%%%%%%%%%%%%%%%%%%%%%%%%%%%%%%%%%%%%%%%%%%%%%%%%%%%%%%%%
%%%%%%%%%%%%%%%%%%%%%%%%%%%%%%%%%%%%%%%%%%%%%%%%%%%%%%%%%%%%%%%%%%%%%%%%
%%%%%%%%%%%%%%%%%%%%%%%%%%%%%%%%%%%%%%%%%%%%%%%%%%%%%%%%%%%%%%%%%%%%%%%%
%%%%%%%%%%%%%%%%%%%%%%%%%%%%%%%%%%%%%%%%%%%%%%%%%%%%%%%%%%%%%%%%%%%%%%%%
%%%%%%%%%%%%%%%%%%%%%%%%%%%%%%%%%%%%%%%%%%%%%%%%%%%%%%%%%%%%%%%%%%%%%%%%

\section{Regularity and exactness} \label{section exactness}

The goal of this section is to study the notions of regularity and exactness in the setting of higher category theory, and to construct the exact completion, both in the $\infty$-categorical and the $(n,1)$-categorical settings.

 We begin in \ref{subsection regular} with a discussion of  regularity. One of the central features of regular categories is that they admit a good notion of image for morphisms. More precisely, if $f: X \rightarrow Y$ is a morphism in a regular category $\ccal$, the geometric realization of the \v{C}ech nerve of $f$ defines a subobject of $Y$ called the image of $f$, which is also the smallest subobject of $Y$ through which $f$ factors. Throughout the paper we work with functors of regular categories which preserve finite limits and image factorizations; we call these regular functors.
 
  A fundamental tool in the study of regular categories is the regular topology. This is a Grothendieck topology on $\ccal$ defined by the property that a sieve is covering if and only if it contains a morphism $f: X \rightarrow Y$ whose image is $Y$.  We use the regular topology in \ref{subsection exact} to show that every regular category $\ccal$ admits a  univesal regular functor into an exact category $\ccal^\ex$, which we call the exact completion of $\ccal$.
  
  Finally, in \ref{subsection n1} we study the notion of exactness for $(n,1)$-categories. We show that the exact completion provides a fully faithful embedding from the category of exact $(n,1)$-categories and regular functors into the category of exact categories and regular functors, and provide a characterization of its image. We close this section by providing a description of the $(n,1)$-categorical exact completion of a regular $(n,1)$-category.

%%%%%%%%%%%%%%%%%%%%%%%%%%%%%%%%%%%%%%%%%%%%%%%%%%%%%%%%%%%%%%%%%%%%%%%%
%%%%%%%%%%%%%%%%%%%%%%%%%%%%%%%%%%%%%%%%%%%%%%%%%%%%%%%%%%%%%%%%%%%%%%%%
%%%%%%%%%%%%%%%%%%%%%%%%%%%%%%%%%%%%%%%%%%%%%%%%%%%%%%%%%%%%%%%%%%%%%%%%
%%%%%%%%%%%%%%%%%%%%%%%%%%%%%%%%%%%%%%%%%%%%%%%%%%%%%%%%%%%%%%%%%%%%%%%%
%%%%%%%%%%%%%%%%%%%%%%%%%%%%%%%%%%%%%%%%%%%%%%%%%%%%%%%%%%%%%%%%%%%%%%%%
%%%%%%%%%%%%%%%%%%%%%%%%%%%%%%%%%%%%%%%%%%%%%%%%%%%%%%%%%%%%%%%%%%%%%%%%

\subsection{Regular categories}\label{subsection regular}

We begin by discussing the notion of regularity.

\begin{definition}\label{def regular}
Let $\Ccal$ be a finitely complete category. We say that $\Ccal$ is regular if for every map $f: X \rightarrow Y$ in $\ccal$, the \v{C}ech nerve $X_\bullet$ of $f$  admits a geometric realization, and for every base change $f': X' \rightarrow Y'$ of $f$, the comparison map $|Y' \times_Y X_\bullet| \rightarrow Y' \times_Y |X_\bullet|$ is an equivalence.
\end{definition}

\begin{remark}
Let $\Ccal$ be a finitely complete $(1,1)$-category and let $f : X \rightarrow Y$ be a morphism in $\Ccal$ with \v{C}ech nerve $X_\bullet$.  Let $\Delta^\op_{< 2}$ be the full subcategory of $\Delta^\op$ on the objects $[0]$ and $[1]$. A colimit of the restriction of $X_\bullet$  to $\Delta^\op_{< 2}$ is the same as a quotient object for the equivalence relation $K = X \times_Y X$ on $X$. The fact that $\Ccal$ is a $(1,1)$-category implies that such a quotient object exists if and only if a geometric realization for $X_\bullet$ exists, and in such a case we have $|X_\bullet| = X/K$. It follows that $\Ccal$ is regular in the sense of definition  \ref{def regular} if and only if it satisfies condition (1) from the introduction. In other words, when specialized to $(1,1)$-categories, definition \ref{def regular} recovers the classical notion of regularity.
\end{remark}

The main feature of regular categories is that they have a well behaved notion of image factorizations for morphisms:

\begin{notation}
Let $\Ccal$ be a regular category and let $f: X \rightarrow Y$ be a morphism in $\Ccal$. We denote by $\operatorname{Im}(f)$ the geometric realization of the \v{C}ech nerve of $f$.
\end{notation}

\begin{proposition}\label{prop image is subobject}
Let $\Ccal$ be a regular category and let $f: X \rightarrow Y$ be a morphism in $\Ccal$.
\begin{enumerate}[\normalfont(1)]
\item The canonical map $i: \operatorname{Im}(f) \rightarrow Y$ is $(-1)$-truncated.
\item A subobject $Y'$ of $Y$ contains $\operatorname{Im}(f)$ if and only if $f$ factors through $Y'$.
\end{enumerate} 
\end{proposition}
\begin{proof}
We first prove item (1).   The fact that $\Ccal$ is regular implies that the base change of $i$ along $f$  agrees with the morphism $\Im(g) \rightarrow X$, where $g: X \times_{Y} X \rightarrow  X$ is the projection. Since $g$ has a section we have that the augmented \v{C}ech nerve for $g$ is split, and hence $\Im(g) = X$. Let $X_\bullet$ be the \v{C}ech nerve of $f$. Using again the fact that $\Ccal$ is regular we have
\[
\Im(f) \times_{Y} \Im(f) = |X_\bullet| \times_Y \Im(f) = | X_\bullet \times_Y \Im(f)| =| X_\bullet | = \Im(f)
\]
which implies that $i$ is a monomorphism, as desired.

We now prove item (2). The only if direction is clear. Assume now that $Y'$ is such that $f$ factors through $Y'$. The fact that $Y'$ is a subobject of $Y$ implies that the morphism of \v{C}ech nerves of $f$ and of the corestriction of $f$ to $Y'$ agree. Passing to geometric realizations shows that $\Im(f)$ is equivalent to the image of $X \rightarrow Y'$, which is contained in $Y'$, as desired.
\end{proof}

\begin{definition}
Let $\Ccal$ be a regular category. A morphism $f: X \rightarrow Y$ in $\ccal$ is said to be an effective epimorphism if $\Im(f) = Y$.
\end{definition}

The following properties of effective epimorphisms are clear:

\begin{proposition}
Let $\Ccal$ be a regular category. Then:
\begin{enumerate}[\normalfont (1)]
\item A morphism $f: X \rightarrow Y$ in $\Ccal$ is an effective epimorphism if and only if there are no nontrivial subobjects of $Y$ through which $f$ factors.
\item For every morphism $f: X \rightarrow Y$ in $\Ccal$, the morphism $X \rightarrow \Im(f)$ is an effective epimorphism.
\item Invertible morphisms in $\Ccal$ are effective epimorphisms.
\item Effective epimorphisms in $\Ccal$ are stable under base change.
\item Let $f: X \rightarrow Y$ and $g: Y \rightarrow Z$ be morphisms in $\Ccal$. If $f$ and $g$ are effective epimorphisms then $g \circ f$ is an effective epimorphism.
\item Let $f: X \rightarrow Y$ and $g: Y \rightarrow Z$ be morphisms in $\Ccal$. If $g \circ f$ is an effective epimorphism then $g$ is an effective epimorphism.
\end{enumerate}
\end{proposition}

\begin{remark}\label{remark effective epis in trucation}
Let $\Ccal$ be a regular category and let $k \geq 0$. It follows from part (1) of proposition \ref{prop image is subobject} that if $f: X \rightarrow Y$ is a morphism in $\Ccal_{\leq k}$ then $\Im(f)$ belongs to $\Ccal_{\leq k}$. It follows that $\Ccal_{\leq k}$ is regular and a morphism in $\Ccal_{\leq k}$ is an effective epimorphism if and only if its image in $\Ccal$ is an effective epimorphism.
\end{remark}

\begin{definition}
A functor $F: \Ccal \rightarrow \Dcal$ between regular categories is said to be regular if it is left exact and sends effective epimorphisms to effective epimorphisms.
\end{definition}

\begin{remark}
Let $F: \Ccal \rightarrow \Dcal$ be a functor between regular categories. Then $F$ is regular if and only if it is left exact and preserves geometric realizations of \v{C}ech nerves.
\end{remark}

\begin{notation}
We denote by $\Cat_{\reg}$ the subcategory of $\Cat$ on the regular categories and regular functors. For each pair of regular categories $\Ccal, \Dcal$ we  denote by $\Fun^\reg(\Ccal, \Dcal)$ the full subcategory of $\Fun(\Ccal, \Dcal)$ on the regular functors.
\end{notation}

As in the case of topoi, the notion of effective epimorphism in a regular category is the first of a series of connectivity conditions:

\begin{definition}
Let $\Ccal$ be a regular category. For each integer $n \geq -1$ we  define the notion of $n$-connectivity for morphisms in $\Ccal$, inductively as follows:
\begin{itemize}
\item Every morphism is $(-1)$-connective.
\item For each $n \geq 0$, a morphism $X \rightarrow Y$ is $n$-connective if and only if it is an effective epimorphism and the diagonal $X \rightarrow X \times_Y X$ is $(n-1)$-connective.
\end{itemize}
We say that a morphism in $\Ccal$ is $\infty$-connective if it is $n$-connective for all $n \geq 0$. We say that an object $X$ in $\Ccal$ is $n$-connective (resp. $\infty$-connective) if the projection from $X$ to the terminal object in $\Ccal$ is $n$-connective (resp. $\infty$-connective). 
\end{definition}

A fundamental tool in the study of regular categories is the regular topology:

\begin{definition}\label{definition regular topology}
Let $\Ccal$ be a regular category. The regular topology on $\Ccal$ is the Grothendieck topology where a sieve is covering if and only if it contains an effective epimorphism. We will denote by $\Sh(\Ccal)$ the topos of sheaves on $\Ccal$ with respect to the regular topology.
\end{definition}

\begin{remark}
Let $\ccal$ be a regular category. Then every representable presheaf on $\ccal$ is a sheaf for the regular topology. We will usually identify $\ccal$ with the full subcategory of $\Sh(\ccal)$ on the representable presheaves.
\end{remark}

\begin{remark}\label{remark univ prop sh}
Let $\Ccal$ be a regular category. Then the inclusion $\Ccal \rightarrow \Sh(\Ccal)$ creates finite limits and effective epimorphisms. In particular, it is regular. Furthermore, an application of \cite{HTT} proposition 6.2.3.20 shows that for every topos $\Tcal$, restriction to $\Ccal$ induces an equivalence between the category of geometric morphisms of topoi  $\Sh(\Ccal) \rightarrow \Tcal$ and the category of regular functors $\Ccal \rightarrow \Tcal$.
\end{remark}

\begin{remark}
Let $\Ccal$ be a regular category. Then a morphism $f: X \rightarrow Y$ in $\Ccal$ is $n$-connective (resp. $n$-truncated) if and only if its image in $\Sh(\Ccal)$ is $n$-connective (resp. $n$-truncated). In this way, one may deduce many properties of the class of $n$-connective (resp. $n$-truncated) morphisms in $\Ccal$ by reduction to the case of topoi. 
\end{remark}

%%%%%%%%%%%%%%%%%%%%%%%%%%%%%%%%%%%%%%%%%%%%%%%%%%%%%%%%%%%%%%%%%%%%%%%%
%%%%%%%%%%%%%%%%%%%%%%%%%%%%%%%%%%%%%%%%%%%%%%%%%%%%%%%%%%%%%%%%%%%%%%%%
%%%%%%%%%%%%%%%%%%%%%%%%%%%%%%%%%%%%%%%%%%%%%%%%%%%%%%%%%%%%%%%%%%%%%%%%
%%%%%%%%%%%%%%%%%%%%%%%%%%%%%%%%%%%%%%%%%%%%%%%%%%%%%%%%%%%%%%%%%%%%%%%%
%%%%%%%%%%%%%%%%%%%%%%%%%%%%%%%%%%%%%%%%%%%%%%%%%%%%%%%%%%%%%%%%%%%%%%%%
%%%%%%%%%%%%%%%%%%%%%%%%%%%%%%%%%%%%%%%%%%%%%%%%%%%%%%%%%%%%%%%%%%%%%%%%

\subsection{Exact categories and exact completion}\label{subsection exact}

We now discuss the notion of exact category.

\begin{definition}\label{definition exactness}
Let $\Ccal$ be a regular category. A groupoid object in $\Ccal$ is said to be effective if it is the \v{C}ech nerve of some morphism. We say that $\Ccal$ is exact if every groupoid object in $\Ccal$ is effective.
\end{definition}

\begin{notation}
We denote by $\Cat_\ex$ the full subcategory of $\Cat_\reg$ on the exact categories.
\end{notation}
 
\begin{remark}
Let $\Ccal$ be a regular category. A groupoid object $X_\bullet$ in $\Ccal$ is effective if and only if it admits a geometric realization, and $X_\bullet$ agrees with the \v{C}ech nerve of  the map $X_0 \rightarrow |X_\bullet|$.
\end{remark}

\begin{example}
Every topos is exact. Indeed, exactness is a subset of the conditions imposed by Giraud's axioms.
\end{example}

\begin{example}\label{example created regular and exact}
Let $f_\alpha: \Ccal \rightarrow \Dcal_\alpha$ be a family of functors, with $\Dcal_\alpha$ regular (resp. exact) for all $\alpha$. Assume that the family $f_\alpha$ creates pullbacks and geometric realizations of \v{C}ech nerves (resp. groupoid objects)\footnote{Recall that a family of functors $f_\alpha: \ccal \rightarrow \dcal_\alpha$ create the (co)limit of a diagram $F: \Ical \rightarrow \ccal$ if $F$ admits a (co)limit and furthermore a (co)cone for $F$ is a (co)limit if and only if its image under $f_\alpha$ is a (co)limit for all $\alpha$}. Then $\Ccal$ is regular (resp. exact). Specializing this fact we deduce the following:
\begin{itemize}
\item Let $\Ical$ be a category and $\ccal$ be a regular (resp. exact) category. Then $\Fun(\Ical, \ccal)$ is regular (resp. exact).
\item Let $\Tcal$ is a category with finite products and $\ccal$ be a regular (resp. exact) category. Then the category of product preserving functors $\Tcal \rightarrow \ccal$ is also regular (resp. exact). In other words, passing to categories of models over multisorted finite product theories preserves regularity and exactness.
\item Let $\ccal$ be a regular (resp. exact) category and let $X$ be an object of $\ccal$. Then $\Ccal_{/X}$ and $\ccal_{X/}$ are regular (resp. exact).
\end{itemize}
\end{example}

\begin{warning}\label{warning exact}
Let $\Ccal$ be a $(1,1)$-category satisfying conditions (1) and (2) from the introduction (in other words, $\Ccal$ is exact in the classical sense). Then $\Ccal$ does not generally satisfy definition  \ref{definition exactness}. For instance, the category of sets is a $(1,1)$-topos and in particular is exact as a $1$-category, however only those groupoid objects in sets which come from equivalence relations are effective. We will by default use the word exact to  refer to the notion from definition \ref{definition exactness}. A $(1,1)$-category which is exact in the classical sense will be said to be $1$-exact, or an exact $(1,1)$-category (see definition \ref{def n exact} below).
\end{warning}

Our next goal is to construct for each regular category $\ccal$ the universal embedding of $\ccal$ into an exact category.

\begin{notation}\label{notation exact completion}
Let $\Ccal$ be a regular category. We let $\Ccal^{\ex}$ be the smallest subcategory of $\Sh(\Ccal)$ containing the representable sheaves and closed under geometric realizations of groupoid objects.
\end{notation}

\begin{remark}\label{remark regular generates exact completion}
Let $\Ccal$ be a regular category and let $X$ be an object of $\Ccal^\ex$. Then there exists an object $X'$ in $\Ccal$ and an effective epimorphism $X' \rightarrow X$.
\end{remark}

\begin{remark}\label{remark sizes exact completion}
According to our conventions regarding sizes from \ref{subsection size management}, in notation \ref{notation exact completion} we are implicitly assuming $\ccal$ to be small. By virtue of being a full subcategory of a large category, $\ccal^\ex$ is a priori a large category. We may write $\ccal^\ex$ as the union of an increasing sequence of small full subcategories $(\ccal^\ex)_\alpha$ of $\Sh(\ccal)$ indexed by ordinals, defined in the following way:
\begin{itemize}
\item $(\ccal^\ex)_0 = \ccal$.
\item For each $\alpha$ the full subcategory $(\ccal^\ex)_{\alpha + 1}$ consists of those objects of $\Sh(\ccal)$ which are geometric realizations of groupoids that factor through $(\ccal^\ex)_\alpha$.
\item For each limit ordinal $\alpha$ the full subcategory $(\ccal^\ex)_{\alpha}$ is the union of $(\ccal^\ex)_\beta$ over all $\beta < \alpha$.
\end{itemize}
The above sequence stabilizes after the first uncountable ordinal. It follows from this that $\ccal^\ex$ is in fact a small category.

Note that our definition of $\ccal^\ex$ depends a priori on the boundary between small and large (since it makes use of the topos $\Sh(\ccal)$ of small sheaves on $\ccal$). It turns out however that $\ccal^\ex$ is independent of this, since the category of small sheaves is closed under small colimits and limits inside the category of large sheaves.
\end{remark}

The category $\Ccal^{\ex}$ is called the exact completion of $\Ccal$. The name is justified by the following:

\begin{theorem}\label{teo properties completion}
Let $\Ccal$ be a regular category.  Then:
\begin{enumerate}[\normalfont (1)]
\item The category $\Ccal^{\ex}$ is exact and the inclusion $\ccal \rightarrow \ccal^\ex$ is regular.
\item Let $\Dcal$ be an exact category. Then composition with the inclusion $\Ccal \rightarrow \Ccal^{\ex}$  induces an equivalence $\Fun^{\reg}(\Ccal^\ex ,\Dcal) = \Fun^{\reg}(\Ccal, \Dcal)$.
\end{enumerate}
\end{theorem}

As a consequence of theorem \ref{teo properties completion}, we have the following:

\begin{corollary}
 The inclusion $\Cat_\ex \hookrightarrow \Cat_\reg$ admits a left adjoint that sends each regular category $\ccal$ to $\ccal^\ex$. In particular, if $\ccal$ is an exact category then $\ccal = \ccal^\ex$.
 \end{corollary}
 
The proof of theorem \ref{teo properties completion} requires some preliminary lemmas.
  
\begin{lemma}\label{lemma density} 
Let $\iota: \Ccal \rightarrow \Dcal$ be a fully faithful left exact functor between finitely complete categories. Assume that $\ccal$ and $\dcal$ are equipped with Grothendieck topologies, and that the following hypothesis is satisfied:
\begin{enumerate}[\normalfont ($\ast$)]
\item Let  $\lbrace f_\alpha: X_\alpha \rightarrow X \rbrace$ be a covering sieve on an object $X$ in $\ccal$. Then the sieve  on $\iota(X)$ generated by $\lbrace \iota (f_\alpha): \iota (X_\alpha) \rightarrow \iota(X) \rbrace$ is a covering sieve. Furthermore, every covering sieve on $\iota(X)$ contains a sieve that arises in this way.
\end{enumerate}
Then the induced functors $\iota^*: \Sh(\ccal) \rightarrow \Sh(\dcal)$ and $\iota^{*, \hyp}: \Sh(\ccal)^\hyp \rightarrow \Sh(\dcal)^\hyp$ are fully faithful.
  \end{lemma}
  \begin{proof}
  We begin by showing that the functor $\iota^*: \Sh(\Ccal) \rightarrow \Sh(\Dcal)$ induced by $\iota$ is fully faithful.  Let $\iota_*$ be the right adjoint to $\iota^*$ and let $\eta: \id \rightarrow \iota_* \iota^*$ be the unit of the adjunction. We will prove that $\iota^*$ is fully faithful by showing that $\eta$ is an isomorphism.
  
   Let $\iota^{*,\pre}: \Pcal(\ccal) \rightarrow \Pcal(\dcal)$ be the functor induced by $\iota$. Let $\iota_*^\pre$ be the right adjoint to $\iota^{*,\pre}$ and let $\eta^\pre: \id \rightarrow \iota^\pre_* \iota^{*, \pre}$ be the unit of the adjunction.  Then for every $X$ in $\Sh(\Ccal)$ the map $\eta_X$ may be written as a composition
\[
X \xrightarrow{\eta^{\text{pre}}_X} \iota^\pre_* \iota^{*,\pre} (X) \xrightarrow{\iota^\pre_* \nu} \iota^\pre_* \iota^*(X) = \iota_* \iota^*(X)
\]
where $\nu: \iota^{*,\pre}(X) \rightarrow \iota^*(X)$ is the sheafification of $\iota^{*, \pre}(X)$. Since $\iota$ is fully faithful the functor $\iota^{*,\pre}$ is also fully faithful, and hence $\eta^\pre_X$ is an isomorphism. We may thus reduce to showing that $\iota_*^\pre \nu$ is an isomorphism. Since its source and target are sheaves, it is enough for this to show that $\iota_*^\pre$ maps local isomorphisms to local isomorphisms. Using the fact that $\iota_*^\pre$ preserves colimits  we reduce to proving that for every object $Y$ in $\dcal$ and every covering sieve $j: U \hookrightarrow Y$ on $Y$, the map $\iota_*^\pre(j)$ is a local isomorphism. To do so it suffices to show that for every $X$ in $\ccal$ and every morphism $f: X \rightarrow \iota_*^\pre(Y)$ the base change of $\iota_*^\pre(j)$ along $f$ is a local isomorphism. This agrees with the image under $\iota_*^\pre$ of the base change of $j$ along the induced map $\iota(X) \rightarrow Y$. Replacing $j$ with this base change we may now reduce to the case when $Y = \iota(X)$. In this case $j$ defines a covering sieve on $\iota(X)$, and the desired claim follows from an application of ($\ast$).

We now prove that the functor $\iota^{*,\hyp}: \Sh(\Ccal)^\hyp \rightarrow \Sh(\Dcal)^\hyp$ induced by $\iota$ is fully faithful. Let $\iota_*^\hyp$ be the right adjoint to $\iota^{*,\hyp}$ and let $\eta^\hyp: \id \rightarrow \iota_*^\hyp \iota^{*,\hyp}$ be the unit of the adjunction. Then for every $X$ in $\Sh(\Ccal)^\hyp$ the map $\eta^\hyp_X$ may be written as a composition
\[
X \xrightarrow{\eta_X} \iota_* \iota^* (X) \xrightarrow{\iota_* \nu} \iota_* \iota^{*, \hyp}(X) = \iota_*^\hyp \iota^{*, \hyp}(X)
\]
where $\nu: \iota^*(X) \rightarrow \iota^{*, \hyp}(X)$ is the hypercompletion of $\iota^*(X)$. Recall from the first part of the proof that $\eta_X$ is an isomorphism. Furthermore, since $\iota_*^\pre$ maps local isomorphisms to local isomorphisms we have that  $\iota_*$ is colimit preserving. The functor $\iota_*$ is a right adjoint so it is also limit preserving. Hence $\iota_* \nu$ is $\infty$-connective, and since its source and target are hypercomplete, we conclude that $\iota_* \nu$ is an isomorphism. It follows that $\eta^\hyp_X$ is an isomorphism. Since $X$ was arbitrary, we have that $\eta^\hyp$ is an isomorphism, and the result follows. 
\end{proof}

\begin{lemma}\label{lemma fully faithfulness ex and hyp}
Let $\iota: \Ccal \rightarrow \Dcal$ be a fully faithful regular functor between regular categories. Assume that for every effective epimorphism $f: X \rightarrow \iota(Y)$ in $\Dcal$ there exists a morphism $g: \iota(X') \rightarrow X$ such that  $f \circ g$ is an effective epimorphism. Then the induced functors $\iota^*: \Sh(\ccal) \rightarrow \Sh(\dcal)$ and $\iota^{*, \hyp}: \Sh(\ccal)^\hyp \rightarrow \Sh(\dcal)^\hyp$ are fully faithful.
 \end{lemma}
\begin{proof}
This is a specialization of lemma \ref{lemma density}.
\end{proof} 

\begin{proof}[Proof of theorem \ref{teo properties completion}]
We first prove item (1). Since this holds with $\Sh(\ccal)$ replacing $\ccal^\ex$ and $\ccal^\ex$ is closed under geometric realizations of groupoid objects in $\Sh(\ccal)$, it suffices to show that $\Ccal^{\ex}$ is closed under finite limits inside $\Sh(\Ccal)$. The fact that $\Ccal^{\ex}$ contains the final object of $\Sh(\Ccal)$ follows from the fact that the Yoneda embedding $\Ccal \rightarrow \Sh(\Ccal)$ preserves final objects. It remains to show that $\Ccal^{\ex}$ is closed under pullbacks.

Say that an object $Z$ in $\Ccal^\ex$ is good if for every pair of maps $X \rightarrow Z$ and $Y \rightarrow Z$ with $X, Y$ in $\Ccal$, the fiber product $X \times_Z Y$ belongs to $\Ccal^\ex$. We first claim that if $Z$ is good then in fact $X \times_Z Y$ belongs to $\Ccal^\ex$ for every pair of maps $X \rightarrow Z$ and $Y \rightarrow Z$ with $X, Y$ in $\Ccal^\ex$. Indeed, since $\Ccal^{\ex}$ is closed under geometric realizations of groupoids, the condition that $X \times_Z Y$ belongs to $\Ccal^\ex$ is also closed under realizations of groupoids in the $X$ and $Y$ variables, so if it holds for $X, Y$ in $\Ccal$ it also holds for $X, Y$ in $\Ccal^\ex$.

To show that $\Ccal^\ex$ is closed under pullbacks it will suffice to show  that every object in $\Ccal^\ex$ is good. Say that an object $Z$ in $\Ccal^\ex$ is excellent if $Z^n$ is good for all $n > 0$. We will show that every object in $\Ccal^\ex$ is excellent. Since $\Ccal$ is closed under finite limits inside $\Sh(\Ccal)$, we see that every object in $\Ccal$ is excellent. We may therefore reduce our claim to showing that if $Z_\bullet$ is a groupoid in $\Ccal^\ex$ which is termwise excellent, then $|Z_\bullet|$ is excellent. Since for all $n > 0$ we have that $|Z_\bullet|^n$ is the geometric realization of the groupoid $Z_\bullet^n$ (which is termwise excellent) we may reduce to showing that $|Z_\bullet|$ is good.

Let $f_X: X \rightarrow |Z_\bullet|$ and $f_Y: Y \rightarrow |Z_\bullet|$ be a pair of maps with $X, Y$ in $\Ccal$. We will show that $X \times_{|Z_\bullet|} Y$ belongs to $\Ccal^\ex$. Since the map $Z_0 \rightarrow |Z_\bullet|$ is an effective epimorphism, we may pick an effective epimorphism $X_0 \rightarrow X$ with $X_0$ in $\Ccal$, such that the restriction of $f_X$ to $X_0$ factors through $Z_0$. Let $X_\bullet$ be the \v{C}ech nerve of the map $X_0 \rightarrow X$. Then $X \times_{|Z_\bullet|} Y$ is the geometric realization of the groupoid $X_{\bullet} \times_{ |Z_\bullet|} Y$. We may therefore replace $X$ by $X_n$ and in this way reduce to the case when $f_X$ factors through $Z_0$. Similarly, we may assume that $f_Y$ factors through $Z_0$. We now have an equivalence 
\[
X \times_{|Z_\bullet|} Y = (X \times Y) \times_{Z_0 \times Z_0} Z_1.
 \]
  Our claim follows from the fact that  $X \times Y$ and $Z_1$ belong to $\Ccal^\ex$, since $Z_0$ is excellent.

We now prove item (2). We first claim that $\Dcal = \Dcal^\ex$. Let $X_{\bullet}$ be a groupoid object of $\Dcal$. Since $\Dcal$ is exact, we have that $X_{\bullet}$ is the \v{C}ech nerve of an effective epimorphism $X_0 \rightarrow Y$ in $\Dcal$. It follows that $X_{\bullet}$ is the \v{C}ech nerve of $X_0 \rightarrow Y$ in $\Sh(\Dcal)$. Since $X_0 \rightarrow Y$ is an effective epimorphism in $\Sh(\Dcal)$, we have that the geometric realization for $X_{\bullet}$ inside $\Sh(\Dcal)$ recovers $Y$. Hence $\Dcal$ is closed under geometric realizations in $\Sh(\Dcal)$, as claimed.

Since the inclusions $\Ccal \rightarrow \Ccal^\ex$ and $\Dcal \rightarrow \Sh(\Dcal)$ are regular, we have a commutative square of categories
\[
\begin{tikzcd}
\Fun^\reg(\Ccal^\ex, \Dcal) \arrow{r}{} \arrow{d}{} & \Fun^\reg(\Ccal, \Dcal) \arrow{d}{} \\
\Fun^\reg(\Ccal^\ex, \Sh(\Dcal)) \arrow{r}{} & \Fun^\reg(\Ccal, \Sh(\Dcal)).
\end{tikzcd}
\]
The vertical arrows are fully faithful, and their images consist of those regular functors into $\Sh(\Dcal)$ which factor through $\Dcal$. The fact that $\Dcal$ is closed under geometric realizations of groupoid objects inside $\Sh(\Dcal)$ implies that a regular functor $\Ccal^\ex \rightarrow \Sh(\Dcal)$ factors through $\Dcal$ if and only if its restriction to $\Ccal$ does. We may therefore reduce to showing that the bottom horizontal arrow is an equivalence. Applying remark \ref{remark univ prop sh} we reduce to showing that the geometric morphism of topoi $\iota^*: \Sh(\Ccal) \rightarrow \Sh(\Ccal^\ex)$ induced by the regular embedding $\Ccal \rightarrow \Ccal^\ex$ is an equivalence. Combining remark \ref{remark regular generates exact completion} with lemma \ref{lemma fully faithfulness ex and hyp} we deduce that $\iota^*$ is fully faithful. Since $\ccal^\ex$ is generated under geometric realizations of groupoid objects by $\ccal$, we see that $\Sh(\ccal^\ex)$ is generated under colimits by $\ccal$. Using the fact that $\iota^*$ is colimit preserving we deduce that $\iota^*$ is surjective, and hence it is an equivalence, as desired. 
\end{proof}

\begin{remark}
Let $\ccal$ be a category with finite limits, and let $\ccal'$ be the full subcategory of $\Pcal(\ccal)$ on the objects of the form $\Im(f)$ for some map $f: X \rightarrow Y$ in $\ccal$. Then $\ccal'$ contains $\ccal$, and is closed under finite limits and geometric realizations of \v{C}ech nerves inside $\Pcal(\ccal)$.  It follows that $\ccal'$ is a regular category, and the inclusion $\ccal \rightarrow \ccal'$ is a regular functor. A slight variant of the proof of part (2) of theorem \ref{teo properties completion} shows that $\ccal'$ is the universal regular category equipped with a left exact functor from $\ccal$. We call $\ccal'$ the regular completion of $\ccal$. Note that we have an identification $\Pcal(\ccal) = \Sh(\ccal')$, and hence the exact completion $(\ccal')^\ex$ is given by the closure of $\ccal$ under geometric realizations of groupoid objects in $\Pcal(\ccal)$. The inclusion $\ccal \rightarrow (\ccal')^\ex$ is the universal left exact functor from $\ccal$ into an exact category.
\end{remark}

The embedding $\Ccal \rightarrow \Sh(\Ccal)$ can be used to import several basic results from higher topos theory into the world of exact categories:

\begin{proposition}\label{proposition 0 truncated are exact}
Let $\Ccal$ be an exact category and let $k \geq -2$.
\begin{enumerate}[\normalfont (1)]
\item The classes of $(k+1)$-connective and $k$-truncated morphisms in $\Ccal$ form an orthogonal factorization system which is stable under base change.
\item The inclusion $\Ccal_{\leq k} \rightarrow \Ccal$ admits a left adjoint $\tau_{\leq k}$ which preserves finite products.
\item A morphism $f: X \rightarrow Y$ in $\Ccal$ is an effective epimorphism if and only if $\tau_{\leq 0}(f)$ is an effective epimorphism.
\item Let $f: X \rightarrow Y$ and $g: Y \rightarrow Z$ be a pair of morphisms in $\ccal$. If $g \circ f$ is $(k+1)$-connective and $g$ is $(k+2)$-connective then $f$ is $(k+1)$-connective.
\end{enumerate}  
\end{proposition}

The proof of proposition \ref{proposition 0 truncated are exact} rests on the following lemma:

\begin{lemma}\label{lemma stable under truncations}
Let $\Ccal$ be an exact category and let $k \geq -2$. Let $f: X \rightarrow Y$ be a morphism in $\Ccal$ and let $f': X' \rightarrow Y$ be the $k$-truncation of $f$ in $\Sh(\Ccal)_{/Y}$. Then $X'$ belongs to $\Ccal$.
\end{lemma}
\begin{proof}
We argue by induction on $k$. The case $k = -2$ is clear. Assume now that $k \geq -1$ and the proposition holds for $k-1$. We will show that $X'$ is the geometric realization of a groupoid taking values in $\Ccal$. Since the map $p: X \rightarrow X'$ is an effective epimorphism it will suffice to show that the \v{C}ech nerve of $p$ takes values in $\Ccal$. Since $\Ccal$ is closed under finite limits in $\Sh(\Ccal)$ we may reduce to proving that $X \times_{X'} X$ belongs to $\Ccal$. The map $X \rightarrow X \times_{X'} X$ exhibits $ X \times_{X'} X$ as the $(k-1)$-truncation of $X$ in $\Sh(\Ccal)_{/X \times_Y X}$.  The desired claim now follows from the inductive hypothesis.
\end{proof}

\begin{proof}[Proof of proposition \ref{proposition 0 truncated are exact}]
All of these assertions are well known in the setting of topoi. The proposition now follows from lemma \ref{lemma stable under truncations} together with the fact that the inclusion $\Ccal = \Ccal^\ex \rightarrow \Sh(\Ccal)$ is regular and reflects effective epimorphisms.
\end{proof}

%%%%%%%%%%%%%%%%%%%%%%%%%%%%%%%%%%%%%%%%%%%%%%%%%%%%%%%%%%%%%%%%%%%%%%%%
%%%%%%%%%%%%%%%%%%%%%%%%%%%%%%%%%%%%%%%%%%%%%%%%%%%%%%%%%%%%%%%%%%%%%%%%
%%%%%%%%%%%%%%%%%%%%%%%%%%%%%%%%%%%%%%%%%%%%%%%%%%%%%%%%%%%%%%%%%%%%%%%%
%%%%%%%%%%%%%%%%%%%%%%%%%%%%%%%%%%%%%%%%%%%%%%%%%%%%%%%%%%%%%%%%%%%%%%%%
%%%%%%%%%%%%%%%%%%%%%%%%%%%%%%%%%%%%%%%%%%%%%%%%%%%%%%%%%%%%%%%%%%%%%%%%
%%%%%%%%%%%%%%%%%%%%%%%%%%%%%%%%%%%%%%%%%%%%%%%%%%%%%%%%%%%%%%%%%%%%%%%%

\subsection{Exact \texorpdfstring{$(n,1)$}{(n,1)}-categories}\label{subsection n1}

We now study the theory of exact $(n,1)$-categories. We will need the  following notion from \cite{HTT} section 6.4:

\begin{definition}
Let $\Ccal$ be a finitely complete category and let $n \geq 0$ be an integer. We say that a groupoid object $X_\bullet$ in $\Ccal$ is $n$-efficient if the map $X_1 \rightarrow X_0 \times X_0$ is $(n-2)$-truncated. 
\end{definition} 

The definition of exactness for $(n,1)$-categories is a variant of the $\infty$-categorical definition where only $n$-efficient groupoid objects are required to be effective:

\begin{definition}\label{def n exact}
Let $n \geq 0$ be an integer and let $\Ccal$ be a regular $(n,1)$-category. We say that $\Ccal$ is $n$-exact (or an exact $(n,1)$-category) if every $n$-efficient groupoid object in $\Ccal$ is effective.
\end{definition}

\begin{example}
Let $n \geq 0$ and let $\Ccal$ be an $(n,1)$-topos in the sense of \cite{HTT}. Then $\Ccal$ is $n$-exact. Indeed, $n$-exactness is part of Giraud's axioms for $(n,1)$-topoi. 
\end{example}

\begin{example}\label{example truncate exact is n exact}
Let $\Ccal$ be an exact category. Then for each $n \geq 0$ the category $\Ccal_{\leq n-1}$ is $n$-exact. Indeed, $\Ccal_{\leq n-1} = \Ccal^\ex \cap \Sh(\Ccal)_{\leq n-1}$ is closed under finite limits and geometric realizations of $n$-efficient groupoid objects inside the $(n,1)$-topos $\Sh(\Ccal)_{\leq n-1}$.
\end{example}

\begin{example}
Let $n \geq 0$ and let $f_\alpha: \Ccal \rightarrow \Dcal_\alpha$ be a family of functors, with $\Dcal_\alpha$ an exact $(n,1)$-category for all $\alpha$. Assume that the family $f_\alpha$ creates pullbacks and geometric realizations of $n$-efficient groupoid objects for all $\alpha$. Then $\Ccal$ is an exact $(n,1)$-category. As in example \ref{example created regular and exact} this implies that exact $(n,1)$-categories are stable under passage to categories of diagrams, categories of models under multisorted finite product theories, overcategories, and undercategories. 
\end{example}

\begin{remark}
Let $\Ccal$ be a finitely complete $(1,1)$-category. Then the data of a $1$-efficient groupoid object $X_\bullet$ on $\Ccal$ is the same as the data of an object $X_0$ and an equivalence relation $R = X_1$ on $X_0$. Assume now that $\Ccal$ is regular. Then $X_\bullet$ is effective if and only if the quotient $X_0 / R$ exists and the map $X_1 \rightarrow X_0 \times_{X_0/R} X_0$ is an isomorphism. In other words, $X_\bullet$ is effective if and only if $R$ is an effective equivalence relation on $X_0$. It follows that $\Ccal$ is $1$-exact if and only if it satisfies condition (2) from the introduction. In other words, the case $n = 1$ of definition \ref{def n exact} recovers the classical notion of exactness.
\end{remark}

An exact $(n,1)$-category is usually not exact when regarded as an $(\infty,1)$-category (see warning \ref{warning exact}). We now study the behavior of the exact completion from \ref{subsection exact} when applied to exact $(n,1)$-categories.

\begin{proposition}\label{prop ex completion of ex n cat}
Let $n \geq 0$ and let $\Ccal$ be an exact $(n,1)$-category. Then the inclusion $\Ccal \rightarrow (\Ccal^\ex)_{\leq n-1}$ is an equivalence.
\end{proposition}
\begin{proof}
 We need to show that if $X$ in $\Ccal^\ex$ is $(n-1)$-truncated then $X$ belongs to $\Ccal$. We will in fact show that if $f: X \rightarrow Y $ is an $(n-1)$-truncated map and $Y$ belongs to $\Ccal$ then $X$ belongs to $\Ccal$. Assume that $f$ is $k$-truncated for some $-2 \leq k \leq n-1$. We argue by induction on $k$. The case $k = -2$ is clear.  Assume now that $-1 \leq k \leq n-1$ and that the assertion holds for $k-1$.  Pick an effective epimorphism $g: X' \rightarrow X$ with $X'$ in $\Ccal$. It suffices to show that the \v{C}ech nerve of $g$ takes values in $\Ccal$. Since $\Ccal$ is closed under finite limits we may reduce to showing that $X' \times_X X'$ belongs to $\Ccal$. This follows from the inductive hypothesis, since 
 \[
 X' \times_X X' = (X' \times_{Y} X') \times_{X \times_Y X} X
 \]
  has a $(k-1)$-truncated map to $X' \times_{Y} X'$. 
\end{proof}

\begin{definition}
Let $\Ccal$ be a regular category and let $k \geq -2$ be an integer. We say that $\Ccal$ is $k$-complicial if for every $X$ in $\Ccal$ there exists a $k$-truncated object $X'$ and an effective epimorphism $X' \rightarrow X$.
\end{definition}

\begin{notation}\label{notation bounded}
Let $\ccal$ be a regular category. We denote by $\ccal_{<\infty}$ the full subcategory of $\ccal$ on the truncated objects. If  $\ccal = \ccal_{<\infty}$ we will say that $\ccal$ is bounded. We denote by $\Cat^\bounded_\ex$ the full subcategory of $\Cat_\ex$ on the bounded exact categories. 
\end{notation}

\begin{remark}\label{remark Clessinfty closed under}
Let $\ccal$ be an exact category. Then $\Ccal_{< \infty}$ is closed under finite limits inside $\ccal$. We claim that  $\Ccal_{< \infty}$ is also closed under geometric realizations of groupoid objects inside $\ccal$ (and therefore $\ccal_{<\infty}$ is exact and the inclusion $\ccal_{<\infty} \rightarrow \ccal$ is regular). To see this we need to show that if $X_\bullet$ is a groupoid object in $\Ccal$ which is levelwise truncated, then $|X_\bullet|$ is truncated. Equivalently, we must show that the diagonal $|X_\bullet| \rightarrow |X_\bullet| \times |X_\bullet|$ is truncated. Since the map $X_0 \rightarrow X_\bullet$ is an effective epimorphism, we may reduce to showing that the projection
\[
X_1 = (X_0 \times X_0) \times_{|X_\bullet| \times |X_\bullet|} |X_\bullet| \rightarrow X_0 \times X_0
\]
is truncated, which follows from the fact that both $X_1$ and $X_0$ are truncated.
\end{remark}

\begin{proposition}\label{prop image is complicial}
Let $n \geq 0$.
\begin{enumerate}[\normalfont (1)]
\item Let $\Ccal$ be a regular $(n,1)$-category. Then $\Ccal^\ex$  is $(n-1)$-complicial and bounded.
\item Let $\Dcal$ be an $(n-1)$-complicial exact category. Then the functor $(\Dcal_{\leq n-1})^\ex \rightarrow \Dcal$ restricts to an equivalence $(\Dcal_{\leq n-1})^\ex = \dcal_{<\infty}$.
\end{enumerate}
\end{proposition}
\begin{proof}
We first prove item (1). The fact that $\Ccal^\ex$ is $(n-1)$-complicial is a direct consequence of the fact that every object in $\Ccal^\ex$ receives an effective epimorphism from an object in $\Ccal$. The fact that $\ccal^\ex$ is bounded follows from remark \ref{remark Clessinfty closed under}, since $\ccal$ is bounded and $\ccal^\ex$ is generated under geometric realizations of groupoid objects by $\ccal$.

We now prove (2).  Denote by $\iota$ the inclusion $\dcal_{\leq n-1} \rightarrow \dcal$. An application of lemma \ref{lemma fully faithfulness ex and hyp} shows that the induced functor $\iota^\ex: (\dcal_{\leq n-1})^\ex \rightarrow \dcal^\ex = \dcal$ is fully faithful. It remains to show that the image of $\iota^\ex$ consists of the truncated objects of $\Dcal$. It suffices for this to show that $\iota^\ex$ defines an equivalence when restricted to the full subcategories of $k$-truncated objects, for every $k \geq -2$. We prove this by induction on $k$. The case $k \leq n-1$ follows from proposition \ref{prop ex completion of ex n cat}. Assume now that $k > n-1$ and that the assertion is known for $k-1$.  Let $X$ be a $k$-truncated object of $\Dcal$.  Since $\dcal$ is $(n-1)$-complicial we may find an effective epimorphism $f: X' \rightarrow X$  with $X'$ an $(n-1)$-truncated object of $\Dcal$. Let $X'_\bullet$ be the \v{C}ech nerve of $f$, so that $X = |X'_\bullet|$. The projection $X' \times_X X' \rightarrow X' \times X'$ is a base change of the diagonal of $X$, and hence it is $(k-1)$-truncated. Since $X'$ is $(n-1)$-truncated we see that $X' \times_X X'$ is $(k-1)$-truncated. It follows that $X'_\bullet$ is levelwise $(k-1)$-truncated. Our inductive hypothesis now implies that $X'_\bullet$ is the image under $\iota^\ex$ of a groupoid object $S_\bullet$ in $(\Dcal_{\leq n-1})^\ex$. Since $\iota^\ex$ is regular we have that $X = \iota^\ex(|S_\bullet|)$, which proves the inductive step.
\end{proof}

\begin{notation}
For each integer $n \geq 0$ we let $(n,1)\kr\Cat_\ex$ be the full subcategory of $\Cat_\reg$ on the exact $(n,1)$-categories.
\end{notation}

\begin{corollary}\label{corollary fully faithful exact on n1}
Let $n \geq 0$. Then the exact completion functor $(n,1)\kr\Cat_\ex \rightarrow \Cat_{\ex}$ is fully faithful, and its image consists of the $(n-1)$-complicial bounded exact categories. 
\end{corollary}
\begin{proof}
Let $\Ccal$ be an exact $(n,1)$-category and let $\Dcal$ be an exact category. Every left exact functor $\Ccal \rightarrow \Dcal$ factors through $\Dcal_{\leq n-1}$. It follows that composition with the inclusion $\Dcal_{\leq n-1} \rightarrow \Dcal$ induces an equivalence $\Fun^\reg(\Ccal, \Dcal_{\leq n-1}) = \Fun^\reg(\Ccal, \Dcal)$. Hence the exact completion functor $(n,1)\kr\Cat_\ex \rightarrow \Cat_{\ex}$  admits a right adjoint which sends each exact category $\Dcal$ to $\Dcal_{\leq n-1}$. The fully faithfulness part of the corollary now follows from proposition \ref{prop ex completion of ex n cat}, while the description of the image follows from proposition \ref{prop image is complicial}.
\end{proof}

We finish by recording the following result which extends the theory of exact completions to the setting of $(n,1)$-categories:

\begin{proposition}\label{prop n ex completion}
Let $n \geq 0$. Let $\Ccal$ be a regular $(n,1)$-category and let $\Dcal$ be an exact $(n,1)$-category. Then precomposition with the inclusion $\Ccal \rightarrow (\Ccal^\ex)_{\leq n-1}$  induces an equivalence $\Fun^{\reg}((\Ccal^\ex)_{\leq n-1}, \Dcal) \rightarrow \Fun^\reg(\Ccal, \Dcal)$.
\end{proposition}
\begin{proof}
We have a commutative square of categories
\[
\begin{tikzcd}
\Fun^\reg((\Ccal^\ex)_{\leq n-1}, \Dcal) \arrow{r}{} \arrow{d}{} & \Fun^\reg(\Ccal, \Dcal) \arrow{d}{} \\
 \Fun^\reg((\Ccal^\ex)_{\leq n-1}, \Dcal^\ex) \arrow{r}{} & \Fun^\reg(\Ccal, \Dcal^\ex).
\end{tikzcd}
\]
It follows from proposition \ref{prop ex completion of ex n cat} applied to $\Dcal$ that the vertical arrows are equivalences. We may thus reduce to showing that the bottom horizontal arrow is an equivalence. This sits inside a commutative square
\[
\begin{tikzcd}
\Fun^\reg((\Ccal^\ex)_{\leq n-1}, \Dcal^\ex) \arrow{r}{} \arrow{d}{} & \Fun^\reg(\Ccal, \Dcal^\ex) \arrow{d}{} \\ 
\Fun^\reg(((\Ccal^\ex)_{\leq n-1})^\ex, \Dcal^\ex) \arrow{r}{} & \Fun^\reg(\Ccal^\ex, \Dcal^\ex)
\end{tikzcd}
\]
where the vertical arrows are an equivalence by theorem \ref{teo properties completion}. We may thus reduce to showing that the map $\eta: \Ccal \rightarrow (\Ccal^{\ex})_{\leq n-1}$ becomes an equivalence after passing to exact completions. The identity on $\Ccal^\ex$ may be factored as
\[
\Ccal^\ex \xrightarrow{\eta^\ex} ((\Ccal^\ex)_{\leq n-1})^\ex \xrightarrow{\epsilon} \Ccal^\ex
\]
where $\epsilon$ is induced by the inclusion $(\Ccal^\ex)_{\leq n-1} \rightarrow \Ccal^\ex$. We may thus reduce to showing that $\epsilon$ is an equivalence. By proposition \ref{prop image is complicial} we have that both $\Ccal^\ex$ and $((\Ccal^\ex)_{\leq n-1})^\ex$ are $(n-1)$-complicial and have the property that all their objects are truncated. We may thus reduce to showing that $\epsilon$ induces an equivalence on $(n-1)$-truncated objects, which follows from proposition \ref{prop ex completion of ex n cat}.
\end{proof}

\ifx\inmain\undefined
\bibliographystyle{myamsalpha2}
\bibliography{References}
\fi

%%%%%%%%%%%%%%%%%%%%%%%%%%%%%%%%%%%%%%%%%%%%%%%%%%%%%%%%%%%%%%%%%%%%%%%%
%%%%%%%%%%%%%%%%%%%%%%%%%%%%%%%%%%%%%%%%%%%%%%%%%%%%%%%%%%%%%%%%%%%%%%%%
%%%%%%%%%%%%%%%%%%%%%%%%%%%%%%%%%%%%%%%%%%%%%%%%%%%%%%%%%%%%%%%%%%%%%%%%
%%%%%%%%%%%%%%%%%%%%%%%%%%%%%%%%%%%%%%%%%%%%%%%%%%%%%%%%%%%%%%%%%%%%%%%%
%%%%%%%%%%%%%%%%%%%%%%%%%%%%%%%%%%%%%%%%%%%%%%%%%%%%%%%%%%%%%%%%%%%%%%%%
%%%%%%%%%%%%%%%%%%%%%%%%%%%%%%%%%%%%%%%%%%%%%%%%%%%%%%%%%%%%%%%%%%%%%%%%

\section{Regular topoi}\label{section regular topoi}

As we discussed in section \ref{section exactness}, a fundamental tool in the study of regular and exact categories is the topos of sheaves for the regular topology. If a topos $\ccal$ arises as the topos of sheaves on a regular category, then the full subcategory $\ccal_0 \subseteq \ccal$ on the  representable sheaves has the following properties:
\begin{enumerate}[\normalfont (a)]
\item $\ccal_0$ is a generating subcategory for $\ccal$: in other words, every object of $\ccal$ admits an effective epimorphism from a coproduct of objects of $\ccal_0$.
\item $\ccal_0$ is closed under finite limits inside $\ccal$.
\item Let $f_\alpha: X_\alpha \rightarrow Y$ be a jointly effectively epimorphic family of maps in $\ccal$, with $Y$ in $\ccal_0$. Then there exists $\alpha$ such that $f_\alpha$ is an effective epimorphism.
\end{enumerate}

The goal of this section is to abstract the above features. We begin in \ref{subsection regular topoi} by introducing the classes of regular and locally regular topoi, as a variant of the classes of coherent and locally coherent topoi from \cite{DAGVII,SAG}. As we shall see, a topos is regular and locally regular if and only if it admits a full subcategory $\ccal_0$ satisfying conditions (a), (b) and (c). There is in this case a maximal choice of $\ccal_0$, whose objects we call regular. Various properties of regular objects and regular topoi are collected in theorem \ref{teo properties regularity}.

Every regular category gives rise to a regular topos upon passage to sheaves for the regular topology. In \ref{subsection singleton} we discuss a generalization of this fact: the topos of sheaves on any finitely complete category equipped with a subcanonical singleton Grothendieck topology is  regular and locally regular. In this setting, descent and hyperdescent conditions admit particularly simple characterizations, see proposition \ref{prop describe hypercompletes}.

Finally, in \ref{subsection n1 topoi} we study an $(n,1)$-categorical variant of the theory of regular topoi. Our main result on this subject is theorem \ref{theorem regular n1 are}, which provides a one to one correspondence between regular $(n,1)$-topoi and (small) exact $(n,1)$-categories.  	As we shall see in section \ref{section completeness}, a similar result holds also in the limit $n \to \infty$ under additional hypercompleteness hypotheses. In particular,  up to hypercompletion every regular and locally regular topos arises as the topos of sheaves with respect to some subcanonical singleton Grothendieck topology.

%%%%%%%%%%%%%%%%%%%%%%%%%%%%%%%%%%%%%%%%%%%%%%%%%%%%%%%%%%%%%%%%%%%%%%%%
%%%%%%%%%%%%%%%%%%%%%%%%%%%%%%%%%%%%%%%%%%%%%%%%%%%%%%%%%%%%%%%%%%%%%%%%
%%%%%%%%%%%%%%%%%%%%%%%%%%%%%%%%%%%%%%%%%%%%%%%%%%%%%%%%%%%%%%%%%%%%%%%%
%%%%%%%%%%%%%%%%%%%%%%%%%%%%%%%%%%%%%%%%%%%%%%%%%%%%%%%%%%%%%%%%%%%%%%%%
%%%%%%%%%%%%%%%%%%%%%%%%%%%%%%%%%%%%%%%%%%%%%%%%%%%%%%%%%%%%%%%%%%%%%%%%
%%%%%%%%%%%%%%%%%%%%%%%%%%%%%%%%%%%%%%%%%%%%%%%%%%%%%%%%%%%%%%%%%%%%%%%%

\subsection{Regular objects in topoi}\label{subsection regular topoi}

We begin with a discussion of the notion of regular topos.

\begin{definition}\label{definition regular topos}
Let $\Ccal$ be a topos. We say that an object $X$ in  $\Ccal$ is $0$-regular if for every family of subobjects $X_\alpha$ of $X$ whose union is $X$, we have $X_\alpha = X$ for some $\alpha$. We say that $X$ is locally $0$-regular if the collection of $0$-regular objects in $\Ccal_{/X}$ is a generating family. We now define the notions of (local) $n$-regularity for all integers $n \geq 1$ by induction, as follows:
\begin{enumerate}[\normalfont (1)]
\item We say that $X$ is $n$-regular if $\Ccal_{/X}$ is locally $(n-1)$-regular, and the collection of $(n-1)$-regular objects in $\ccal_{/X}$ is closed under finite products.
\item We say that $X$ is locally $n$-regular if  the collection of $n$-regular objects in $\Ccal_{/X}$ is a generating family.
\end{enumerate} 
We also extend the definition to the case $n = -1$, by declaring every object to be $(-1)$-regular and locally $(-1)$-regular. We say that $X$ is regular if it is $n$-regular for all $n \geq 0$, and locally regular if the collection of regular objects in $\Ccal_{/X}$ is a generating family. We say that $\Ccal$ is locally regular (resp. regular, locally $n$-regular, $n$-regular) if its terminal object is locally regular (resp. regular, locally $n$-regular, $n$-regular).
\end{definition}

\begin{warning}
The usage of regularity in definition \ref{definition regular topos} should not be confused with that of definition \ref{def regular}: every topos is a regular category in the sense of definition \ref{def regular}, however not every topos is a regular topos in the sense of definition \ref{definition regular topos}.
\end{warning}

\begin{remark}
It follows from the definitions that the property that an object $X$ in a topos $\Ccal$ be locally regular (resp. regular, locally $n$-regular, $n$-regular) only depends on the topos $\Ccal_{/X}$.
\end{remark}

\begin{remark}
Let $\Ccal$ be a topos. Then an object $X$ in $\Ccal$ is $0$-regular if and only if for every jointly effectively epimorphic family $f_\alpha: X_\alpha \rightarrow X$, there exists an index $\alpha$ such that $f_\alpha$ is an effective epimorphism. In particular, if $\Ccal_0$ is a generating subcategory of $\Ccal$ and $X$ is $0$-regular, there exists an effective epimorphism $U \rightarrow X$ with $U$ in $\Ccal_0$.
\end{remark}

\begin{remark}
Let $\Ccal$ be a topos. Then the initial object of $\Ccal$ is not $0$-regular, since it admits an empty covering family.
\end{remark}

The following result records some basic properties of the notion of regularity. See also \cite{SAG} appendix A where similar properties are established for the notion of coherence.\footnote{Parts (1), (4) through (8), and (10) hold with regularity replaced with coherence, with essentially the same proofs, while part (9)  holds if one assumes that $\Ccal_0$ is closed under finite coproducts.}

\begin{theorem}\label{teo properties regularity}
Let $\Ccal$ be a topos.
\begin{enumerate}[\normalfont (1)]
\item Let $f: X \rightarrow Y$ be a morphism in $\Ccal$. If $Y$ is locally regular then $X$ is locally regular. The converse holds provided that $f$ is an effective epimorphism.
\item Every regular object of $\Ccal$ is coherent, in the sense of \cite{SAG} definition A.2.0.12.
\item Every regular object of $\Ccal$ is almost compact.
\item The collection of regular objects of $\Ccal$ is closed under pullbacks.
\item Assume that $\Ccal$ is regular. Then an object $X$ in $\Ccal$ is regular if and only if $\tau_{\leq k}(X)$ is regular for all $k$.
\item Let $f: X \rightarrow Y$ be an effective epimorphism in $\Ccal$, where $X$ is regular. Then $Y$ is regular if and only if $X\times_Y X$ is regular.
\item Assume that $\Ccal$ is regular. Then its hypercompletion $\Ccal^\hyp$ is regular, and an object in $\Ccal$ is regular if and only if its image in $\Ccal^\hyp$ is regular.
\item Let $X_\bullet$ be a semisimplicial hypercover of an object $X$ in $\Ccal$ (see definition \ref{definition hypercover}). If $X_n$ is  regular for all $n$ then $X$ is regular.
\item Let $\Ccal_0$ be a generating subcategory of $\Ccal$. Assume that every object of $\Ccal_0$ is regular. Then an object in $\Ccal$ is regular if and only if it admits a semisimplicial hypercover by objects in $\Ccal_0$.
\item Let $\Ccal_0$ be a generating subcategory of $\Ccal$. Assume that $\Ccal_0$ is closed under fiber products and that every object of $\Ccal_0$ is $0$-regular. Then every object of $\Ccal_0$ is regular.
\end{enumerate}
\end{theorem}

The proof of theorem \ref{teo properties regularity} needs some preliminary lemmas.

\begin{lemma} \label{lemma estimates regularity}
Let $\Ccal$ be a topos. Let $n \geq 0$ and let $f: X \rightarrow Y$ be a morphism in $\Ccal$.
\begin{enumerate}[\normalfont (1)]
\item If $Y$ is locally $n$-regular then $X$ is locally $n$-regular. The converse holds provided that $f$ is an effective epimorphism.
\item  If $f$ is $n$-connective and  $X$ is $n$-regular then $Y$ is $n$-regular.
\item If $f$ is $(n+1)$-connective and  $Y$ is $n$-regular then $X$ is $n$-regular.
\item If $f$ is $(n-2)$-truncated, $Y$ is $n$-regular and $X$ is $(n-1)$-regular, then $X$ is $n$-regular.
\end{enumerate}
\end{lemma}
\begin{proof}
The forward direction of part (1) follows readily from the definitions. To prove the converse, we observe that the image of the forgetful functor $f_!: \Ccal_{/X} \rightarrow \Ccal_{/Y}$ is a generating subcategory, and hence $f_!$ maps generating subcategories to generating subcategories.

We now prove parts (2) and (3) jointly by induction on $n$. The case $n = 0$ of (3) follows from the fact that pullback along $1$-connective morphisms induces equivalences between posets of subobjects.  Consider now the case $n = 0$ of (2). Assume given a jointly effectively epimorphic family $g_\alpha: Y_\alpha \rightarrow Y$. Then the base changed family $g'_\alpha: X \times_Y Y_\alpha \rightarrow X$ is jointly effectively epimorphic. Since $X$ is $0$-regular then there exists an index $\alpha$ such that $g'_\alpha$ is an effective epimorphism. The composition of $g'_\alpha$ with the effective epimorphism $f$ is an effective epimorphism which factors through $f_\alpha$, and hence $f_\alpha$ is an effective epimorphism. 
 
 Assume now that $n \geq 1$ and that (2) and (3) are known  for $n-1$. We first prove (2). Applying part (1) we see that $Y$ is locally $(n-1)$-regular. Furthermore, by inductive hypothesis we have that $Y$ is $(n-1)$-regular. It remains to show that if  $Y_0 \rightarrow Y$ and $Y_1 \rightarrow Y$ are two $(n-1)$-regular objects with maps to $Y$ then $Y_0 \times_Y Y_1$ is $(n-1)$-regular. Let $X_0$ and $X_1$ be the base changes of $Y_0$ and $Y_1$ to $X$. There are $n$-connective maps $X_0 \rightarrow Y_0$ and $X_1 \rightarrow Y_1$. By inductive hypothesis we see that $X_0$ and $X_1$ are $(n-1)$-regular. Since $X$ is $n$-regular we have that $X_0  \times_X X_1$ is $(n-1)$-regular. Now $Y_0 \times_Y Y_1$ receives an $n$-connective map from $X_0 \times_X X_1$, and is therefore $(n-1)$-regular by our inductive hypothesis.
 
 We now prove (3). Applying part (1) we see that $X$ is locally $(n-1)$-regular. Furthermore, by inductive hypothesis we have that $X$ is $(n-1)$-regular. It remains to show that if $ X_0 \rightarrow X$ and $X_1 \rightarrow X$ are two $(n-1)$-regular objects with maps to $X$, then $X_0 \times_X X_1$ is $(n-1)$-regular. Since $Y$ is $n$-regular we have that $X_0 \times_Y X_1$ is $(n-1)$-regular. The claim now follows from the inductive hypothesis using the fact that the map $X_0 \times_X X_1 \rightarrow X_0 \times_Y X_1$ is a base change of the $n$-connective morphism $X \rightarrow X\times_Y X$.

  We now prove part (4). We argue by induction on $n$. The case $n = 0$ is clear. Assume now that $n \geq 1$ and that the assertion is known for $n-1$. Applying part (1) we see that $X$ is locally $(n-1)$-regular. It remains to show that if $X_0 \rightarrow X$ and $X_1 \rightarrow X$ are two $(n-1)$-regular objects with maps to $X$, then $X_0 \times_X X_1$ is $(n-1)$-regular. Since $X$ is $(n-1)$-regular we have that $X_0 \times_X X_1$ is $(n-2)$-regular, and since $Y$ is $n$-regular we have that $X_0 \times_Y X_1$ is $(n-1)$-regular. Our claim now follows by inductive hypothesis using the fact that  $X_0 \times_X X_1 \rightarrow X_0 \times_Y X_1$ is a base change of the $(n-3)$-truncated morphism $X \rightarrow X \times_Y X$.
  \end{proof}

\begin{lemma}\label{lemma regularity depends on truncationss}
Let $n \geq 0$ and let $F^*: \Ccal \rightarrow \Dcal$ be a colimit preserving, left exact functor of topoi. Assume that $\Ccal$ and $\Dcal$ are locally $(n-1)$-regular and that $F^*$ restricts to an equivalence between the full subcategories on the $n$-truncated objects. Then an object $X$ in $\Ccal$ is $n$-regular if and only if $F^*X$ is $n$-regular.
\end{lemma}
\begin{proof}
By lemma \ref{lemma estimates regularity}, we may replace $X$ by $\tau_{\leq n}(X)$ and thus assume that $X$ is $n$-truncated. Replacing $\Ccal$ by $\Ccal_{/X}$ and $\Dcal$ by $\Dcal_{/F^*X}$, we may further reduce to the case when $X$ is terminal.

 We argue by induction on $n$. The case $n = 0$ is clear since the definition of $0$-regularity only makes reference to $(-1)$-truncated objects. Assume now that $n \geq 1$ and that the lemma is known for $n-1$. It follows from lemma \ref{lemma estimates regularity} that an object in $\Ccal$ is $(n-1)$-regular if and only if its image in $\Ccal_{\leq n-1}$  is $(n-1)$-regular. Hence $\Ccal$ is $n$-regular if and only if the collection of $(n-1)$-regular $(n-1)$-truncated objects in $\Ccal$ is closed under finite products. Similarly, $\Dcal$ is $n$-regular if and only if the collection of $(n-1)$-regular $(n-1)$-truncated objects in $\Dcal$ is closed under finite products. The lemma now follows directly by the inductive hypothesis.
\end{proof}

\begin{lemma}\label{lemma descend regularity}
Let $\Ccal$ be a topos and let $n \geq 0$. Assume given a generating subcategory $\Ccal_0$ of $\Ccal$ such that every object in $\Ccal_0$ is $(n-1)$-regular. Let $f: X \rightarrow Y$ be an effective epimorphism in $\Ccal$ where $X$ is $n$-regular. Then $Y$ is $n$-regular if and only if for every $Z$ in $\Ccal_0$ and every morphism $Z \rightarrow Y$ the object $X \times_Y Z$ is $(n-1)$-regular. 
\end{lemma}
\begin{proof}
The only if direction is clear from the definitions. We will prove the if direction by induction on $n$. The case $n = 0$ follows from part (2) of lemma \ref{lemma estimates regularity}. Assume now that $n \geq 1$ and that the assertion is known for $n-1$. By the inductive hypothesis we have that $Y$ is $(n-1)$-regular. Furthermore, applying part (1) of lemma \ref{lemma estimates regularity} we see that $Y$ is locally $(n-1)$-regular. It remains to show that if $Y_0 \rightarrow Y$ and $Y_1 \rightarrow Y$ are two $(n-1)$-regular objects with maps to $Y$, then $Y_0 \times_Y Y_1$ is $(n-1)$-regular. 

We first claim that $Y_0 \times_Y X$ is $(n-1)$-regular. Since $Y_0$ is $0$-regular we may pick an effective epimorphism $Y'_0 \rightarrow Y_0$ with $Y'_0$ in $\Ccal_0$. Then the induced map $Y'_0 \times_Y X \rightarrow Y_0 \times_Y X$ is an effective epimorphism. By our assumption, $Y'_0 \times_Y X$ is $(n-1)$-regular. Furthermore, given a map $Z \rightarrow Y_0 \times_Y X$ with $Z$ in $\Ccal_0$, we have
\[
 Z \times_{Y_0 \times_Y X} (Y'_0 \times_Y X) = Z \times_{Y_0} Y'_0
\]
which is $(n-2)$-regular since $Y_0$ is $(n-1)$-regular. By our inductive hypothesis we see that $Y_0 \times_Y X$ is $(n-1)$-regular. Similarly, we conclude that $Y_1 \times_Y X$ is also $(n-1)$-regular. 

Since $X$ is $n$-regular we have that $Y_0 \times_Y Y_1 \times_Y X$ is also $(n-1)$-regular. The base change of $f$ induces an effective epimorphism $Y_0 \times_Y Y_1 \times_Y X \rightarrow Y_0 \times_Y Y_1$. Applying our inductive hypothesis once more, we reduce to showing that for every map $Z \rightarrow Y_0 \times_Y Y_1$ with $Z$ in $\Ccal_0$, the fiber product
\[
Z \times_{Y_0 \times_Y Y_1} (Y_0 \times_Y Y_1 \times_Y X)
\]
is $(n-2)$-regular. Indeed, the above is equivalent to $Z \times_Y X$ and is $(n-2)$-regular since $Y$ is $(n-1)$-regular.
\end{proof}

\begin{lemma}\label{lemma descend regularity 2}
Let $\Ccal$ be a topos and let $n \geq 0$. Let $f: X \rightarrow Y$ be an effective epimorphism in $\Ccal$, where $X$ is $n$-regular. Then $Y$ is $n$-regular if and only if $X \times_Y X$ is $(n-1)$-regular.
\end{lemma}
\begin{proof}
The only if direction is clear. We now prove the if direction. Replacing $\Ccal$ by $\Ccal_{/Y}$ we may reduce to the case when $Y$ is terminal. Let $\Ccal_0$ be the image under the forgetful functor $\Ccal_{/X} \rightarrow \Ccal$ of the collection of $(n-1)$-regular objects of $\Ccal_{/X}$. As in the proof of lemma \ref{lemma estimates regularity} part (1) we see that $\Ccal_0$ is generating. Applying lemma \ref{lemma descend regularity} we may reduce to showing that if $Z$ belongs to $\Ccal_0$ then $Z \times X$ is $(n-1)$-regular. Since $Z$ belongs to $\Ccal_0$, there exists a map $Z \rightarrow X$, and we may write $Z \times X = Z \times_X (X \times X)$. By our assumption we have that $X \times X$ is $(n-1)$-regular. Our claim now follows from the fact that $X$ is $n$-regular.
\end{proof}

\begin{lemma}\label{lemma check regular on generating subcat}
Let $\Ccal$ be a topos and let $n \geq 0$. Let $\Ccal_0$ be a generating subcategory and assume that every object of $\Ccal_0$ is $(n-1)$-regular. Then $\Ccal$ is $n$-regular if and only if it is $0$-regular and for every pair of objects $X, Y$ in $\Ccal_0$, the product $X \times Y$ is $(n-1)$-regular.
\end{lemma}
\begin{proof}
The only if direction is clear so we focus on the if direction. We argue by induction on $n$. The case $n = 0$ is clear, so assume that $n \geq 1$ and that the assertion is known for $n-1$. Our assumptions guarantee that $\Ccal$ is locally $(n-1)$-regular and the inductive hypothesis implies that $\Ccal$ is $(n-1)$-regular. It remains to show that if $X$ and $Y$ are two $(n-1)$-regular objects of $\Ccal$ then $X \times Y$ is $(n-1)$-regular. Pick an effective epimorphism $X' \rightarrow X$ with $X'$ in $\Ccal_0$. Then the induced map $X' \times Y \rightarrow X \times Y$ is an effective epimorphism. Furthermore, for every morphism $Z \rightarrow X \times Y$ with $Z$ in $\Ccal_0$ we have $Z \times_{X \times Y} (X' \times Y) = Z \times_X X'$ which is $(n-2)$-regular by virtue of the fact that $X$ is $(n-1)$-regular. Applying lemma \ref{lemma descend regularity} we may reduce to showing that $X' \times Y$ is $(n-1)$-regular. Replacing $X$ by $X'$ if necessary we may now assume that $X$ belongs to $\Ccal_0$. A similar argument shows that we may further reduce to the case when $Y$ belongs to $\Ccal_0$. The desired assertion now follows from our hypothesis.
\end{proof}

\begin{proof}[Proof of theorem \ref{teo properties regularity}]
Part (1) follows as in the proof of lemma \ref{lemma estimates regularity} part (1).  
 We now prove part (2) . We will show that if $X$ is $n$-regular for some $n \geq 0$ then $X$ is $n$-coherent, by induction on $n$. The case $n = 0$ is clear. Assume now that $n \geq 1$ and the assertion is known to hold for $n-1$. Since $X$ is $(n-1)$-regular and locally $(n-1)$-regular we see by the inductive hypothesis that $X$ is $(n-1)$-coherent and locally $(n-1)$-coherent. To show that $X$  is $n$-coherent it remains to show that if $U\rightarrow X$ and $V\rightarrow X$ are $(n-1)$-coherent objects with maps to  $X$, then $U \times_X V$ is $(n-1)$-coherent. Pick an effective epimorphism $f: \coprod_\alpha U_\alpha \rightarrow U$, where $U_\alpha$ is $(n-1)$-regular for all $\alpha$. Since $U$ is $(n-1)$-coherent, it is in particular quasicompact, so we may assume that the covering family is finite. By induction, $\coprod_\alpha U_\alpha$ is $(n-1)$-coherent, and hence $f$ is relatively $(n-2)$-coherent.  Applying \cite{SAG} proposition A.2.1.3, we may reduce to showing that $\coprod_\alpha U_\alpha \times_X V$ is $(n-1)$-coherent. This is the same as $\coprod_\alpha (U_\alpha \times_X V)$, so replacing $U$ by $U_\alpha$ we may now reduce to the case when $U$ is $(n-1)$-regular. Similarly, we may further reduce to the case when $V$ is $(n-1)$-regular. The fact that $X$ is $n$-regular implies that $U \times_X V$ is $(n-1)$-regular, and hence $(n-1)$-coherent by inductive hypothesis.
 
 Part (3) follows from part (2) together with \cite{SAG} proposition A.2.3.1. Part (4) follows directly from the definitions. We now prove part (5). Assume first that $X$ is regular and let $k \geq -2$. Then the map $X \rightarrow \tau_{\leq k}(X)$ is $(k+1)$-connective, so by lemma \ref{lemma estimates regularity} part (2) we have that $\tau_{\leq k}(X)$ is $(k+1)$-regular. Applying lemma \ref{lemma estimates regularity} part (4) inductively we see that $\tau_{\leq k}(X)$  is $n$-regular for all $n$, and hence $\tau_{\leq k}(X)$ is regular. Conversely, assume that $\tau_{\leq k}(X)$ is regular for all $k$. Applying lemma \ref{lemma estimates regularity} part (3) to the $(k+1)$-connective map $X \rightarrow \tau_{\leq k}(X)$ we see that $X$ is $k$-regular. This holds for all $k$, and therefore $X$ is regular.
 
 Part (6) is a direct consequence of lemma \ref{lemma descend regularity 2}.  We now prove part (7). It suffices to show that an object in $\Ccal$ is $n$-regular if and only if its image in $\Ccal^\hyp$ is $n$-regular, for all $n \geq -1$. We prove this by induction on $n$. The case $n = -1$ is vacuous. Assume now that $n \geq 0$ and that the assertion is known for $n-1$. Since $\Ccal$ is regular, the class of $(n-1)$-regular objects in $\Ccal$ is generating, and therefore its image in $\Ccal^\hyp$ is a generating subcategory of $\Ccal^\hyp$. By inductive hypothesis, we see that $\Ccal^\hyp$ is locally $(n-1)$-regular. Our claim now follows by an application of lemma \ref{lemma regularity depends on truncationss}.
 
   We now prove part (8). We prove that $X$ is $n$-regular for all $n \geq 0$, by induction on $n$. The case $n = 0$ follows from the fact that $X_0$ is $0$-regular and the map $X_0 \rightarrow X$ is an effective epimorphism. Assume now that $n \geq 1$ and that the assertion is known for $n-1$. By lemma \ref{lemma descend regularity 2} we may reduce to showing that $X_0 \times_X X_0$ is $(n-1)$-regular. Let $X_\bullet^+$ be the semisimplicial object obtained from $X_\bullet$ by composing with the functor $\Delta^\op_{\text{s}} \rightarrow \Delta^\op_{\text{s}}$ sending $[m]$ to $[m] \star [0]$. Observe that we have  a projection $X_\bullet^+ \rightarrow X_\bullet$ which fits into a morphism of augmented semisimplicial objects whose value on the cone point is the projection $X_0 \rightarrow X$. Applying a combination of \cite{SAG} theorem A.5.4.1,  lemma A.5.5.3 and lemma A.5.5.4 we see that $X_\bullet^+ \times_{X_\bullet} X_\bullet^+$ is a Kan semisimplicial object, and furthermore the canonical map
   \[
   | X_\bullet^+ \times_{X_\bullet} X_\bullet^+| \rightarrow X_0 \times_X X_0
   \]
   is $\infty$-connective. Applying lemma \ref{lemma estimates regularity} part (2) we reduce to showing that  $| X_\bullet^+ \times_{X_\bullet} X_\bullet^+|$ is $(n-1)$-regular. This follows by our inductive hypothesis, in light of part (4).
   
The if direction of (9) follows from part (8). We now prove the only if direction. Let $X$ be a regular object of $\Ccal$. Replacing $\Ccal$ by $\Ccal_{/X}$ we may reduce to the case when $X = 1$ is the terminal object. We construct a hypercover $S_\bullet : \Delta^{\op}_{\text{s}} \rightarrow \Ccal_0$ inductively. Let $m \geq 0$ and assume that the restriction of $S_\bullet$ to $\Delta^{\op}_{\text{s}, < m}$ has been constructed. Combining (4) with the fact that the terminal object of $\Ccal$ is regular, we deduce that the collection of regular objects in $\Ccal$ is closed under finite limits. It follows that the matching object $S(\partial [m])$ is regular.  Since $\Ccal_0$ is generating, we may pick an effective epimorphism $f: U \rightarrow S(\partial [m])$ with $U$ in $\Ccal_0$.  The inductive step concludes by setting $S_m = U$, with attaching map given by $f$.

Finally, we prove part (10). We will show that every object of $\Ccal_0$ is $n$-regular for all $n \geq 0$, by induction on $n$. The case $n = 0$ is clear, so assume that $n \geq 1$ and that the assertion is known for $n-1$. Let $X$ be an object in $\Ccal_0$. Then $(\Ccal_0)_{/X}$ is a generating subcategory of $\Ccal_{/X}$ consisting of $(n-1)$-regular objects. By lemma \ref{lemma check regular on generating subcat}, to show that $X$ is $n$-regular it suffices to show that $(\Ccal_0)_{/X}$ is closed under finite products inside $\Ccal_{/X}$. This follows directly from the fact that $\Ccal_0$ is closed under fiber products.  
  \end{proof}

%%%%%%%%%%%%%%%%%%%%%%%%%%%%%%%%%%%%%%%%%%%%%%%%%%%%%%%%%%%%%%%%%%%%%%%%
%%%%%%%%%%%%%%%%%%%%%%%%%%%%%%%%%%%%%%%%%%%%%%%%%%%%%%%%%%%%%%%%%%%%%%%%
%%%%%%%%%%%%%%%%%%%%%%%%%%%%%%%%%%%%%%%%%%%%%%%%%%%%%%%%%%%%%%%%%%%%%%%%
%%%%%%%%%%%%%%%%%%%%%%%%%%%%%%%%%%%%%%%%%%%%%%%%%%%%%%%%%%%%%%%%%%%%%%%%
%%%%%%%%%%%%%%%%%%%%%%%%%%%%%%%%%%%%%%%%%%%%%%%%%%%%%%%%%%%%%%%%%%%%%%%%
%%%%%%%%%%%%%%%%%%%%%%%%%%%%%%%%%%%%%%%%%%%%%%%%%%%%%%%%%%%%%%%%%%%%%%%%

 \subsection{Singleton Grothendieck topologies}\label{subsection singleton}
 
 We now discuss a class of Grothendieck topologies which give rise to regular topoi.

\begin{definition}\label{definition singleton topology}
Let $\Ccal$ be a category with pullbacks. We say that a Grothendieck topology on $\Ccal$ is singleton if every covering sieve contains a morphism which itself generates a covering sieve.
\end{definition}

\begin{example}
Let $\ccal$ be a regular category. Then the regular topology on $\ccal$   is singleton.
\end{example}

Recall that if $\ccal$ is a category equipped with a Grothendieck topology then a morphism in $\Pcal(\ccal)$ is said to be a local effective epimorphism  if its projection to $\Sh(\ccal)$ is an effective epimorphism. Definition \ref{definition singleton topology} may be rephrased by saying that  the topology on $\ccal$ is singleton if and only if every covering sieve contains a local effective epimorphism.

\begin{proposition}\label{prop singleton topology}
Let $\Ccal$ be a category with pullbacks, equipped with a subcanonical singleton Grothendieck topology. Then every representable sheaf is regular. In particular, $\Sh(\Ccal)$ is locally regular and is regular provided that $\Ccal$ admits a terminal object.
\end{proposition}

\begin{proof}
Since $\Ccal$ admits pullbacks, an application of theorem \ref{teo properties regularity} part (10) reduces us to showing that every representable sheaf is $0$-regular. Let $X$ be an object of $\Ccal$ and let $X_\alpha$ be a cover of $X$ by subobjects. Let $X'$ be the disjoint union of the objects $X_\alpha$ computed in $\Pcal(\Ccal)$. Then the  projection $X' \rightarrow X$ is a local effective epimorphism. Since the topology on $\Ccal$ is singleton, there exists an object $Y$ in $\Ccal$ and a local effective epimorphism $p: Y \rightarrow X$  which factors through $X'$. Then $p$ factors through $X_\alpha$ for some $\alpha$, so we see that the inclusion $X_\alpha \rightarrow X$ is a local effective epimorphism, and hence an isomorphism.
\end{proof}

\begin{proposition}\label{prop describe hypercompletes}
Let $\Ccal$ be a category with pullbacks, equipped with a subcanonical singleton Grothendieck topology. Let $F$ be a presheaf on $\ccal$.
\begin{enumerate}[\normalfont (1)]
\item $F$ is a sheaf if and only if for every local effective epimorphism $f: X_0 \rightarrow X$ in $\Ccal$ with \v{C}ech nerve $X_\bullet$, the map $F(X) \rightarrow \Tot F(X_\bullet)$ is an isomorphism. 
\item $F$   is a hypercomplete sheaf if and only if for every object $X$ in $\Ccal$ and every semisimplicial hypercover $X_\bullet$ of $X$ inside $\Ccal$, the map $F(X) \rightarrow \Tot F(X_\bullet)$ is an equivalence.
\end{enumerate}
\end{proposition}
\begin{proof}
The only if direction in both parts is clear. We now prove the if direction of (1). Let $X$ be an object of $\ccal$ and let $U \rightarrow X$ be a covering sieve on $X$, which we think about as a subobject of $X$ inside $\Pcal(\ccal)$. We wish to show that the restriction map $\Hom_{\Pcal(\ccal)}(X , F) \rightarrow \Hom_{\Pcal(\ccal)}(U, F)$ is an isomorphism. Since the topology on $\ccal$ is singleton, there is a local effective epimorphism $f: X_0 \rightarrow X$ in $\ccal$ which factors through $U$. Let $V$ be the image of $f$ computed in $\Pcal(\ccal)$. Then our assumption guarantees that the restriction map $\Hom_{\Pcal(\ccal)}(X , F) \rightarrow \Hom_{\Pcal(\ccal)}(V, F)$  is an isomorphism. We may thus reduce to showing that the restriction map  $\Hom_{\Pcal(\ccal)}(U , F) \rightarrow \Hom_{\Pcal(\ccal)}(V, F)$ is an isomorphism. To do so it is enough to show that for every map $X' \rightarrow U$ with $X'$ in $\ccal$ the restriction map $\Hom_{\Pcal(\ccal)}(X' , F) \rightarrow \Hom_{\Pcal(\ccal)}(X' \times_U V, F)$ is an isomorphism. Note that $X' \times_U V$ is the image of the base change $f': X'_0 \rightarrow X'$ of $f$ computed in $\Pcal(\ccal)$. The desired claim now follows from our assumptions since $f'$ is a local effective epimorphism.

We now prove the if direction of (2). Let $f: X_0 \rightarrow X$ be a local effective epimorphism in $\Ccal$. Then the semisimplicial object underlying the \v{C}ech nerve of $f$ is a hypercover of $X$, and hence $F$ satisfies descent with respect to the sieve generated by $f$. Combined with (1) this implies that $F$ is a sheaf. 

It remains to show that $F$ is hypercomplete. To do so it is enough to show that if $\alpha: F \rightarrow F'$ is an $\infty$-connective morphism of sheaves on $\ccal$ such that  for every object $X$ in $\Ccal$ and every semisimplicial hypercover $X_\bullet$ of $X$ inside $\Ccal$ the maps $F(X) \rightarrow \Tot F(X_\bullet)$ and $F'(X) \rightarrow \Tot F'(X_\bullet)$ are equivalences, then $\alpha$ is an isomorphism. To do so it suffices to show that for every object $Y$ in $\ccal$ the map of spaces $\alpha(Y) : F(Y) \rightarrow F'(Y)$ is $n$-connective for all $n \geq 0$. We argue by induction on $n$. Assume for the moment that the case $n = 0$ has been proven. Then the inductive step follows as a consequence of the case $n = 0$ together with the inductive hypothesis applied to the diagonal $F \rightarrow F \times_{F'} F$. We may thus reduce to showing that for every object $Y$ in $\ccal$ the map of spaces $\alpha(Y) : F(Y) \rightarrow F'(Y)$ is essentially surjective. 

Fix a morphism $\eta: Y \rightarrow F'$. Our task is to provide a lift of $\eta$ to $F$. Combining proposition \ref{prop singleton topology} with part (5) of theorem \ref{teo properties regularity} we deduce that $Y \times_{F'} F$ is a regular object of $\Sh(\ccal)$. Applying part (9) of theorem \ref{teo properties regularity} we may construct a hypercover $Y_\bullet$ of $Y \times_{F'} F$ which factors through $\ccal$. By theorem \ref{theorem relate kan and hypercover topos} we have that $Y_\bullet$ also defines a hypercover of $Y$. Let $Z$ be the geometric realization of $Y_\bullet$ in $\Pcal(\ccal)$. Our assumptions on $F$ and $F'$ guarantee that the restriction maps $F(Y) \rightarrow F(Z)$ and $F'(Y) \rightarrow F'(Z)$ are isomorphisms, so we may reduce to showing that the composite map $Z \rightarrow Y \rightarrow F'$ lifts to $F$, which follows from the fact that the map $Z \rightarrow Y$ factors through $Y \times_{F'} F$.
\end{proof}

We now specialize the above discussion to the case of the regular topology on a regular category. 

\begin{proposition}\label{prop regular in sheaf topos}
Let $\Ccal$ be a regular category. Then $\Sh(\Ccal)$ and $\Sh(\Ccal)^\hyp$ are regular and locally regular topoi. Furthermore, for an object $X$ in $\Sh(\Ccal)$, the following are equivalent:
\begin{enumerate}[\normalfont (1)]
\item $X$ is regular.
\item There exists a semisimplicial hypercover of $X$ with values in $\Ccal$.
\item $\tau_{\leq k}(X)$ belongs to $\Ccal^\ex$ for all $k \geq -2$.
\end{enumerate}
\end{proposition}
\begin{proof}
The fact that $\Sh(\Ccal)$ is regular and locally regular is a direct consequence of proposition \ref{prop singleton topology}, since the regular topology on $\Ccal$ is subcanonical and singleton. The fact that the same holds for $\Sh(\Ccal)^\hyp$ follows from theorem \ref{teo properties regularity} part (7). The equivalence between properties (1) and (2) follows from a combination of proposition \ref{prop singleton topology} and part (9) of theorem  \ref{teo properties regularity}.

 We will finish the proof by showing that properties (1) and (3) are equivalent. By part (5) of theorem \ref{teo properties regularity} we may reduce to showing that if $X$ is $k$-truncated then $X$ is regular if and only if it belongs to $\Ccal^\ex$. The if direction follows from the fact that $\Ccal^\ex$ is generated by $\Ccal$ under geometric realizations of groupoids and regular objects are closed under geometric realizations of groupoids (theorem \ref{teo properties regularity} part (8)).  It remains to show that if $X$ is a $k$-truncated regular object then $X$ belongs to $\Ccal^\ex$.

We argue by induction on $k$. The case $k = -2$ is clear, so assume that $k \geq -1$ and that the assertion is known for $k-1$. Since $X$ is $0$-regular, we may pick an object $X_0$ in $\Ccal$ and an effective epimorphism $f: X_0 \rightarrow X$.  Let $X_\bullet$ be the \v{C}ech nerve of $f$. We have $X = |X_\bullet|$, so it will suffice to show that $X_\bullet$ factors through $\Ccal^\ex$. Since $X_\bullet$ is a groupoid object and $\Ccal^\ex$ is closed under finite limits in $\Sh(\Ccal)$, we may reduce to showing that $X_0$ and $X_1$ belong to $\Ccal^\ex$. The fact that $X_0$ belongs to $\Ccal^\ex$ is clear, so it remains to show that $X_1$ belongs to $\Ccal^\ex$. 

We have $X_1 = X_0 \times_X X_0$ and therefore $X_1$ is regular. Furthermore, since $X$ is $k$-truncated we see that the map $X_1 \rightarrow X_0 \times X_0$ is $(k-1)$-truncated. Observe that the inclusion $\Ccal_{/X_0 \times X_0} \rightarrow \Sh(\Ccal)_{/X_0 \times X_0}$ extends to an equivalence $\Sh(\Ccal_{/X_0 \times X_0}) = \Sh(\Ccal)_{/X_0 \times X_0}$. We may thus apply the inductive hypothesis to conclude that $X_1$ belongs to the closure of $\Ccal_{/X_0 \times X_0}$ under geometric realizations of groupoid objects in $\Sh(\Ccal)_{/X_0 \times X_0}$. The proof finishes by observing that the forgetful functor $\Sh(\Ccal)_{/X_0 \times X_0} \rightarrow \Sh(\Ccal)$ preserves geometric realizations of groupoid objects.
\end{proof}

%%%%%%%%%%%%%%%%%%%%%%%%%%%%%%%%%%%%%%%%%%%%%%%%%%%%%%%%%%%%%%%%%%%%%%%%
%%%%%%%%%%%%%%%%%%%%%%%%%%%%%%%%%%%%%%%%%%%%%%%%%%%%%%%%%%%%%%%%%%%%%%%%
%%%%%%%%%%%%%%%%%%%%%%%%%%%%%%%%%%%%%%%%%%%%%%%%%%%%%%%%%%%%%%%%%%%%%%%%
%%%%%%%%%%%%%%%%%%%%%%%%%%%%%%%%%%%%%%%%%%%%%%%%%%%%%%%%%%%%%%%%%%%%%%%%
%%%%%%%%%%%%%%%%%%%%%%%%%%%%%%%%%%%%%%%%%%%%%%%%%%%%%%%%%%%%%%%%%%%%%%%%
%%%%%%%%%%%%%%%%%%%%%%%%%%%%%%%%%%%%%%%%%%%%%%%%%%%%%%%%%%%%%%%%%%%%%%%%

\subsection{Regular \texorpdfstring{$(n,1)$}{(n,1)}-topoi}\label{subsection n1 topoi}

Let $n \geq 0$ be an integer. Then definition \ref{definition regular topos} makes sense whenever $\ccal$ is an $(n,1)$-topos instead of a topos. Our next goal is to study the resulting theory of regular $(n,1)$-topoi.

\begin{proposition}\label{prop equiv regularity in C o Ctruncated}
Let $n \geq 0$. Let $\ccal$ be a topos and assume that $\ccal_{\leq n-1}$ is a generating subcategory for $\ccal$. Let $X$ be an object of $\ccal_{\leq n-1}$ and let $k \geq 0$. The following are equivalent:
\begin{enumerate}[\normalfont (1)]
\item $X$ is $k$-regular as an object of the $(n,1)$-topos $\ccal_{\leq n-1}$.
\item $X$ is $k$-regular as an object of the topos $\ccal$.
\end{enumerate}
\end{proposition}
\begin{proof}
 We first show that (1) implies (2). We argue by induction on $k$. The case $k = 0$ follows directly from the definitions. Assume now that $k > 0$ and that the assertion is known for $k-1$.  Our assumptions on $\ccal$ imply that the topos $\ccal_{/X}$ is generated by the category of $(k-1)$-regular objects of the $(n,1)$-topos  $(\ccal_{\leq n-1})_{/X} = (\ccal_{/X})_{\leq n-1}$. Applying lemma \ref{lemma check regular on generating subcat} together with the inductive hypothesis we reduce to showing that for every pair of maps $X_0 \rightarrow X \leftarrow X_1$ in $\ccal_{\leq n-1}$ where $X_0$ and $X_1$ are $(k-1)$-regular objects of $\ccal_{\leq n-1}$ the object $X_0 \times_X X_1$ is $(k-1)$-regular in $\ccal_{\leq n-1}$. This follows from the fact that $X$ is $k$-regular in $\ccal_{\leq n-1}$

We now show that (2) implies (1). As before, the case $k = 0$ is clear, so we assume that $k > 0$ and that the assertion is known for $k-1$. Using that (1) implies (2) together with the inductive hypothesis we deduce that the full subcategory of $(\ccal_{\leq n-1})_{/X}$ on the $(k-1)$-regular objects is closed under finite limits. We may now reduce to showing that $X$ is locally $(k-1)$-regular as an object of $\ccal_{\leq n-1}$. Since $X$ is locally $(k-1)$-regular as an object of $\ccal$, it is enough to show that if $Y$ is a $(k-1)$-regular object of $\ccal_{/X}$ then $\tau_{\leq n-1}(Y)$ is a $(k-1)$-regular object of $\ccal_{\leq n-1}$. By the inductive hypothesis it suffices to prove that $\tau_{\leq n-1}(Y)$ is a $(k-1)$-regular object of $\ccal$. The case $k \leq n+1$ follows from part (2) of lemma \ref{lemma estimates regularity}. Assume now that $k \geq n+2$. Then by  lemma \ref{lemma estimates regularity} we have that $\tau_{\leq n-1}(Y)$ is an $n$-regular object of $\ccal$. Part (4) of lemma \ref{lemma estimates regularity} now implies that $\tau_{\leq n-1}(Y)$ is a $k$-regular object of $\ccal$.
\end{proof}

\begin{proposition}\label{prop regular topos vs regular n topos}
Let $n \geq 0$. Let $\ccal$ be a topos and assume that $\ccal_{\leq n-1}$ is a generating subcategory for $\ccal$. The following are equivalent:
\begin{enumerate}[\normalfont (1)]
\item The topos $\ccal$ is regular and locally regular.
\item The topos $\ccal$ is $(n+1)$-regular.
\item The $(n,1)$-topos $\ccal_{\leq n-1}$ is regular and locally regular.
\item The $(n,1)$-topos $\ccal_{\leq n-1}$ is $(n+1)$-regular.
\end{enumerate}
\end{proposition}
\begin{proof}
Note that (1) implies (2) and (3) implies (4). Conditions (2) and (4) are equivalent thanks to proposition \ref{prop equiv regularity in C o Ctruncated}. The equivalence between (1) and (3) follows from another application of proposition \ref{prop equiv regularity in C o Ctruncated} together with part (5) of theorem \ref{teo properties regularity}. 

We will finish the proof by showing that (2) and (4) imply (1). It will be enough for this to show that every $n$-regular $(n-1)$-truncated object of $\ccal$ is regular. We will do so by showing that if $k \geq n$ then every $n$-regular $(n-1)$-truncated object of $\ccal$ is $k$-regular. We argue by induction on $k$. Assume that $k > n$ and that the assertion is known for $k-1$. By part (4) of lemma \ref{lemma estimates regularity} it is enough to show that $\ccal$ is $k$-regular. The case $k = n+1$ is in our assumptions, while the case $k > n+1$ follows from our inductive hypothesis by applying lemma \ref{lemma check regular on generating subcat} with the subcategory of $(n-1)$-truncated $n$-regular objects of $\ccal$.
\end{proof}

Combining propositions \ref{prop regular topos vs regular n topos} and \ref{prop regular in sheaf topos}, we see that for every exact $(n,1)$-category $\ccal$ the $(n,1)$-topos $\Sh(\ccal)_{\leq n-1}$ is regular. Our next goal is to show that this provides a one to one correspondence between exact $(n,1)$-categories and regular $(n,1)$-topoi.

\begin{notation}
Let $n \geq 0$. We denote by $(n,1)\kr\Top$ the category whose objects are $(n,1)$-topoi and whose morphisms are geometric functors. For each $(n,1)$-topos $\ccal$ we denote by $\ccal^\reg$ the full subcategory of $\ccal$ on the regular objects.
\end{notation}

\begin{remark}\label{remark univ prop sh less n}
Let $n \geq 0$ and let $\Ccal$ be a regular $(n,1)$-category. Then for every $(n,1)$-topos $\Tcal$, restriction along the inclusion $\ccal \rightarrow \Sh(\ccal)_{\leq n-1}$ induces an equivalence between the category of geometric morphisms of topoi  $\Sh(\Ccal)_{\leq n-1} \rightarrow \Tcal$ and the category of regular functors $\Ccal \rightarrow \Tcal$. In particular, the assignment $\ccal \mapsto \Sh(\Ccal)_{\leq n-1} $ assembles into a functor $(n,1)\kr\Cat_{\ex} \rightarrow (n,1)\kr\Top$.
\end{remark}

\begin{theorem}\label{theorem regular n1 are}
Let $n \geq 0$. 
\begin{enumerate}[\normalfont (1)]
\item Let $\ccal$ be an exact $(n,1)$-category. Then the inclusion $\ccal \rightarrow \Sh(\ccal)_{\leq n-1}$ restricts to an equivalence $\ccal = (\Sh(\ccal)_{\leq n-1})^\reg$.
\item Let $\Tcal$ be a regular $(n,1)$-topos. Then $\Tcal^\reg$ is an small exact $(n,1)$-category, the inclusion $\Tcal^\reg \rightarrow \Tcal$ is regular, and the induced functor $\Sh(\Tcal^\reg)_{\leq n-1} \rightarrow \Tcal$ is an equivalence.
\item The assignment $\ccal \mapsto \Sh(\ccal)_{\leq n-1}$ induces an equivalence between $(n,1)\kr\Cat_{\ex}$ and the subcategory of $(n,1)\kr\Top$ on the regular $(n,1)$-topoi and geometric functors which preserve regular objects.
\end{enumerate} 
\end{theorem}
\begin{proof}
We begin by proving (1). Applying proposition \ref{prop equiv regularity in C o Ctruncated} we reduce to showing that the inclusion $\ccal \rightarrow \Sh(\ccal)$ induces an equivalence between $\ccal$ and the full subcategory of $\Sh(\ccal)$ on the $(n-1)$-truncated regular objects. This follows from a combination of propositions \ref{prop regular in sheaf topos} and \ref{prop ex completion of ex n cat}.

We now prove (2). Let $\kappa$ be a regular cardinal such that $\Tcal$ is $\kappa$-accessible, the full subcategory $\Tcal^\kappa$ of $\Tcal$ on the $\kappa$-compact objects is closed under finite limits, and $\Tcal^\kappa$ contains $\Tcal^\reg$. Equip $\Tcal^\kappa$ with the canonical topology (so that a sieve is covering if and only if it contains a jointly effectively epimorphic family) so that we have $\Tcal = \Sh(\Tcal^\kappa)_{\leq n-1}$. By propositions  \ref{prop equiv regularity in C o Ctruncated}, $\Tcal^\reg$ is the full subcategory of $\Sh(\Tcal^\kappa)$ on the $(n-1)$-truncated regular objects. Proposition \ref{prop regular topos vs regular n topos} shows that $\Sh(\Tcal^\kappa)$ is a regular and locally regular topos. By theorem \ref{teo properties regularity} we have that $\Sh(\Tcal^\kappa)^\hyp$ is coherent and locally coherent, and $\Tcal^\reg$ embeds inside the full subcategory of $\Sh(\Tcal^\kappa)^\hyp$ on the coherent objects. This implies that $\Tcal^\reg$ is small, by virtue of \cite{SAG} proposition A.6.6.1.  An application of theorem \ref{teo properties regularity} shows that the full subcategory of $\Sh(\Tcal^\kappa)$ on the regular objects is closed under finite limits and geometric realizations of groupoid objects.  The fact that $\Tcal^\reg$ is an exact $(n,1)$-category now follows from example \ref{example truncate exact is n exact}, while the regularity of the inclusion $\Tcal^\reg \rightarrow \Tcal$ is a consequence of remark \ref{remark effective epis in trucation}. 

To finish the proof of (2) it will suffice to show that the canonical map $\Sh(\Tcal^\reg)_{\leq n-1} \rightarrow \Sh(\Tcal^\kappa)_{\leq n-1} = \Tcal$ is an equivalence.  We claim that in fact the induced functor $\Sh(\Tcal^\reg)^\hyp \rightarrow \Sh(\Tcal^\kappa)^\hyp$ is an equivalence. This follows from an application of lemma \ref{lemma density}, using the fact that $\Tcal^\reg$ is a generating subcategory of $\Sh(\Tcal^\kappa)^\hyp$.

It remains to establish (3). The fact that the assignment $\ccal \mapsto \Sh(\ccal)_{\leq n-1}$  factors through the desired subcategory of $(n,1)\kr\Top$ follows at the level of objects from a combination of propositions \ref{prop regular topos vs regular n topos} and \ref{prop regular in sheaf topos}, and at the level of morphisms from part (1). The surjectivity of its corestriction follows from (2). It remains to prove that the corestriction is fully faithful. Let $\ccal$ and $\dcal$ be a pair of exact $(n,1)$-categories. We wish to show that the induced functor  $\Fun^\reg(\ccal, \dcal) \rightarrow \Fun(\Sh(\ccal)_{\leq n-1}, \Sh(\dcal)_{\leq n-1})$ is fully faithful, and its image is the full subcategory of $ \Fun(\Sh(\ccal)_{\leq n-1}, \Sh(\dcal)_{\leq n-1})$ on the geometric functors which preserve regular objects. Applying remark \ref{remark univ prop sh less n} we reduce to showing that a geometric functor $\Sh(\ccal)_{\leq n-1} \rightarrow \Sh(\dcal)_{\leq n-1}$ preserves regular objects if and only if it maps $\ccal$ to $\dcal$. This is a consequence of (1).
\end{proof}

\begin{remark}
In the case $n = 0$, theorem \ref{theorem regular n1 are} provides an equivalence between the category of posets with finite meets and the category of $1$-regular locales. Every such locale is spatial, so these categories are equivalent to the opposite of a full subcategory of the category of topological spaces. The resulting contravariant equivalence is a form of Stone duality.
\end{remark}
 
\ifx\inmain\undefined
\bibliographystyle{myamsalpha2}
\bibliography{References}
\fi

%%%%%%%%%%%%%%%%%%%%%%%%%%%%%%%%%%%%%%%%%%%%%%%%%%%%%%%%%%%%%%%%%%%%%%%%
%%%%%%%%%%%%%%%%%%%%%%%%%%%%%%%%%%%%%%%%%%%%%%%%%%%%%%%%%%%%%%%%%%%%%%%%
%%%%%%%%%%%%%%%%%%%%%%%%%%%%%%%%%%%%%%%%%%%%%%%%%%%%%%%%%%%%%%%%%%%%%%%%
%%%%%%%%%%%%%%%%%%%%%%%%%%%%%%%%%%%%%%%%%%%%%%%%%%%%%%%%%%%%%%%%%%%%%%%%
%%%%%%%%%%%%%%%%%%%%%%%%%%%%%%%%%%%%%%%%%%%%%%%%%%%%%%%%%%%%%%%%%%%%%%%%
%%%%%%%%%%%%%%%%%%%%%%%%%%%%%%%%%%%%%%%%%%%%%%%%%%%%%%%%%%%%%%%%%%%%%%%%

\section{Completeness conditions on exact categories}\label{section completeness}

The goal of this section  is to study various conditions that one may impose on an exact category which describe its behavior ``at $\infty$''.

 We begin in \ref{subsection hypercomplete exact} by introducing the notion of hypercomplete exact category, as a generalization of the notion of hypercompleteness of topoi from \cite{HTT}.  In higher topos theory, one usually defines a topos to be hypercomplete if all its $\infty$-connective morphisms are invertible. While every hypercomplete exact category has this property, not every exact category with this property is hypercomplete. The defining feature of hypercomplete exact categories is instead a correspondence between a class of semisimplicial objects called Kan semisimplicial objects, and a class of augmented semisimplicial objects called hypercovers. 
 
The first main result of this section is theorem \ref{theorem regular hypercomplete topoi}, which provides a one to one correspondence between (small) hypercomplete exact categories and hypercomplete regular and locally regular topoi. We may think of this as a regular version of the correspondence between small hypercomplete pretopoi and hypercomplete coherent and locally coherent topoi from \cite{SAG} appendix A.
 
In \ref{subsection postnikov} we study the notion of Postnikov complete  exact category. This is a joint generalization of the notions of Postnikov complete $\infty$-topos studied in \cite{SAG} appendix A, and (via the connection between exactness and prestability discussed in section \ref{section exact}) of the notion of complete Grothendieck prestable category studied in \cite{SAG} appendix C. We prove here that every exact category $\ccal$ admits a universal regular functor into a Postnikov complete exact category $\widehat{\ccal}$, which we call the Postnikov completion of $\ccal$.
  
Finally, in \ref{subsection hypercompletion} we apply the theory of regular topoi to show that every regular category $\ccal$ admits a universal regular functor into a hypercomplete exact category $\ccal^\hyp$, which we call the hypercompletion of $\ccal$. We show that if $\ccal$ is a presentable exact category then the hypercompletion of $\ccal$ agrees with the  localization of $\ccal$ at the class of $\infty$-connective morphisms; in particular, our notion of hypercompletion specializes in the setting of topoi to the usual definition. We then study the behavior of the hypercompletion in the case when $\ccal$ is an exact $(n,1)$-category. In this case we show that the functor $\ccal \rightarrow \ccal^\hyp$ is fully faithful, and that the assignment $\ccal \mapsto \ccal^\hyp$ provides a fully faithful embedding $(n,1)\kr\Cat_{\ex} \rightarrow \Cat_{\ex}$ whose image admits a concrete description.

%%%%%%%%%%%%%%%%%%%%%%%%%%%%%%%%%%%%%%%%%%%%%%%%%%%%%%%%%%%%%%%%%%%%%%%%
%%%%%%%%%%%%%%%%%%%%%%%%%%%%%%%%%%%%%%%%%%%%%%%%%%%%%%%%%%%%%%%%%%%%%%%%
%%%%%%%%%%%%%%%%%%%%%%%%%%%%%%%%%%%%%%%%%%%%%%%%%%%%%%%%%%%%%%%%%%%%%%%%
%%%%%%%%%%%%%%%%%%%%%%%%%%%%%%%%%%%%%%%%%%%%%%%%%%%%%%%%%%%%%%%%%%%%%%%%
%%%%%%%%%%%%%%%%%%%%%%%%%%%%%%%%%%%%%%%%%%%%%%%%%%%%%%%%%%%%%%%%%%%%%%%%
%%%%%%%%%%%%%%%%%%%%%%%%%%%%%%%%%%%%%%%%%%%%%%%%%%%%%%%%%%%%%%%%%%%%%%%%

\subsection{Hypercomplete exact categories}\label{subsection hypercomplete exact}

We begin by reviewing the notion of (trivial) Kan fibration  of semisimplicial objects from \cite{SAG} appendix A.5.

\begin{notation}\label{notation evaluate semisimplicial object}
Let $\Ccal$ be a finitely complete category and let $X_\bullet$ be a semisimplicial object in $\Ccal$. Since $\Ccal$ is finitely complete, $X_\bullet$ admits a right Kan extension to the opposite of the category of finite semisimplicial spaces. For each finite semisimplicial space $K$ we denote by $X(K)$ the value on $K$ of the right Kan extension of $X_\bullet$.
\end{notation}

\begin{notation}
Let $n \geq 0$. We denote by $\partial [n]$ the boundary of the $n$-simplex, and for each $0 \leq i \leq n$ we denote by $\Lambda^n_i$ the $i$-th horn. We regard $[n]$, $\partial[n]$ and $\Lambda^n_i$ as finite semisimplicial spaces. In particular, for any semisimplicial object $X_\bullet$ in a finitely complete category it makes sense to consider the objects $X([n])$, $X(\partial[n])$ and $X(\Lambda^n_i)$ defined as in notation \ref{notation evaluate semisimplicial object}.
\end{notation}

\begin{definition}\label{definition kan fibrations}
Let $\Ccal$ be a regular category. A morphism $f: X_\bullet  \rightarrow Y_\bullet$ of semisimplicial objects in $\Ccal$ is said to be a Kan fibration if for every $n \geq 1$ and every $0 \leq i \leq n$, the map
\[
X([n]) \rightarrow Y([n]) \times_{Y(\Lambda^n_i ) } X(\Lambda^n_i)
\]
is an effective epimorphism. We say that $f$ is a trivial Kan fibration if for every $n \geq 0$ the map
\[
X([n]) \rightarrow Y([n]) \times_{Y(\partial[n])} X(\partial[n])
\]
is an effective epimorphism.
\end{definition}

The following basic properties are established\footnote{Although stated for topoi, the proofs only require regularity. Alternatively, one may reduce to the case of topoi by composing with the inclusion $\Ccal \rightarrow \Sh(\Ccal)$.} in \cite{SAG} A.5.2.

\begin{proposition}
Let $\Ccal$ be a regular category.
\begin{enumerate}[\normalfont (1)]
\item Trivial Kan fibrations are Kan fibrations.
\item Trivial Kan fibrations and Kan fibrations are stable under composition and base change, and contain isomorphisms.
\item Let $f: X_\bullet  \rightarrow Y_\bullet $ and $g: Y_\bullet \rightarrow Z_\bullet$ be morphisms of semisimplicial objects in $\Ccal$. If $f$ is a trivial Kan fibration and $g \circ f$ is a Kan fibration (resp. trivial Kan fibration) then $g$ is a  Kan fibration (resp. trivial Kan fibration).
\end{enumerate}
\end{proposition}

\begin{definition}
Let $\ccal$ be a regular category. We say that a semisimplicial object $X_\bullet$ in $\ccal$ satisfies the Kan condition if the projection from $X_\bullet$ to the terminal semisimplicial object is a Kan fibration. In this case we will also say that $X_\bullet$ is a Kan semisimplicial object.
\end{definition}

\begin{example}
Let $\ccal$ be a regular category and let $X_\bullet$ be a  groupoid object in $\ccal$. Then the semisimplicial object underlying $X_\bullet$ satisfies the Kan condition.
\end{example}

\begin{remark}
Let $\ccal$ be a regular category and let $X$ be an object in $\ccal$. Then a semisimplicial object in $\ccal_{/X}$ satisfies the Kan condition if and only if is image in $\ccal$ satisfies the Kan condition.
\end{remark}

\begin{definition}\label{definition hypercover}
Let $\ccal$ be a regular category. Let $Y$ be an object in $\ccal$ and $X_\bullet$ be a semisimplicial object in $\ccal_{/Y}$. We say that $X_\bullet$ is a hypercover of $Y$ if the projection from $X_\bullet$ to the terminal semisimplicial object in $\ccal_{/Y}$ is a trivial Kan fibration.
\end{definition}

\begin{remark}
Let $\ccal$ be a regular category and let $Y$ be an object of $\ccal$. Then we will usually identify semisimplicial objects in $\ccal_{/Y}$ with pairs of a semisimplicial object $X_\bullet$ in $\ccal$ and an augmentation $X_\bullet \rightarrow Y$.
\end{remark}

In the context of topoi, hypercovers and Kan semisimplicial objects are related by the following:

\begin{theorem}[\cite{SAG} corollary A.5.6.3]\label{theorem relate kan and hypercover topos}
Let $\ccal$ be a topos and let $X_\bullet \rightarrow Y$ be an augmented semisimplicial object in $\ccal$. The following are equivalent:
\begin{enumerate}[\normalfont (1)]
\item $X_\bullet$ is a hypercover of $Y$.
\item $X_\bullet$ satisfies the Kan condition and the map $|X_\bullet| \rightarrow Y$ is $\infty$-connective.
\end{enumerate}
In particular, every Kan semisimplicial object of $\ccal$ is a hypercover of its geometric realization.
\end{theorem}

We now arrive at the central definition of this section:

\begin{definition}\label{definition hypercomplete exact}
Let $\ccal$ be an exact category. We say that $\ccal$ is hypercomplete if the following conditions are satisfied:
\begin{enumerate}[\normalfont (1)]
\item Every Kan semisimplicial object $X_\bullet$ in $\ccal$ admits a geometric realization, and furthermore $X_\bullet$ is a hypercover of $|X_\bullet|$.
\item Let $X_\bullet$ be a semisimplicial hypercover of an object $Y$ in $\ccal$. Then $Y$ is the geometric realization of $X_\bullet$. 
\end{enumerate}
\end{definition}

\begin{notation}
Let $\Cat_\hyp$ be the full subcategory of $\Cat_\ex$ on the hypercomplete exact categories.
\end{notation}

\begin{example}
Let $\ccal$ be a topos. Then it follows from theorem \ref{theorem relate kan and hypercover topos} that $\ccal$ is hypercomplete in the sense of definition \ref{definition hypercomplete exact} if and only if every $\infty$-connective morphism in $\ccal$ is invertible (that is, if and only if $\ccal$ is hypercomplete in the usual sense).
\end{example}

\begin{remark} 
It turns out that the condition that $X_\bullet$ be a hypercover of $|X_\bullet|$ in item (1) of definition \ref{definition hypercomplete exact} is automatic, see proposition \ref{prop colimit of Kan is hypercover} below. Furthermore, item (2) may be replaced with the condition that $\infty$-connective morphisms in $\ccal$ be invertible, see corollary \ref{coro equivalence hypercomplete}. 
\end{remark}

\begin{remark}
Let $\ccal$ be an exact category. Then applying theorem \ref{theorem relate kan and hypercover topos} to $\Sh(\ccal)$ we deduce that if $X_\bullet$ is a hypercover of an object $Y$ in $\ccal$ then $X_\bullet$ satisfies the Kan condition. If $\ccal$ is hypercomplete then this provides a one to one correspondence between Kan semisimplicial objects and hypercovers in $\ccal$.
\end{remark}

\begin{remark}
Let $\ccal$ be an exact category. Then the assignment $X \mapsto \ccal_{/X}$ assembles into a functor $\ccal^\op \rightarrow \Cat$ which satisfies descent for the regular topology on $\ccal$, as may be seen for instance by using the fact that the assignment $X \mapsto \Sh(\ccal)_{/X}$ maps colimits in $\Sh(\ccal)$ to limits of categories. Assume now that $\ccal$ is hypercomplete. Then the same arguments show that the assignment $X \mapsto \ccal_{/X}$ sends hypercovers in $\ccal$ to totalizations in $\Cat$.
\end{remark}

 Our next goal is to show that the assignment $\ccal \mapsto \Sh(\ccal)^\hyp$ provides a one to one correspondence between hypercomplete exact categories and hypercomplete regular and locally regular topoi.

\begin{notation}
 We denote by $\Top$ the category whose objects are  topoi and whose morphisms are geometric functors. For each topos $\ccal$ we denote by $\ccal^\reg$ the full subcategory of $\ccal$ on the regular objects.
\end{notation}

\begin{theorem}\label{theorem regular hypercomplete topoi}
\hfill
\begin{enumerate}[\normalfont (1)]
\item Let $\ccal$ be a hypercomplete exact category. Then an object of $\Sh(\Ccal)$ belongs to $\Ccal$ if and only if it is hypercomplete and regular.
\item Let $\Tcal$ be a hypercomplete, regular, and locally regular topos. Then $\Tcal^\reg$ is a small hypercomplete exact category, the inclusion $\Tcal^\reg \rightarrow \Tcal$ is regular, and the induced functor $\Sh(\Tcal^\reg)^\hyp \rightarrow \Tcal$ is an equivalence.
\item The assignment $\ccal \mapsto \Sh(\ccal)^\hyp$ induces an equivalence between $\Cat_{\hyp}$ and the subcategory of $\Top$ on the hypercomplete, regular, and locally regular topoi and the geometric functors which preserve regular objects.
\end{enumerate} 
\end{theorem}
\begin{proof}
We begin by proving (1). The fact that every representable sheaf is hypercomplete is a consequence of the fact that $\Ccal$ is hypercomplete, by proposition \ref{prop describe hypercompletes}. The fact that representable sheaves are regular follows directly from proposition \ref{prop singleton topology}. It remains to show that every hypercomplete regular object $X$ in $\Sh(\Ccal)$ belongs to $\Ccal$. By proposition \ref{prop regular in sheaf topos} there exists a semisimplicial hypercover $X_\bullet$ of $X$ by objects in $\Ccal$. Then $X_\bullet$ is a Kan semisimplicial object in $\Ccal$, and it is therefore a hypercover for an object $X'$ in $\Ccal$ (since $\ccal$ is hypercomplete). In $\Sh(\Ccal)$ we have morphisms $|X_\bullet| \rightarrow X$ and $|X_\bullet| \rightarrow X'$  which are $\infty$-connective. Since $X'$ belongs to $\Ccal$ it is hypercomplete. The fact that $X$ is also hypercomplete now implies that $X = X'$, and so $X$ belongs to $\Ccal$, as desired.

We now prove (2). The fact that $\Tcal^\reg$ is hypercomplete regular is a direct consequence of the fact that $\Tcal^\reg$ is closed under finite limits and geometric realizations of Kan semisimplicial objects in $\Tcal$ (theorem \ref{teo properties regularity}). The same facts imply that the inclusion $\Tcal^\reg \rightarrow \Tcal$ is regular and creates effective epimorphisms. The fact that $\Tcal^\reg$ is small follows from the fact that the full subcategory of $\Tcal$ on the coherent objects is small, see \cite{SAG} proposition A.6.6.1.

Let $\kappa$ be a regular cardinal such that $\Tcal$ is $\kappa$-accessible, the full subcategory $\Tcal^\kappa$ of $\Tcal$ on the $\kappa$-compact objects is closed under finite limits, and $\Tcal^\kappa$ contains $\Tcal^\reg$. Equip $\Tcal^\kappa$ with the canonical topology (so that a sieve is covering if and only if it contains a jointly effectively epimorphic family) so that we have $\Tcal = \Sh(\Tcal^\kappa)^\hyp$. To finish the proof of (2) it will suffice to show that the canonical map $\Sh(\Tcal^\reg)^\hyp \rightarrow \Sh(\Tcal^\kappa)^\hyp = \Tcal$ is an equivalence.  This follows from  lemma \ref{lemma density} using the fact that $\Tcal^\reg$ is a generating subcategory for $\Tcal$.

We now prove (3). The fact that the assignment $\ccal \mapsto \Sh(\ccal)^\hyp$  factors through the desired subcategory of $ \Top$ follows at the level of objects from proposition \ref{prop regular in sheaf topos}, and at the level of morphisms from part (1). The surjectivity of its corestriction follows from (2). It remains to prove that the corestriction is fully faithful. Let $\ccal$ and $\dcal$ be a pair of hypercomplete exact categories. We wish to show that the induced functor  $\Fun^\reg(\ccal, \dcal) \rightarrow \Fun(\Sh(\ccal)^\hyp, \Sh(\dcal)^\hyp)$ is fully faithful, and its image is the full subcategory of $ \Fun(\Sh(\ccal)^\hyp, \Sh(\dcal)^\hyp)$ on the geometric functors which preserve regular objects. Applying remark \ref{remark univ prop sh} we reduce to showing that a geometric functor $\Sh(\ccal)^\hyp \rightarrow \Sh(\dcal)^\hyp$ preserves regular objects if and only if it maps $\ccal$ to $\dcal$. This is a consequence of (1).
\end{proof}

%%%%%%%%%%%%%%%%%%%%%%%%%%%%%%%%%%%%%%%%%%%%%%%%%%%%%%%%%%%%%%%%%%%%%%%%
%%%%%%%%%%%%%%%%%%%%%%%%%%%%%%%%%%%%%%%%%%%%%%%%%%%%%%%%%%%%%%%%%%%%%%%%
%%%%%%%%%%%%%%%%%%%%%%%%%%%%%%%%%%%%%%%%%%%%%%%%%%%%%%%%%%%%%%%%%%%%%%%%
%%%%%%%%%%%%%%%%%%%%%%%%%%%%%%%%%%%%%%%%%%%%%%%%%%%%%%%%%%%%%%%%%%%%%%%%
%%%%%%%%%%%%%%%%%%%%%%%%%%%%%%%%%%%%%%%%%%%%%%%%%%%%%%%%%%%%%%%%%%%%%%%%
%%%%%%%%%%%%%%%%%%%%%%%%%%%%%%%%%%%%%%%%%%%%%%%%%%%%%%%%%%%%%%%%%%%%%%%%

\subsection{Postnikov complete exact categories}\label{subsection postnikov}

We now discuss Postnikov completions in the context of exact categories.

\begin{notation}
Let $\Ccal$ be an exact category. We denote by $\widehat{\Ccal}$ the limit $\lim_k \Ccal_{\leq k}$, and by $e: \Ccal \rightarrow \widehat{\Ccal}$ the induced functor. 
\end{notation}

\begin{definition}
Let $\Ccal$ be an exact category. We say that $\Ccal$ is Postnikov complete if the functor $e: \Ccal \rightarrow \widehat{\Ccal}$ is an equivalence. We denote by $\Cat_\comp$ the full subcategory of $\Cat_\ex$ on the Postnikov complete exact categories.
\end{definition}

\begin{proposition}\label{prop postnikov}
Let $\Ccal$ be an exact category. Then:
\begin{enumerate}[\normalfont (1)]
\item $\widehat{\Ccal}$ is a Postnikov complete exact category.
\item $\widehat{\Ccal}$ is hypercomplete.
\item The functor $e: \Ccal \rightarrow \widehat{\Ccal}$ is regular.
\item For every $k \geq -2$ the functor $\Ccal_{\leq k} \rightarrow (\widehat{\Ccal})_{\leq k}$ induced by $e$ is an equivalence.
\end{enumerate}
\end{proposition}
\begin{proof}
We make use of \cite{SAG} theorem A.7.2.4, which shows that the analogous proposition holds in the case of topoi (where regular functors are replaced by geometric morphisms). We identify $\widehat{\Ccal}$ with the full subcategory of the topos $\shchat$ on those objects $X$ such that $\tau_{\leq k}(X)$ belongs to $\Ccal_{\leq k}$ for all $k$. To prove the proposition it suffices to show that $\widehat{\Ccal}$ is closed under finite limits and colimits of Kan semisimplicial objects inside $\shchat$.

The fact that $\widehat{\Ccal}$ contains the terminal object of $\shchat$ follows from the fact that $\Ccal$ contains the final object of $\Sh(\Ccal)$. We now prove that $\widehat{\Ccal}$ is closed under pullbacks. Let $X \rightarrow Y$ and $Y' \rightarrow Y$ be a pair of maps in $\widehat{\Ccal}$. For each $k \geq -2$ we have 
\[
\tau_{\leq k}(X \times_Y Y') = \tau_{\leq k}(\tau_{\leq k+1}(X) \times_{\tau_{\leq k+1}(Y)} \tau_{\leq k+1}(Y')).
\]
Since the projection $\Sh(\Ccal) \rightarrow \shchat$ preserves fiber products and truncations and $\ccal$ is closed under fiber products and truncations inside $\Sh(\ccal)$ we see that the right hand side in the above equation belongs to $\ccal_{\leq k} \subseteq \shchat_{\leq k}$. This holds for all $k$, and therefore $X \times_Y Y'$ belongs to $\widehat{\Ccal}$ as well. This shows that $\widehat{\Ccal}$ is closed under fiber products inside $\shchat$, as desired.

It remains to show that if $X_\bullet$ is a Kan semisimplicial object in $\widehat{\Ccal}$, then $X = |X_\bullet|$ belongs to $\widehat{\Ccal}$. In other words, we have to show that for every $n \geq -1$, the truncation $\tau_{\leq n-1}(X)$ belongs to $\Ccal_{\leq {n-1}}$. Combining proposition \ref{prop regular in sheaf topos} with lemma \ref{lemma estimates regularity} we may reduce to showing that $\tau_{\leq n}(X)$ is an $n$-regular object of $\Sh(\Ccal)$ for every $n \geq -1$. We will prove this assertion by induction on $n$. The case $n = -1$ is clear, so assume that $n \geq 0$ and that the assertion is known for $n-1$.

Let $X_\bullet^+$ be the semisimplicial object obtained from $X_\bullet$ by composing with the functor $\Delta^\op_{\text{s}} \rightarrow \Delta^\op_{\text{s}}$ sending $[m]$ to $[m] \star [0]$. Observe that we have a projection $X^+_\bullet \rightarrow X_\bullet$ which fits into a morphism of augmented semisimplicial objects whose value on the cone point is the projection $X_0 \rightarrow X$. Applying a combination of \cite{SAG} theorem A.5.4.1, lemma A.5.5.3 and lemma A.5.5.4 we see that $X_\bullet^+ \times_{X_\bullet} X^+_\bullet$ is a Kan semisimplicial object, and furthermore the canonical map $|X_\bullet^+ \times_{X_\bullet} X_\bullet^+| \rightarrow X_0 \times_X X_0$ is an isomorphism. Combining the fact that $\widehat{\Ccal}$ is closed under fiber products in $\shchat$ together with our inductive hypothesis we may conclude that $\tau_{\leq n-1}(X_0 \times_X X_0)$ is an $(n-1)$-regular object of $\Sh(\Ccal)$.  Note that we  have an $n$-connective morphism
\[
X_0 \times_X X_0 \rightarrow \tau_{\leq n}(X_0) \times_{\tau_{\leq n}(X)} \tau_{\leq n}(X_0)
\]
and so an application of lemma \ref{lemma estimates regularity} shows that the right hand side is $(n-1)$-regular as well. The fact that $\tau_{\leq n}(X)$ is $n$-regular now follows from lemma \ref{lemma descend regularity 2}.
\end{proof}

\begin{corollary}
Let $\ccal$ be a Postnikov complete exact category. Then $\ccal$ is hypercomplete.
\end{corollary}
\begin{proof}
Follows directly from part (2) of proposition \ref{prop postnikov}.
\end{proof}

The category $\widehat{\Ccal}$ is called the Postnikov completion of $\Ccal$. The terminology is justified by the following:

\begin{proposition}\label{proposition univer prop postnikov}
Let $\Ccal$ be an exact category and let $\Dcal$ be a Postnikov complete exact category. Then precomposition with $e$ induces an equivalence $\Fun^\reg(\widehat{\Ccal}, \Dcal) = \Fun^\reg(\Ccal, \Dcal)$.
\end{proposition}
\begin{proof}
Note that the assignment $\Ccal \mapsto \widehat{\Ccal}$ assembles into a functor $Q: \Cat_\ex \rightarrow \Cat$. We claim that $Q$ factors through $\Cat_{\ex}$. This is true objectwise by proposition \ref{prop postnikov}. Assume now given a regular functor of exact categories $F: \Ccal \rightarrow \Dcal$, and denote by $F^*: \Sh(\Ccal )\rightarrow \Sh(\Dcal)$ the induced geometric morphism. We have a commutative square
\[
\begin{tikzcd}
\widehat{\Ccal} \arrow{r}{Q(F)} \arrow{d}{} & \widehat{\Dcal} \arrow{d}{} \\
\widehat{\Sh(\ccal)} \arrow{r}{Q(F^*)} \arrow{r}{} &  \widehat{\Sh(\dcal)}
\end{tikzcd}
\]
where the upper vertical arrows create finite limits and geometric realization of groupoid objects, as shown in the proof of proposition \ref{prop postnikov}. The functor $Q(F^*)$ is a geometric morphism of topoi by \cite{SAG} theorem A.7.2.4. It now follows that $Q(F)$ is a regular functor, and hence $Q$ factors through $\Cat_{\ex}$, as desired.

Denote by  $L: \Cat_\ex \rightarrow \Cat_\ex$  the corestriction of $Q$. By part (3) of proposition \ref{prop postnikov} we have a natural transformation $\eta: \id \rightarrow L$ whose value on each exact category $\Ccal$ is the functor $e: \Ccal \rightarrow \widehat{\Ccal}$. We claim that $L$ is a localization functor with image $\Cat_\comp$. The fact that the image of $L$ is contained in $\Cat_\comp$ follows from part (1) of proposition \ref{prop postnikov}. The fact that $\eta$ is an equivalence when restricted to $\Cat_\comp$ follows from the definition of Postnikov completeness. Finally, the fact that $L(\eta)$ is an equivalence follows from part (4) of proposition \ref{prop postnikov}.

It now follows that the functor $\Fun^\reg(\widehat{\Ccal}, \Dcal) \rightarrow \Fun^\reg(\Ccal, \Dcal)$ of precomposition with $e$ induces an equivalence on spaces of objects. To show that it is an equivalence it remains to show that it also induces an equivalence on spaces of arrows. This follows from the fact that the functor $\Fun^\reg(\widehat{\ccal}, \Fun([1], \dcal)) \rightarrow \Fun^\reg(\ccal, \Fun([1],\dcal))$ of precomposition with $e$ induces an equivalence on spaces of objects, since $\Fun([1], \dcal)$ is a Postnikov complete exact category.
\end{proof}

A Postnikov complete exact category $\ccal$ is completely determined by the sequence of exact $(n,1)$-categories $\ccal_n = \ccal_{\leq n-1}$ together with the equivalences $\ccal_n = (\ccal_{n+1})_{\leq n-1}$. Our next goal is to show that any such sequence gives rise to a Postnikov complete exact category.

\begin{construction}\label{construction inverse tower}
Let $\Cat_\lex$ be the subcategory of $\Cat$ on the finitely complete categories and functors which preserve finite limits, and for each $n \geq 0$ let $(n,1)\kr\Cat_\lex$ be the full subcategory of $\Cat_\lex$ on the finitely complete $(n,1)$-categories. We have a diagram of inclusions as follows:
\[
\Cat_\lex \hookleftarrow \ldots \hookleftarrow  (2,1)\kr\Cat_{\lex} \hookleftarrow (1,1)\kr\Cat_\lex \hookleftarrow (0,1)\kr\Cat_\lex
\]
Each of the above functors admits a right adjoint given by passage to full subcategories of ($n$-)truncated objects, so we have a diagram of categories
\[
\Cat_\lex \rightarrow \ldots \rightarrow  (2,1)\kr\Cat_{\lex} \rightarrow (1,1)\kr\Cat_\lex \rightarrow (0,1)\kr\Cat_\lex.
\]
Passing to the subcategories on the exact categories and exact $(n,1)$-categories we obtain a diagram
\[
\Cat_\ex \rightarrow \ldots \rightarrow  (2,1)\kr\Cat_{\ex} \rightarrow (1,1)\kr\Cat_\ex \rightarrow (0,1)\kr\Cat_\ex.
\]
\end{construction}

\begin{proposition}\label{proposition bounded and complete vs truncated}
Let $p: \Cat_\ex \rightarrow \lim (n,1)\kr \Cat_\ex$ be the functor induced from the diagram of construction \ref{construction inverse tower}.
\begin{enumerate}[\normalfont (1)]
\item The functor $p$ admits a fully faithful left adjoint whose image is $\Cat_\ex^\bounded$.
\item The functor $p$ admits  a fully faithful right adjoint whose image is $\Cat_\comp$.
\end{enumerate} 
\end{proposition}
\begin{proof}
The restriction of $p$ to $\Cat_{\ex}^\bounded$ is conservative, so to prove (1) it will suffice to show that $p$ has a fully faithful left adjoint which factors through $\Cat_{\ex}^\bounded$.  Let $(\ccal_n)$ be an object in $\lim (n,1)\kr\Cat_\ex$. In other words, $(\ccal_n)$ is a sequence of exact $(n,1)$-categories such that $\ccal_n = (\ccal_{n+1})_{\leq n-1}$ for all $n \geq 0$. It follows from  \cite{HTT} proposition 5.5.7.11 that the projection $\Cat_\lex  \rightarrow \lim (n,1)\kr \Cat_\lex$  admits a left adjoint which sends $(\ccal_n)$ to $\ccal = \colim \ccal_n$. Note that for each $n \geq 0$ we have $\ccal_{\leq n-1} = \ccal_n$.  We may now reduce to showing the following:
\begin{enumerate}[\normalfont (i)]
\item $\ccal$ is the colimit of $\ccal_n$ inside $\Cat_\reg$.
\item $\ccal$ is a bounded exact category.
\end{enumerate}
Since the inclusions $\ccal_n \rightarrow \ccal$ are left exact and every object in $\ccal_n$ is truncated we deduce that every object in $\ccal$ is truncated. Let $f: X \rightarrow Y$ be a morphism in $\ccal$ with \v{C}ech nerve $X_\bullet$ and let $n\geq 0$ be such that $X$ and $Y$ belong to $\ccal_n$.  The fact that the inclusion $\ccal_n \rightarrow \ccal$ is left exact implies that $X_\bullet$ also factors through $\ccal_n$.  Since the inclusions $\ccal_m \rightarrow \ccal_{m+1}$ preserve geometric realizations of \v{C}ech nerves we have that $X_\bullet$ admits a geometric realization in $\ccal$, which belongs to $\ccal_n$.

Assume now given a base change $f': X' \rightarrow Y'$  for $f$ with \v{C}ech nerve $X'_\bullet$. Enlarging $n$ if necessary we may assume that $X'$ and $Y'$ belong to $\ccal_n$, which implies that $X_\bullet$ factors through $\ccal_n$. Since $\ccal_n$ is regular we have that $|X'_\bullet| = Y' \times_Y |X_\bullet|$, which implies that $\ccal$ is regular. To finish the proof of (i) we only need to observe that the class of effective epimorphisms in $\ccal$ is the union of the classes of effective epimorphisms on $\ccal_n$, which follows from the above description of the geometric realization of \v{C}ech nerves in $\ccal$. 

To prove (ii) it only remains to show that every groupoid object $X_\bullet$ in $\ccal$ is effective. Let $n\geq 0$ be such that $X_0$ and $X_1$ belong to $\ccal_n$ and $X_\bullet$ is $n$-efficient. Since $\ccal_n$ is closed under finite limits inside $\ccal$ we see that $X_\bullet$ factors through $\ccal_n$. The fact that $X_\bullet$ is a \v{C}ech nerve now follows from the fact that $\ccal_n$ is an exact $(n,1)$-category.

We now prove (2). By proposition \ref{proposition univer prop postnikov}, the inclusion $\Cat_\comp \rightarrow \Cat_\ex$ admits a left adjoint given by Postnikov completion. It will therefore suffice to show the following:
\begin{enumerate}[\normalfont (a)]
\item For each exact category $\ccal$, the induced map $p(\ccal) \rightarrow p(\widehat{\ccal})$ is an isomorphism.
\item The restriction of $p$ to $\Cat_\comp$ is an equivalence.
\end{enumerate}
Item (a) follows from proposition \ref{prop postnikov}. We now prove (b). The restriction of $p$ to $\Cat_\comp$ is conservative, so it suffices to show that it admits a fully faithful left adjoint. It follows from the first part of the proof that $p|_{\Cat_\comp}$ admits a left adjoint given by the composition
\[
\lim (n,1)\kr \Cat_\ex = \Cat_\ex^\bounded \hookrightarrow \Cat_\ex \xrightarrow{\widehat{(-)}} \Cat_\comp.
\]
We may thus reduce to showing that for every pair of bounded exact categories $\ccal, \dcal$ the canonical functor $\Fun^\reg(\ccal, \dcal) \rightarrow \Fun^\reg(\widehat{\ccal}, \widehat{\dcal})$ is an equivalence. This factors as the composition
\[
\Fun^\reg(\ccal, \dcal) \rightarrow \Fun^\reg(\ccal, \widehat{\dcal}) = \Fun^\reg(\widehat{\ccal}, \widehat{\dcal})
\]
so it is enough to show that the first arrow in the above composition is an equivalence. Since $\ccal$ is bounded it suffices to show that the map $\dcal \rightarrow \widehat{\dcal}$ induces an equivalence between the full subcategories  of $\dcal$ and $\widehat{\dcal}$ on the truncated objects, which follows from proposition \ref{prop postnikov}.
\end{proof}

 Proposition \ref{prop postnikov} and its proof may be used to deduce some useful general facts about Kan semisimplicial objects and hypercovers in exact categories.
 
\begin{proposition}\label{prop colimit of Kan is hypercover}
Let $\Ccal$ be an exact category and let $X_\bullet$ be a Kan semisimplicial object in $\Ccal$. Assume that $X_\bullet$ admits a colimit $X$. Then $X_\bullet$ is a hypercover of $X$.
\end{proposition}
\begin{proof}
Since the functor $e: \Ccal \rightarrow \widehat{\Ccal}$ preserves finite limits and creates effective epimorphisms, it suffices to show that the Kan semisimplicial object $e(X_\bullet)$ is a hypercover of $e(X)$. Since all truncations $\Ccal \rightarrow \Ccal_{\leq k}$ are left adjoints we have that $e$ preserves all colimits that exist in $\Ccal$, and in particular $e(X)$ is the colimit of $e(X_\bullet)$. Recall from the proof of proposition \ref{prop postnikov} that the inclusion $\widehat{\ccal} \rightarrow \shchat$ creates finite limits and geometric realizations of Kan semisimplicial objects. It follows that $e(X)$ is the colimit of the Kan semisimplicial object $e(X_\bullet)$ in $\shchat$. The result now follows from an application of theorem \ref{theorem relate kan and hypercover topos}.
\end{proof}

\begin{proposition}\label{prop hypercover infty connective}
Let $\Ccal$ be an exact category and let $X_\bullet$ be a hypercover of an object $Y$ in $\Ccal$. If $X_\bullet$ admits a colimit $X$ then the induced map $X \rightarrow Y$ is $\infty$-connective.
\end{proposition}
\begin{proof}
We apply proposition \ref{prop colimit of Kan is hypercover} to deduce that $X_\bullet$ is a hypercover of both $X$ and $Y$. This remains true after composing with the inclusion $\Ccal \rightarrow \Sh(\Ccal)$. Inside $\Sh(\Ccal)$ we have a commutative triangle
\[
\begin{tikzcd}
{|X_\bullet|} \arrow{r}{} \arrow{dr}{} & X \arrow{d}{} \\ {}  & Y
\end{tikzcd}
\]
where the horizontal and diagonal arrows are $\infty$-connective by virtue of \cite{SAG} corollary A.5.3.3. It follows that the right vertical arrow is $\infty$-connective, as desired.
\end{proof}

\begin{corollary}\label{coro equivalence hypercomplete}
Let $\ccal$ be an exact category. Then $\ccal$ is hypercomplete if and only if the following conditions are satisfied:
\begin{enumerate}[\normalfont (a)]
\item $\ccal$ admits geometric realizations  of Kan semisimplicial objects.
\item Every $\infty$-connective morphism in $\ccal$ is invertible.
\end{enumerate}
\end{corollary}
\begin{proof}
The if direction  follows from propositions \ref{prop colimit of Kan is hypercover} and \ref{prop hypercover infty connective}. To prove the only if direction we only need to show that if $f: X \rightarrow Y$ is an $\infty$-connective morphism in $\ccal$ then $f$ is invertible. Let $X_\bullet$ be the constant semisimplicial object of $\ccal$ with value $X$. Then $X_\bullet$ is a hypercover of $X$, and since $f$ is $\infty$-connective we have that $X_\bullet$ is also a hypercover of $Y$. The fact that $f$ is invertible now follows from the fact that $\ccal$ is hypercomplete.
\end{proof}
  
%%%%%%%%%%%%%%%%%%%%%%%%%%%%%%%%%%%%%%%%%%%%%%%%%%%%%%%%%%%%%%%%%%%%%%%%
%%%%%%%%%%%%%%%%%%%%%%%%%%%%%%%%%%%%%%%%%%%%%%%%%%%%%%%%%%%%%%%%%%%%%%%%
%%%%%%%%%%%%%%%%%%%%%%%%%%%%%%%%%%%%%%%%%%%%%%%%%%%%%%%%%%%%%%%%%%%%%%%%
%%%%%%%%%%%%%%%%%%%%%%%%%%%%%%%%%%%%%%%%%%%%%%%%%%%%%%%%%%%%%%%%%%%%%%%%
%%%%%%%%%%%%%%%%%%%%%%%%%%%%%%%%%%%%%%%%%%%%%%%%%%%%%%%%%%%%%%%%%%%%%%%%
%%%%%%%%%%%%%%%%%%%%%%%%%%%%%%%%%%%%%%%%%%%%%%%%%%%%%%%%%%%%%%%%%%%%%%%%
  
\subsection{Hypercompletion}\label{subsection hypercompletion}

We now study the procedure of hypercompletion of regular categories.

\begin{proposition}\label{prop univ prop hypercompl}
Let $\Ccal$ be a regular category and let $j$ be the composite functor
\[
\Ccal \hookrightarrow \Sh(\Ccal)^\reg \rightarrow (\Sh(\Ccal)^\hyp)^\reg.
\] 
\begin{enumerate}[\normalfont (1)]
\item $(\Sh(\Ccal)^\hyp)^\reg$ is a small hypercomplete exact category.
\item The functor $j$ is regular.
\item Let $\Dcal$ be a hypercomplete exact category. Then precomposition with $j$ induces an equivalence $\Fun^{\reg}((\Sh(\Ccal)^\hyp)^\reg , \Dcal) = \Fun^\reg(\ccal, \dcal)$.
\end{enumerate}
\end{proposition}
\begin{proof}
Part (1) follows from theorem \ref{theorem regular hypercomplete topoi}. Part (2) follows from the fact that the projection $\Sh(\ccal) \rightarrow \Sh(\ccal)^\hyp$ is regular, since the inclusions $\ccal \rightarrow \Sh(\ccal)$, $\Sh(\ccal)^\reg \rightarrow \Sh(\ccal)$ and $(\Sh(\Ccal)^\hyp)^\reg \rightarrow \Sh(\ccal)^\hyp$ create finite limits and effective epimorphisms. 

 We now prove (3). Composition with the inclusion $\Dcal \rightarrow \Sh(\Dcal)^\hyp$ induces a commutative square of categories
\[
\begin{tikzcd}
\Fun^\reg((\Sh(\Ccal)^\hyp)^\reg, \Dcal) \arrow{r}{} \arrow{d}{} & \Fun^\reg(\Ccal, \Dcal) \arrow{d}{} \\
\Fun^\reg((\Sh(\Ccal)^\hyp)^\reg, \Sh(\Dcal)^\hyp) \arrow{r}{} & \Fun^\reg(\Ccal, \Sh(\Dcal)^\hyp).
\end{tikzcd}
\]
We claim that this is a pullback square of categories. Since the vertical arrows are fully faithful, it suffices to show that if a regular functor $F: (\Sh(\Ccal)^\hyp)^\reg \rightarrow \Sh(\Dcal)^\hyp$ satisfies that $F\circ j$ factors through $\Dcal$, then $F$ factors through $\Dcal$. Let $X$ be an object in $(\Sh(\Ccal)^\hyp)^\reg$. By proposition \ref{prop regular in sheaf topos} there exists a hypercover of $X$ by a Kan semisimplicial object $X_\bullet$ in $\Ccal$. Then $F(X)$ has a hypercover by objects in $\Dcal$, namely $F(j(X_\bullet))$. It follows that $F(X)$ is regular and hypercomplete, and hence it belongs to $\Dcal$ by theorem \ref{theorem regular hypercomplete topoi}.

We may now reduce to showing that the bottom horizontal arrow is an equivalence. Applying remark \ref{remark univ prop sh}, we may reduce to showing that precomposition with the inclusion $(\Sh(\Ccal)^\hyp)^\reg \rightarrow \Sh(\Ccal)^\hyp$ induces an equivalence between the category of regular functors $(\Sh(\Ccal)^\hyp)^\reg \rightarrow \Sh(\Dcal)^\hyp$  and the category of geometric morphisms $\Sh(\Ccal)^\hyp \rightarrow \Sh(\Dcal)^\hyp$. This follows by another application of remark \ref{remark univ prop sh}, in light of theorem \ref{theorem regular hypercomplete topoi}.
\end{proof}

\begin{corollary}
The inclusion $\Cat_\hyp \rightarrow \Cat_\reg$ admits a left adjoint that sends each regular category $\ccal$ to $(\Sh(\ccal)^\hyp)^\reg$. 
\end{corollary}

The category $(\Sh(\Ccal)^\hyp)^\reg$ appearing in proposition \ref{prop univ prop hypercompl} is called the hypercompletion of $\ccal$.  Our next goal is to show that this terminology is compatible with the existing terminology for topoi: in other words, if $\ccal$ is a topos then the hypercompletion of $ \ccal$ defined (after a change in universe) in this  sense  coincides with the localization of $\ccal$ at the class of $\infty$-connective morphisms. In fact, we will show that this holds more generally whenever $\ccal$ is a presentable exact category.

\begin{proposition}\label{prop properties infty connective}
Let $\Ccal$ be a presentable exact category.
\begin{enumerate}[\normalfont (1)]
\item For every $n \geq -1$ the class of $n$-connective morphisms in $\ccal$ is saturated, of small generation, and stable under base change.
\item The class of $\infty$-connective morphisms  in $\ccal$ is strongly saturated, of small generation, and stable under base change.
\end{enumerate}   
\end{proposition}
\begin{proof}
We first prove (1). The fact that  the class of $n$-connective morphisms is saturated and stable under base change follows from proposition \ref{proposition 0 truncated are exact}.  It remains to prove small generation. Observe that the full subcategory of $\Ccal^{[1]}$ on the $(n-1)$-truncated morphisms is reflexive. The corresponding localization functor $L_n$ sends each arrow $f: X \rightarrow Y$  to the $(n-1)$-truncated arrow featuring in the $n$-connective-$(n-1)$-truncated factorization of $f$. Since $\Ccal$ is accessible, it has accessible finite limits, and hence $L_n$ is in fact an accessible localization. It follows that the preimage under $L_n$ of the invertible arrows is accessible. This consists of the full subcategory of $\ccal^{[1]}$ on the $n$-connective arrows.  It follows that the class of $n$-connective arrows is of small generation, as desired.

It remains to prove (2). The fact that the class of $\infty$-connective morphisms is saturated, of small generation, and stable under base change follows from (1), since these properties are stable under intersections. The strong saturation follows from the fact that if $f: X \rightarrow Y$ and $g: Y \rightarrow Z$ are two composable arrows such that $g \circ f$ and $g$ are $k$-connective, then $f$ is $(k-1)$-connective.
\end{proof}

\begin{definition}\label{definition hypercompletion presentable}
Let $\Ccal$ be a presentable exact category. We say that an object $X$ in $\Ccal$ is hypercomplete if for every  $\infty$-connective morphism $f: Y \rightarrow Z$ in $\ccal$ the map $\Hom_\ccal(Z, X) \rightarrow \Hom_\ccal(Y, X)$ of precomposition  with $f$ is an isomorphism. We denote by $\Ccal^\hyp$ the full subcategory of $\Ccal$ on the hypercomplete objects. 
\end{definition}

\begin{proposition}\label{prop hypercompletion when presentable exact}
Let $\Ccal$ be a presentable exact category. Then:
\begin{enumerate}[\normalfont (1)]
\item The inclusion $\Ccal^\hyp \rightarrow \Ccal$ admits a left adjoint $L$ that makes $\Ccal^\hyp$ into a left exact accessible localization of $\Ccal$. In particular, $\Ccal^\hyp$ is presentable and exact.
\item For every $k \geq -2$ the induced functor $L_{\leq k}: \Ccal_{\leq k} \rightarrow \Ccal^\hyp_{\leq k}$ is an equivalence.
\item Let $n \geq -1$. Then a morphism $f: X \rightarrow Y$ in $\Ccal$ is $n$-connective if and only if $L(f)$ is $n$-connective.
\end{enumerate}
\end{proposition}
\begin{proof}
Part (1) is a direct consequence of proposition \ref{prop properties infty connective}. Part (2) is equivalent to the assertion that the inclusion $\Ccal^\hyp_{\leq k} \rightarrow \Ccal_{\leq k}$ is an equivalence, which follows from the fact that $k$-truncated objects are hypercomplete. 

We finish by proving part (3). The case $n = -1$ is clear. The general case may be reduced to the case $n = 0$ by induction. Consider now the case $n = 0$. The only if direction is a direct consequence of the fact that $L$ is left exact and colimit preserving. To prove the converse, we have to show that if a morphism $f: X \rightarrow Y$ is such that $L(f)$ is an effective epimorphism in $\Ccal^\hyp$, then $f$ is an effective epimorphism in $\Ccal$. Observe that $f$ and $L(f)$ fit into a commutative square
\[
\begin{tikzcd}
X \arrow{d}{f} \arrow{r}{\eta_X} & L(X) \arrow{d}{L(f)}\\
Y \arrow{r}{\eta_Y} & L(Y)
\end{tikzcd}
\] 
where the horizontal arrows are $\infty$-connective. It thus suffices to show that $L(f)$ is an effective epimorphism in $\Ccal$. Let $U$ be the image of $L(f)$, computed in $\Ccal$. Then the map $U \rightarrow L(Y)$ is $(-1)$-truncated, and in particular it is right orthogonal to all $\infty$-connective morphisms. Since $L(Y)$ is hypercomplete we have that $U$ is also hypercomplete. The fact that $L(f)$ is an effective epimorphism in $\Ccal^\hyp$ now implies that $U = L(Y)$, and hence $L(f)$ is an effective epimorphism in $\Ccal$, as desired. 
\end{proof}

\begin{proposition}\label{prop univer prop hyp presentable}
Let $\ccal$ be a presentable exact category. Then:
\begin{enumerate}[\normalfont (1)]
\item The exact category $\ccal^\hyp$ is hypercomplete.
\item  Let $\dcal$ be a (possibly large) hypercomplete exact category. Then  precomposition with the localization $L: \ccal \rightarrow \ccal^\hyp$ induces an equivalence $\Fun^\reg(\ccal^\hyp, \dcal) = \Fun^\reg(\ccal, \dcal)$.
\end{enumerate}
\end{proposition}
\begin{proof}
Part (1) follows from corollary \ref{coro equivalence hypercomplete}, using parts (1) and (3) from proposition \ref{prop hypercompletion when presentable exact}. We now prove (2). Another application of parts (1) and (3) from proposition \ref{prop hypercompletion when presentable exact} implies that a functor $\ccal^\hyp \rightarrow \dcal$ is regular if and only if its restriction along $L$ is regular. We may thus reduce to showing that every regular functor $\ccal \rightarrow \dcal$ factors through $L$. This follows from the fact that regular functors preserve $\infty$-connective morphisms, together with corollary \ref{coro equivalence hypercomplete}.
\end{proof}

In light of propositions  \ref{prop univ prop hypercompl} and \ref{prop univer prop hyp presentable}, the assignment $\ccal \mapsto \ccal^\hyp$ from definition \ref{definition hypercompletion presentable} admits the following extension to arbitrary regular categories:

\begin{definition} \label{definition hypercompletion}
Let $\Ccal$ be a regular category. We denote by $\Ccal^\hyp$ the category $(\Sh(\Ccal)^\hyp)^\reg$. We call $\ccal^\hyp$ the hypercompletion of $\ccal$.
\end{definition}

\begin{remark}
In definition \ref{definition hypercompletion} we are implicitly assuming $\ccal$ to be small (with the understanding that this definition applies to arbitrary categories after a change in universe, see \ref{subsection size management}). Since $\ccal^\hyp$ is defined in terms of the topos of small sheaves on $\ccal$, its definition depends a priori on the boundary between small and large. It turns out however that $\ccal^\hyp$ is independent of this, since the category of small hypercomplete sheaves on $\ccal$ is closed under small colimits and limits inside the category of large hypercomplete sheaves.
\end{remark}

While the functor $\ccal \rightarrow \ccal^\hyp$ is a localization whenever $\ccal$ is presentable and exact, that is not the case in general. Our next goal is to study its behavior in the case when $\ccal$ is an exact $(n,1)$-category.

\begin{proposition}\label{proposition hyp of bounded}
Let $\Ccal$ be a bounded regular category.
\begin{enumerate}[\normalfont (1)]
\item The functor $j: \Ccal \rightarrow \Ccal^\hyp$ is fully faithful.
\item Assume that $\ccal$ is exact. Then $j$ induces an equivalence $\ccal = (\ccal^\hyp)_{< \infty}$.
\item Let $n \geq 0$ and assume that $\ccal$ is an exact $(n,1)$-category. Then $j$ induces an equivalence $\ccal = (\ccal^\hyp)_{\leq n-1}$.
\end{enumerate}
\end{proposition}
\begin{proof}
Since every object of $\Ccal$ is truncated, the inclusion $\Ccal \rightarrow \Sh(\Ccal)$ factors through the full subcategory of $\Sh(\ccal)$ on the truncated objects. Every truncated object is hypercomplete, so this inclusion in fact factors through $\ccal^\hyp$, and (1) follows. To prove (2) we need to show that if $\ccal$ is exact then every truncated object of $\ccal^\hyp$ is in the image of $j$. This is a consequence of proposition \ref{prop regular in sheaf topos}. To prove (3) we need to show that if $\ccal$ is an exact $(n,1)$-category then every $(n-1)$-truncated object of  $\ccal^\hyp$ is in the image of $j$. Applying proposition \ref{prop regular in sheaf topos} we may reduce to showing that every $(n-1)$-truncated object of $\ccal^\ex$ belongs to $\ccal$, which is a consequence of proposition \ref{prop ex completion of ex n cat}.
\end{proof}

\begin{definition}
Let $\Ccal$ be a regular category. We say that $\Ccal$ is eventually complicial if for every object $X$ there exists an effective epimorphism $X' \rightarrow X$ where $X'$ is truncated.
\end{definition}

\begin{proposition}\label{proposition eventually complicial is hyp}
\hfill
\begin{enumerate}[\normalfont (1)]
\item Let $\ccal$ be a bounded regular category. Then $\ccal^\hyp$ is eventually complicial.
\item Let $\dcal$ be an eventually complicial hypercomplete exact category. Then the inclusion $\iota: \dcal_{<\infty} \rightarrow \dcal$ induces an equivalence  $(\dcal_{< \infty})^\hyp = \dcal$.
\item Let $n \geq 0$ and let $\ccal$ be a regular $(n,1)$-category. Then $\ccal^\hyp$ is $(n-1)$-complicial.
\item Let $n \geq 0$ and let $\dcal$ be an $(n-1)$-complicial hypercomplete exact category. Then the inclusion $\iota_{n-1}: \dcal_{\leq n-1} \rightarrow \dcal$ induces an equivalence  $(\dcal_{\leq n-1})^\hyp = \dcal$.
\end{enumerate}
\end{proposition}

\begin{proof}
If $\ccal$ is a regular category, then every object of $\ccal^\hyp$ is regular in $\Sh(\ccal)$ and in particular it receives an effective epimorphism from an object in $\ccal$. Parts (1) and (3) follow  directly from this and the definitions. We now prove (2). It suffices to show that the functor $\iota^*: \Sh(\dcal_{<\infty})^\hyp \rightarrow \Sh(\dcal)^\hyp$ induced by $\iota$ is an equivalence. Since $\dcal$ is eventually complicial we have that $\dcal_{<\infty}$ is a generating subcategory of $\Sh(\dcal)^\hyp$, and in particular the same is true about the image of $\iota^*$. We may thus reduce to showing that $\iota^*$ is fully faithful. This follows from lemma \ref{lemma fully faithfulness ex and hyp}.

We now prove (4). Since $\dcal$ is $k$-complicial, it is eventually complicial, so part (2) applies. An application of proposition \ref{prop image is complicial} now shows
\[
(\dcal_{\leq k})^\hyp = ((\dcal_{\leq k})^\ex)^\hyp = (\dcal_{<\infty})^\hyp =\dcal
\]
as desired.
\end{proof}

\begin{corollary}\label{corollary fully faithful hypercomplete on n1}
\hfill
\begin{enumerate}[\normalfont (1)]
\item The hypercompletion function $\Cat^{\bounded}_\ex \rightarrow \Cat_\hyp$ is fully faithful, and its image consists of the eventually complicial hypercomplete exact categories.
\item Let $n \geq 0$. Then the hypercompletion functor $(n,1)\kr\Cat_\ex \rightarrow \Cat_\hyp$ is fully faithful, and its image consists of the $(n-1)$-complicial hypercomplete exact categories.
\end{enumerate}
\end{corollary}
\begin{proof}
We give a proof of (1), with (2) being analogous. Since regular functors map truncated objects to truncated objects, the hypercompletion functor $\Cat^{\normalfont \text{b}}_\ex \rightarrow \Cat_{\hyp}$ admits a right adjoint that sends each hypercomplete exact category $\dcal$ to $\dcal_{< \infty}$. Proposition \ref{proposition hyp of bounded} then shows that the unit of the adjunction is an isomorphism. The description of the image of hypercompletion is given by proposition \ref{proposition eventually complicial is hyp}.
\end{proof}

We finish this section with some examples of hypercompletions.

\begin{example}\label{example hypercompletion sheaves}
Let $n \geq 1$ and let $\ccal$ be an $(n,1)$-category equipped with a Grothendieck topology. Then $\Sh(\ccal)^\hyp$ is an $(n-1)$-complicial hypercomplete exact category. Applying proposition \ref{proposition eventually complicial is hyp} we obtain an equivalence
\[
\Sh(\ccal)^\hyp = (\Sh(\ccal)_{\leq n-1})^\hyp.
\]
In other words, $\Sh(\ccal)^\hyp$ is the hypercompletion of the exact $(n,1)$-category of $(n-1)$-truncated sheaves on $\ccal$.
\end{example}

\begin{example}\label{example hypercompletion diagrams}
Let $n \geq 1$ and let $\ccal$ be an $(n,1)$-category. Then specializing example \ref{example hypercompletion sheaves} to the case of the trivial Grothendieck topology we obtain an equivalence
\[
\Fun(\ccal, \Spc) = \Fun(\ccal, \Spc_{\leq n-1})^\hyp.
\]
\end{example}

\begin{warning}
Examples \ref{example hypercompletion sheaves} and \ref{example hypercompletion diagrams} do not hold in the case $n = 0$. In fact, every poset with finite meets is a hypercomplete exact category, and in particular hypercompletion acts as the identity on exact $(0,1)$-categories.
\end{warning}

\begin{example}\label{example hypercompletion of finite product theory}
Let $\ccal$ be a category with finite products and let $\Fun^\times(\ccal, \Spc)$  be the full subcategory of $\Fun(\ccal, \Spc)$ on those functors that preserve finite products. The evaluation functors $\Fun^\times(\ccal, \Spc) \rightarrow \Spc$ are jointly conservative, left exact, and preserve geometric realizations. It follows from this that  $\Fun^\times(\ccal, \Spc) $ is exact, and furthermore a morphism in $\Fun^\times(\ccal, \Spc) $ is $\infty$-connective if and only if its image under all evaluation functors is $\infty$-connective. Since $\Spc$ is hypercomplete we deduce, by corollary \ref{coro equivalence hypercomplete}, that $\Fun^\times(\ccal, \Spc)$ is a hypercomplete exact category.

Assume now that $\ccal$ is an $(n,1)$-category for some $n \geq 0$. Then $\Fun^\times(\ccal,\Spc)$  is an $(n-1)$-complicial hypercomplete exact category.  Applying proposition \ref{proposition eventually complicial is hyp} we obtain an equivalence
\[
\Fun^\times(\ccal, \Spc) = (\Fun^\times(\ccal, \Spc)_{\leq n-1})^\hyp = \Fun^\times(\ccal, \Spc_{\leq n-1})^\hyp.
\]
\end{example}

\begin{example}
Let $\ccal$ be a regular category and let $X$ be an object of $\ccal$. Then the forgetful functor $(\ccal^\hyp)_{/X} \rightarrow \Ccal^\hyp$ creates pullbacks and geometric realizations of groupoid objects, so we have that $(\ccal^\hyp)_{/X}$ is an exact category. The forgetful functor also creates effective epimorphisms and geometric realizations of Kan semisimplicial objects, so we have that $(\ccal^\hyp)_{/X}$ is hypercomplete.

  Assume now that $\ccal$ is an exact $(n,1)$-category for some $n \geq 0$. Then $\ccal^\hyp$ is $(n-1)$-complicial, and hence the same thing holds for $(\ccal^\hyp)_{/X}$. Applying propositions \ref{proposition hyp of bounded} and  \ref{proposition eventually complicial is hyp} we obtain an equivalence
  \[
  (\ccal^\hyp)_{/X} = (((\ccal^\hyp)_{/X})_{\leq n-1})^\hyp = (((\ccal^\hyp)_{\leq n-1})_{/X})^\hyp = (\ccal_{/X})^\hyp.
  \]
\end{example}

\ifx\inmain\undefined
\bibliographystyle{myamsalpha2}
\bibliography{References}
\fi

%%%%%%%%%%%%%%%%%%%%%%%%%%%%%%%%%%%%%%%%%%%%%%%%%%%%%%%%%%%%%%%%%%%%%%%%
%%%%%%%%%%%%%%%%%%%%%%%%%%%%%%%%%%%%%%%%%%%%%%%%%%%%%%%%%%%%%%%%%%%%%%%%
%%%%%%%%%%%%%%%%%%%%%%%%%%%%%%%%%%%%%%%%%%%%%%%%%%%%%%%%%%%%%%%%%%%%%%%%
%%%%%%%%%%%%%%%%%%%%%%%%%%%%%%%%%%%%%%%%%%%%%%%%%%%%%%%%%%%%%%%%%%%%%%%%
%%%%%%%%%%%%%%%%%%%%%%%%%%%%%%%%%%%%%%%%%%%%%%%%%%%%%%%%%%%%%%%%%%%%%%%%
%%%%%%%%%%%%%%%%%%%%%%%%%%%%%%%%%%%%%%%%%%%%%%%%%%%%%%%%%%%%%%%%%%%%%%%%

\section{Barr's embedding theorem}

Let $\ccal$ be a regular $(1,1)$-category. Then the bifunctor of evaluation $\ccal \times \Fun^\reg(\ccal, \Set) \rightarrow \Set$ induces a regular functor
\[
\ccal \rightarrow \Fun(\Fun^\reg(\ccal, \Set), \Set).
\]
One form of the Barr embedding theorem states that the above functor is fully faithful. A theorem of Makkai asserts that in the case when $\ccal$ is $1$-exact the image of this embedding consists of those functors  which preserve small products and filtered colimits.

 The goal of this section is to extend these results to the setting of higher category theory. Our main theorem states that for any bounded regular category $\ccal$ the canonical functor 
 \[
 \ccal \rightarrow \Fun(\Fun^\reg(\ccal, \Spc), \Spc)
\]
is fully faithful,  and provides a concrete description of its image.

We begin in \ref{subsection statement} by giving a precise statement of our result, and studying some consequences. We show here how one recovers the classical results of Barr and Makkai, as well as variants that apply to regular $(n,1)$-categories as well as Postnikov complete exact categories.

We give the proof of our embedding result in \ref{subsection proof embedding}. Our argument makes essential use of the category $\Pro(\ccal)$ of pro-objects in $\ccal$, whose properties we study in \ref{subsection weak topology}. As we shall see, the regular topology on $\ccal$ extends to a Grothendieck topology on $\Pro(\ccal)$, which we use to construct a hypercomplete exact category $\Pro(\ccal)^+$ containing $\Pro(\ccal)$. The relevance of $\Pro(\ccal)^+$ for our purposes is given by the fact that $\Fun^\reg(\ccal, \Spc)$ may be identified with the opposite of the full subcategory of $\Pro(\ccal)^+$ on the projective objects.

%%%%%%%%%%%%%%%%%%%%%%%%%%%%%%%%%%%%%%%%%%%%%%%%%%%%%%%%%%%%%%%%%%%%%%%%
%%%%%%%%%%%%%%%%%%%%%%%%%%%%%%%%%%%%%%%%%%%%%%%%%%%%%%%%%%%%%%%%%%%%%%%%
%%%%%%%%%%%%%%%%%%%%%%%%%%%%%%%%%%%%%%%%%%%%%%%%%%%%%%%%%%%%%%%%%%%%%%%%
%%%%%%%%%%%%%%%%%%%%%%%%%%%%%%%%%%%%%%%%%%%%%%%%%%%%%%%%%%%%%%%%%%%%%%%%
%%%%%%%%%%%%%%%%%%%%%%%%%%%%%%%%%%%%%%%%%%%%%%%%%%%%%%%%%%%%%%%%%%%%%%%%
%%%%%%%%%%%%%%%%%%%%%%%%%%%%%%%%%%%%%%%%%%%%%%%%%%%%%%%%%%%%%%%%%%%%%%%%

 \subsection{Statement of the theorem} \label{subsection statement}

We begin by giving a formulation of the embedding theorem.

\begin{notation}\label{notation upsilon}
Let $\ccal$ be a regular category. We denote by 
\[
\Upsilon_\ccal: \ccal \rightarrow \Fun(\Fun^\reg(\ccal, \Spc), \Spc)
\] the functor induced from the evaluation functor $\ccal \times \Fun^\reg(\ccal, \Spc) \rightarrow \Spc$.
\end{notation}

\begin{remark}\label{remark upsilon regular}
Let $\ccal$ be a regular category. Then the functor $\Upsilon_\ccal$ from notation \ref{notation upsilon} is regular.
\end{remark}

\begin{remark}
Let $f: \ccal \rightarrow \dcal$ be a regular functor between regular categories. Then we have commutative square of categories
\[
\begin{tikzcd}
\ccal \arrow{r}{\Upsilon_\ccal} \arrow{d}{f} & \Fun(\Fun^\reg(\ccal, \Spc), \Spc)\arrow{d}{}\\
\dcal \arrow{r}{\Upsilon_\dcal} &  \Fun(\Fun^\reg(\dcal, \Spc), \Spc)
\end{tikzcd}
\]
where the right vertical arrow is given by restriction along the functor $\Fun^\reg(\dcal, \Spc) \rightarrow \Fun^\reg(\ccal, \Spc)$ of restriction along $f$.
\end{remark}

\begin{theorem}\label{theorem barr embedding}
 Let $\ccal$ be a regular category.
\begin{enumerate}[\normalfont (1)]
\item The category $\Fun^\reg(\ccal, \Spc)$ is accessible and admits small products and filtered colimits.
\item Assume that $\ccal$ is bounded. Then the functor $\Upsilon_\ccal$ from notation \ref{notation upsilon} is fully faithful.
\item Assume that $\ccal$ is exact and bounded. Then a functor $F: \Fun^\reg(\ccal, \Spc) \rightarrow  \Spc$ belongs to the image of $\Upsilon_\ccal$ if and only if it preserves small products and filtered colimits, and factors through $\Spc_{\leq k}$ for some $k \geq 0$.
\end{enumerate}  
 \end{theorem}

The proof of theorem \ref{theorem barr embedding} will be given in \ref{subsection proof embedding}. We now discuss some consequences.
 
  \begin{corollary}
 Let $n \geq 0$ and let $\ccal$ be a regular $(n,1)$-category. 
\begin{enumerate}[\normalfont (1)]
\item The category $\Fun^\reg(\ccal, \Spc_{\leq n-1})$ is accessible and admits small products and filtered colimits.
\item The functor $\ccal \rightarrow \Fun(\Fun^\reg(\ccal, \Spc_{\leq n-1}), \Spc_{\leq n-1})$ induced by evaluation is fully faithful.
\item Assume that $\ccal$ is an exact $(n,1)$-category. Then a functor $F: \Fun^\reg(\ccal, \Spc_{\leq n-1}) \rightarrow \Spc_{\leq n-1}$ belongs to the image of the embedding from (2) if and only if it preserves small products and filtered colimits.
\end{enumerate} 
  \end{corollary}
\begin{proof}
 Since $\ccal$ is an $(n,1)$-category, composition with the inclusion $\Spc_{\leq n-1} \rightarrow \Spc$ induces an equivalence $\Fun^\reg(\ccal, \Spc_{\leq n-1}) = \Fun^\reg(\ccal, \Spc)$. Item (1) now follows directly from theorem \ref{theorem barr embedding}. Let $\Upsilon'_\ccal$ be the composite functor
\[
\ccal \rightarrow \Fun(\Fun^\reg(\ccal, \Spc_{\leq n-1}), \Spc_{\leq n-1}) \rightarrow \Fun(\Fun^\reg(\ccal, \Spc_{\leq n-1}), \Spc),
\]
 where the second morphism is given by composition with the inclusion $\Spc_{\leq n-1} \rightarrow \Spc$. The corollary will follow if we show the following:
 \begin{enumerate}[\normalfont (1')]
   \setcounter{enumi}{1}
\item $\Upsilon'_\ccal$ is fully faithful.
\item Assume that $\ccal$ is an exact $(n,1)$-category. Then a functor $F: \Fun^\reg(\ccal, \Spc_{\leq n-1}) \rightarrow \Spc$ belongs to the image of $\Upsilon'_\ccal$ if and only if it preserves small products and filtered colimits and factors through $\Spc_{\leq n-1}$.
\end{enumerate} 
Observe that the composition
\[
\ccal \xrightarrow{\Upsilon'_\ccal}  \Fun(\Fun^\reg(\ccal, \Spc_{\leq n-1}), \Spc) = \Fun(\Fun^\reg(\ccal, \Spc), \Spc)
\]
is equivalent to $\Upsilon_\ccal$. Item (2') now follows directly from theorem \ref{theorem barr embedding}. We now prove (3'). Consider the commutative square
\[
\begin{tikzcd}
\ccal \arrow{d}{} \arrow{r}{\Upsilon_\ccal} & \Fun(\Fun^\reg(\ccal, \Spc), \Spc) \arrow{d}{} \\
\ccal^\ex \arrow{r}{\Upsilon_{\ccal^\ex}} & \Fun(\Fun^\reg(\ccal^\ex, \Spc), \Spc) 
\end{tikzcd}
\]
where the vertical arrows are induced by the inclusion $\iota: \ccal \rightarrow \ccal^\ex$. Theorem \ref{teo properties completion} implies that restriction along $\iota$ induces an equivalence $\Fun^\reg(\ccal^\ex, \Spc) = \Fun^\reg(\ccal, \Spc)$. By proposition \ref{prop image is complicial} we have that $\ccal^\ex$ is bounded, and hence we may apply   theorem \ref{theorem barr embedding} to deduce that $\Upsilon_{\ccal^\ex}$ is fully faithful and identifies $(\ccal^\ex)_{\leq n-1}$ with the category of functors $\Fun^\reg(\ccal^\ex, \Spc) \rightarrow \Spc$ which preserve small products and filtered colimits and factor through $\Spc_{\leq n-1}$. Item (3') now follows from proposition \ref{prop ex completion of ex n cat}.
\end{proof}

 \begin{corollary}\label{corollary embed into presheaves}
  Let $\ccal$ be a regular category.
  \begin{enumerate}[\normalfont (1)]
  \item Assume that $\ccal$ is bounded. Then there exists a (small) category $\Ical$ and a fully faithful regular functor $\ccal \rightarrow \Pcal(\Ical)$.
  \item Let $n \geq 0$ and assume that $\ccal$ is an $(n,1)$-category.  Then there exists a (small) $(n,1)$-category $\Jcal$ and a fully faithful regular functor $\ccal \rightarrow \Pcal(\Jcal)_{\leq n-1}$.
\end{enumerate}   
 \end{corollary}
 \begin{proof}
  Combining theorem \ref{theorem barr embedding} and remark \ref{remark upsilon regular} we see that $\Upsilon_\ccal$ is fully faithful, regular, and factors through the full subcategory of $\Fun(\Fun^\reg(\ccal, \Spc), \Spc)^\acc$ of $\Fun(\Fun^\reg(\ccal, \Spc), \Spc)$ on the accessible functors. Since $\ccal$ is small we may pick a regular cardinal $\kappa$ such that $\Fun^\reg(\ccal, \Spc)$ is $\kappa$-accessible and the image of $\Upsilon_\ccal$ consists of $\kappa$-accessible functors. Part (1) now follows by setting $\Ical$ to be the full subcategory of $\Fun^\reg(\ccal, \Spc)$ on the $\kappa$-compact objects, with the inclusion  $\ccal \rightarrow \Pcal(\Ical)$ being given as the composition of $\Upsilon_\ccal$ with the functor $\Fun(\Fun^\reg(\ccal, \Spc), \Spc) \rightarrow \Pcal(\Ical)$ of restriction along the inclusion $\Ical \rightarrow \Fun^\reg(\ccal, \Spc)$.
  
  We now prove (2). Since $\ccal$ is bounded part (1) applies, so we may choose a category $\Ical$ and a fully faithful regular functor $\ccal \rightarrow \Pcal(\Ical)$. Let $\Jcal$ be the full subcategory of $\Pcal(\Ical)_{\leq n-1}$ on those objects of the form $\tau_{\leq n-1}(X)$ for some $X$ in $\Ical$. Item (2) now follows from the fact that we have an equivalence $\Pcal(\Ical)_{\leq n-1} = \Pcal(\Jcal)_{\leq n-1}$.
 \end{proof}

 The functor $\Upsilon_\ccal$ is not fully faithful for a general regular category $\ccal$. In the case when $\ccal$ is exact, the fully faithfulness of $\Upsilon_\ccal$ implies that every object $X$ of $\ccal$ is the limit of its Postnikov tower $\tau_{\leq -1}(X) \leftarrow \tau_{\leq 0}(X) \leftarrow \tau_{\leq 1}(X) \leftarrow \ldots$. Our next corollary shows that this is in fact a sufficient condition for the fully faithfulness of $\Upsilon_\ccal$:
 
 \begin{corollary}\label{coro upsilon when postnikov are limits}
 Let $\ccal$ be an exact category. Assume that every object of $\ccal$ is the limit of its Postnikov tower.
\begin{enumerate}[\normalfont (1)]
\item The functor $\Upsilon_\ccal: \ccal \rightarrow \Fun(\Fun^\reg(\ccal, \Spc), \Spc)$ is fully faithful.
\item There exists a (small) category $\Ical$ and a fully faithful regular functor $\ccal \rightarrow \Pcal(\Ical)$.
\end{enumerate} 
 \end{corollary}
 \begin{proof}
We first prove (1). We have a commutative square of categories
 \[
 \begin{tikzcd}
 \ccal \arrow{r}{\Upsilon_\ccal} & \Fun(\Fun^\reg(\ccal, \Spc), \Spc) \\
 \ccal_{< \infty} \arrow{u}{} \arrow{r}{\Upsilon_{\ccal_{<\infty}}} & \arrow{u}{} \Fun(\Fun^\reg(\ccal_{<\infty}, \Spc), \Spc).
 \end{tikzcd}
 \]
 The bottom horizontal arrow is fully faithful thanks to theorem \ref{theorem barr embedding}. Since $\Spc$ is Postnikov complete we have equivalences 
 \[
 \Fun^\reg(\ccal, \Spc) = \Fun^\reg(\widehat{\ccal}, \Spc) = \Fun^\reg(\smash{\widehat{(\ccal_{<\infty})}}, \Spc) = \Fun^\reg(\ccal_{<\infty},\Spc).
 \]
 It follows that the right vertical arrow in the above commutative square is an equivalence, and consequently $\Upsilon_\ccal$ is fully faithful when restricted to $\ccal_{<\infty}$. Assume now given two arbitrary objects $X, Y$ of $\ccal$. Using the fact that $\Upsilon_\ccal$ is regular we have
 \begin{align*}
 \Hom_{ \Fun(\Fun^\reg(\ccal, \Spc), \Spc) }(\Upsilon_\ccal(X), \Upsilon_\ccal(Y) )   &= \lim   \Hom_{ \Fun(\Fun^\reg(\ccal, \Spc), \Spc) }(\tau_{\leq k}(\Upsilon_\ccal(X)), \tau_{\leq k}(\Upsilon_\ccal(Y)))\\ &= \lim   \Hom_{ \Fun(\Fun^\reg(\ccal, \Spc), \Spc) }(\Upsilon_\ccal(\tau_{\leq k}(X)), \Upsilon_\ccal(\tau_{\leq k}(Y)))\\&= \lim \Hom_{\ccal}(\tau_{\leq k}(X), \tau_{\leq k}(Y) ). 
 \end{align*}
 The above agrees with $\Hom_{\ccal}(X, Y)$ since $Y$ is the limit of its Postnikov tower.
 
 We now prove (2). By theorem \ref{theorem barr embedding}, the functor $\Upsilon_{\ccal_{<\infty}}$ factors through the full subcategory of $ \Fun(\Fun^\reg(\ccal_{<\infty}, \Spc), \Spc)$ on the accessible functors. It follows that $\Upsilon_{\ccal}$ maps truncated objects to accessible functors. Assume now given an arbitrary object $X$ in $\ccal$. We have 
 \[
 \Upsilon_{\ccal}(X) = \lim \Upsilon_{\ccal}(X)_{\leq k} = \lim \Upsilon_\ccal(X_{\leq k}).
 \]
 Since accessible functors are closed under small limits we deduce that $\Upsilon_{\ccal}$ factors through the full subcategory of $\Fun(\Fun^\reg(\ccal, \Spc), \Spc)$ on the accessible functors. Let $\kappa$ be a regular cardinal such that $\Fun^\reg(\ccal, \Spc)$ is $\kappa$-accessible and the image of $\Upsilon_\ccal$ consists of $\kappa$-accessible functors. Then part (2) follows by letting $\Ical$ be the full subcategory of  $\Fun^\reg(\ccal, \Spc)$ on the $\kappa$-compact objects, with the inclusion $\ccal \rightarrow \Pcal(\Ical)$ being given as the composition of $\Upsilon_\ccal$ with the functor of restriction along the inclusion $\Ical \rightarrow \Fun^\reg(\ccal, \Spc)$.
 \end{proof}

 \begin{corollary}
 Let $\ccal$ be a Postnikov complete exact category. Then the functor $\Upsilon_\ccal: \ccal \rightarrow \Fun(\Fun^\reg(\ccal, \Spc), \Spc)$ is fully faithful, and a functor $F:\Fun^\reg(\ccal, \Spc) \rightarrow  \Spc$ belongs to the image of $\Upsilon_\ccal$ if and only if it preserves small products and filtered colimits.
 \end{corollary}
 \begin{proof}
 By corollary \ref{coro upsilon when postnikov are limits}, the functor $\Upsilon_\ccal$ is fully faithful. It remains to identify its image. We note that $\Fun^\reg(\ccal, \Spc)$ is closed under small products and filtered colimits inside $\Fun(\ccal, \Spc)$, and therefore  every functor in the image of $\Upsilon_\ccal$ preserves small products and filtered colimits. Assume now given a functor $F: \Fun^\reg(\ccal, \Spc)  \rightarrow \Spc$ with these properties.  Since $\ccal$ and $\Fun( \Fun^\reg(\ccal, \Spc), \Spc)$ are Postnikov complete and $\Upsilon_\ccal$ is fully faithful and regular, to prove that $F$ belongs to the image of $\Upsilon_\ccal$ it suffices to show that $\tau_{\leq k}(F)$ belongs to the image of $\Upsilon_\ccal$ for every $k\geq 0$. Replacing $F$ with $\tau_{\leq k}(F)$ we may reduce to the case when $F$ is truncated. Consider now the commutative square of categories
\[
 \begin{tikzcd}
 \ccal \arrow{r}{\Upsilon_\ccal} & \Fun(\Fun^\reg(\ccal, \Spc), \Spc) \\
 \ccal_{< \infty} \arrow{u}{} \arrow{r}{\Upsilon_{\ccal_{<\infty}}} & \arrow{u}{} \Fun(\Fun^\reg(\ccal_{<\infty}, \Spc), \Spc).
 \end{tikzcd}
 \]
Since $\ccal$ is Postnikov complete, the functor of restriction $\Fun^\reg(\ccal, \Spc) \rightarrow \Fun^\reg(\ccal_{<\infty}, \Spc)$ is an equivalence, and hence the right vertical arrow is an equivalence. The corollary now follows from an application of theorem \ref{theorem barr embedding} to $\ccal_{<\infty}$.
  \end{proof}
  
%%%%%%%%%%%%%%%%%%%%%%%%%%%%%%%%%%%%%%%%%%%%%%%%%%%%%%%%%%%%%%%%%%%%%%%%
%%%%%%%%%%%%%%%%%%%%%%%%%%%%%%%%%%%%%%%%%%%%%%%%%%%%%%%%%%%%%%%%%%%%%%%%
%%%%%%%%%%%%%%%%%%%%%%%%%%%%%%%%%%%%%%%%%%%%%%%%%%%%%%%%%%%%%%%%%%%%%%%%
%%%%%%%%%%%%%%%%%%%%%%%%%%%%%%%%%%%%%%%%%%%%%%%%%%%%%%%%%%%%%%%%%%%%%%%%
%%%%%%%%%%%%%%%%%%%%%%%%%%%%%%%%%%%%%%%%%%%%%%%%%%%%%%%%%%%%%%%%%%%%%%%%
%%%%%%%%%%%%%%%%%%%%%%%%%%%%%%%%%%%%%%%%%%%%%%%%%%%%%%%%%%%%%%%%%%%%%%%%

\subsection{The weak regular topology} \label{subsection weak topology}

A fundamental role in the proof of theorem \ref{theorem barr embedding} will be played by the category $\Pro(\ccal)$ of pro-objects on $\ccal$.  We now  make a systematic study of this category.

\begin{proposition}\label{prop commute tot and geometric realiz}
Let $\ccal$ be a category with finite limits. Assume that every object of $\ccal$ is truncated. Then geometric realizations of semisimplicial objects commute with cofiltered limits in $\Pro(\ccal)$.
\end{proposition}
\begin{proof}
Let $X_{\alpha, \bullet}$ be a cofiltered diagram of semisimplicial objects in $\Pro(\ccal)$ and let $X_{\bullet}$ be its limit. We wish to show that $|X_\bullet|$ is the limit of $|X_{\alpha, \bullet}|$. To do so it is enough to show that for every object $Y$ in $\ccal$ the canonical map $\colim \Hom_{\Pro(\ccal)}(|X_{\alpha, \bullet}|, Y) \rightarrow \Hom_{\Pro(\ccal)}( |X_{ \bullet}|, Y) $ is an isomorphism. We have 
\[
 \Hom_{\Pro(\ccal)}( |X_{ \bullet}|, Y) = \Tot  \Hom_{\Pro(\ccal)}( X_\bullet , Y) = \Tot \colim  \Hom_{\Pro(\ccal)}(  X_{\alpha, \bullet}, Y).
 \]
 Since $Y$ is truncated, there exists $k$ such that $\Hom_{\Pro(\ccal)}(  X_{\alpha, n}, Y)$ is $k$-truncated for all $\alpha$ and all $n$.  Using that filtered colimits and totalizations commute in $\Spc_{\leq k}$   we deduce
 \[
 \Hom_{\Pro(\ccal)}( |X_{ \bullet}|, Y)  = \colim   \Tot \Hom_{\Pro(\ccal)}(  X_{\alpha, \bullet}, Y) = \colim \Hom_{\Pro(\ccal)}( | X_{\alpha, \bullet}|, Y)
 \]
 as desired.
\end{proof}

 \begin{proposition}\label{prop pro is regular}
 Let $\ccal$ be a bounded regular category.
 \begin{enumerate}[\normalfont (1)]
 \item The category $\Pro(\ccal)$  is regular.
  \item The inclusion $\ccal \rightarrow \Pro(\ccal)$ is regular.
   \item The class of effective epimorphisms in $\Pro(\ccal)$ is the closure under cofiltered limits of the class of effective epimorphisms in $\ccal$.
 \end{enumerate}
  \end{proposition}
\begin{proof}
We first prove (1). Let $f: X \rightarrow Y$ and $g: Y' \rightarrow Y$ be morphisms in $\Pro(\ccal)$. Let $X_\bullet$ be the \v{C}ech nerve of $f$, and let $X'_\bullet$ be the \v{C}ech nerve of the base change of $f$ along $g$. We wish to show that $|X'_\bullet| = Y' \times_Y |X_\bullet|$. Write the cospan 
\[
g: Y' \rightarrow Y \leftarrow X: f
\]
 as a cofiltered limit of cospans 
\[
g_\alpha: Y'_\alpha \rightarrow Y_\alpha \leftarrow X_\alpha: f_\alpha
\]
 with $X_\alpha, Y_\alpha, Y'_\alpha$ in $\ccal$. For each $\alpha$ let $X_{\alpha,\bullet}$ be the \v{C}ech nerve of $f_\alpha$, and $X'_{\alpha, \bullet}$ be the \v{C}ech nerve of the base change of $f_\alpha$ along $g_\alpha$.  Since $\ccal$ is regular and the inclusion $\ccal \rightarrow \Pro(\ccal)$ preserves finite limits and geometric realizations we have that $|X'_{\alpha,\bullet}| = Y'_\alpha \times_{Y_\alpha} |X_{\alpha,\bullet}|$ for all $\alpha$. Applying proposition \ref{prop commute tot and geometric realiz} we deduce
\[
|X'_\bullet| = |\lim X'_{\alpha, \bullet}| = \lim |X'_{\alpha, \bullet}| = \lim ( Y'_\alpha \times_{Y_\alpha} |X_{\alpha,\bullet}| ) = Y' \times_Y \lim |X_{\alpha, \bullet}| = Y' \times_Y |\lim X_{\alpha, \bullet}|
\]
which agrees with $Y' \times_Y |X_\bullet|$, as desired.

Item (2) follows from the fact that the inclusion $\ccal \rightarrow \Pro(\ccal)$ preserves finite limits and geometric realizations.  We now prove (3). The assertion  that effective epimorphisms in $\Pro(\ccal)$ are closed under cofiltered limits is equivalent to the assertion that cofiltered limits in $\Pro(\ccal)$ are regular, which is a consequence of proposition \ref{prop commute tot and geometric realiz}. 

It remains to show that effective epimorphisms in $\Pro(\ccal)$ are generated under cofiltered limits by effective epimorphisms in $\ccal$.  
 Let $f: X \rightarrow Y$ be an effective epimorphism in $\Pro(\ccal)$.  Write $Y$ as the cofiltered limit of a diagram of objects $Y_\alpha$ in $\ccal$. Using the fact that cofiltered limits commute with finite colimits in $\Pro(\ccal)$ we may write $f$ as the cofiltered limit of its cobase changes $f_\alpha: X_\alpha \rightarrow Y_\alpha$. Since effective epimorphisms are the left class of a factorization system on $\Pro(\ccal)$, we have that $f_\alpha$ is an effective epimorphism for all $\alpha$.  Replacing $f$ with $f_\alpha$ we may now reduce to the case when $Y$ belongs to $\ccal$. In this case $f$ may be written as a cofiltered limit in $\Pro(\ccal)_{/Y}$ of a diagram of objects $f_\beta: X_\beta \rightarrow Y$ in $\ccal_{/Y}$. The effective epimorphism $f$ factors through $f_\beta$ for all $\beta$, so we have that $f_\beta$ is an effective epimorphism in $\ccal$ for all $\beta$, and our claim follows.
 \end{proof}
 
 \begin{definition}\label{definition transfinite precomposition}
 Let $\ccal$ be a large category with small inverse limits and let $S$ be a class of morphisms in $\ccal$. We say that a morphism $f: X \rightarrow Y$ is a transfinite precomposition of morphisms in $S$ if there exists an ordinal $\lambda$ and a functor $\lbrace \alpha: \alpha \leq \lambda \rbrace^\op \rightarrow \ccal$ sending each ordinal $\alpha \leq \lambda$ to an object $X_\alpha$ in $\ccal$, with the following properties:
 \begin{itemize}
 \item For each ordinal $\alpha < \lambda$ the map $X_{\alpha + 1} \rightarrow X_\alpha$ belongs to $S$.
 \item For each limit ordinal $\alpha \leq \lambda$ we have $X_\alpha = \lim_{\beta < \alpha} X_\beta$.
 \item The map $X_\lambda \rightarrow X_0$ is equivalent to $f$.
 \end{itemize} 
 \end{definition}
 
 \begin{definition}
 Let $\ccal$ be a regular category. We say that a morphism $f: X \rightarrow Y$ in $\Pro(\ccal)$ is a weak effective epimorphism if it is a transfinite precomposition of base changes of effective epimorphisms in $\ccal$. The weak regular topology on $\Pro(\ccal)$ is the Grothendieck topology where a sieve is a covering sieve if and only if it contains a weak effective epimorphism. 
 \end{definition}

\begin{proposition}\label{prop subcanonical}
Let $\ccal$ be a bounded regular category. Then every representable presheaf on $\Pro(\ccal)$ is a hypercomplete sheaf for the weak  regular topology.  
\end{proposition}
\begin{proof}
Since $\Pro(\ccal)$ is generated under limits by $\ccal$, and every object of $\ccal$ is truncated, it suffices to show that every object $Z$ in $\ccal$ is a sheaf for the weak regular topology. Proposition \ref{prop pro is regular} implies that every weak effective epimorphism in $\Pro(\ccal)$ is an effective epimorphism. The desired claim now follows from the fact that every object of $\Pro(\ccal)$ is a sheaf for the regular topology.
\end{proof}

\begin{notation}
Let $\ccal$ be a bounded regular category. We denote by $\Pro(\ccal)^+$ the full subcategory of the topos of hypercomplete sheaves of large spaces on $\Pro(\ccal)$ on the regular hypercomplete sheaves. 
\end{notation}

\begin{remark}
Let $\ccal$ be a bounded regular category. Then it follows from propositions \ref{prop singleton topology} and  \ref{prop subcanonical} that  $\Pro(\ccal)^+$ contains $\Pro(\ccal)$. Furthermore, an application of theorem \ref{theorem regular hypercomplete topoi} shows that  $\Pro(\ccal)^+$ is a hypercomplete exact category, and its inclusion into the topos of sheaves is regular. We note that every weak effective epimorphism in $\Pro(\ccal)$ is an effective epimorphism when regarded as a morphism in $\Pro(\ccal)^+$, and in particular the inclusion $\ccal \rightarrow \Pro(\ccal)^+$ is regular.
\end{remark}

Our next goal is to study a class of objects of $\Pro(\ccal)$ called weakly projective. As the name suggests, this is a weakening of the notion of projective object, which makes sense in any regular category.
 
 \begin{definition}\label{definition projective}
 Let $\ccal$ be a regular category. An object $P$ in $\ccal$ is said to be projective if the functor $\ccal \rightarrow \Spc$ corepresented by $P$ is regular. We say that $\ccal$ has enough projectives if for every object $X$ in $\ccal$ there exists a projective object $P$ and an effective epimorphism $P \rightarrow X$.
 \end{definition}
 
\begin{remark}
Let $\ccal$ be a regular category and let $P$ be an object in $\ccal$. Then the following are equivalent:
\begin{itemize}
\item $P$ is projective.
\item Let $X \rightarrow Y$ be an effective epimorphism in $\ccal^\ex$. Then every map $P \rightarrow Y$ lifts to $X$.
\item Every effective epimorphism $X \rightarrow P$ in $\ccal$ admits a section.
\end{itemize}
\end{remark}

\begin{definition}
 Let $\ccal$ be a regular category. An object $P$ in $\Pro(\ccal)$ is said to be weakly projective if the functor  $\ccal \rightarrow \Spc$ corepresented by $P$ is regular.
\end{definition}

\begin{notation}
Let $\ccal$ be a regular category. We denote by $\Pro(\ccal)_{\wproj}$ the full subcategory of $\Pro(\ccal)$ on the weakly projective objects.
\end{notation}

\begin{remark}\label{remark wproj closed under}
Let $\ccal$ be a regular category. Since effective epimorphisms in $\Spc$ are closed under small products and filtered colimits we have that $\Pro(\ccal)_{\wproj}$ is closed under small coproducts and cofiltered limits inside $\Pro(\ccal)$. 
\end{remark}

\begin{remark}
Let $\ccal$ be a bounded regular category. Then an object $P$ in $\Pro(\ccal)$ is weakly projective if and only if it is projective when regarded as an object in  $\Pro(\ccal)^+$ (in the sense of definition \ref{definition projective}). Assume now given a projective object $Q$ of $\Pro(\ccal)^+$. Since $Q$ is a regular sheaf on $\Pro(\ccal)$ there exists an effective epimorphism $X \rightarrow Q$ with $X$ in $\Pro(\ccal)$. The fact  that $Q$ is projective implies that $Q$ is a retract of $X$, and since $\Pro(\ccal)$ is idempotent complete we deduce that $Q$ belongs to $\Pro(\ccal)$. It follows that weakly projective objects of $\Pro(\ccal)$ are the same as projective objects of $\Pro(\ccal)^+$.
\end{remark}
 
While not every regular category admits enough projective objects, there are always enough weakly projective objects in $\Pro(\ccal)$, by virtue of the small object argument. In what follows we will in fact need the following more precise assertion:
 
\begin{proposition}\label{prop effective small object argument}
Let $\ccal$ be a regular category and let $\kappa$ be an uncountable regular cardinal such that $\ccal$ is $\kappa$-small. Let $X$ be a $\kappa$-cocompact object of $\Pro(\ccal)$. Then there exists a $\kappa$-cocompact weakly projective object $P$ in $\Pro(\ccal)$ and a weak effective epimorphism $P \rightarrow X$. 
 \end{proposition}
 \begin{proof}
 Fix a set $S = \lbrace f_\alpha: Y_\alpha \rightarrow Z_\alpha \rbrace$ of representatives for isomorphism classes of effective epimorphisms in $\ccal$ (so that each effective epimorphism is isomorphic to a unique $f_\alpha$). We construct an inverse system $ X_0 \leftarrow X_1 \leftarrow X_2 \leftarrow \ldots$ of objects of $\ccal$ by induction as follows:
 \begin{itemize}
 \item Set $X_0 = X$.
 \item Assume that $X_i$ has been constructed. For each $\alpha$ let $T_\alpha = \lbrace g^\alpha_\beta: X_i \rightarrow Z_\alpha\rbrace$ be a set of representatives for isomorphism classes of maps from $X_i$ to $Z_\alpha$. We define $X_{i+1}$ to be the object arising from the following fiber product:
 \[
 \begin{tikzcd}[column sep = huge]
 X_{i+1} \arrow{r}{} \arrow{d}{} &  X_i \arrow{d}{\prod_{(\alpha, \beta \in T_\alpha)} g^\alpha_\beta} \\
 \prod_{(\alpha, \beta \in T_\alpha)} Y_\alpha  \arrow{r}{\prod_{(\alpha, \beta \in T_\alpha)} f_\alpha} & \prod_{(\alpha, \beta \in T_\alpha)} Z_\alpha
\end{tikzcd} 
 \]
 \end{itemize}
 
 Let $P = \lim X_i$. The projection $P \rightarrow X$ is a transfinite precomposition of pullbacks of maps in $S$, and hence it is a weak effective epimorphism. Since the target of all maps in $S$ is cocompact we have that $P$ is weakly left orthogonal to all maps in $S$, and hence it is weakly projective. It remains to show that $P$ is $\kappa$-cocompact. Since $\kappa$ is uncountable it will suffice to show that $X_i$ is $\kappa$-cocompact for all $i \geq 0$. We argue by induction on $i$. The case $i = 0$ is part of our assumptions, so assume that $i > 0$ and that the assertion is known for $i-1$.   Since $\kappa$-cocompact objects of $\Pro(\ccal)$ are closed under $\kappa$-small limits, it will suffice to show that the set of pairs $(\alpha, \beta \in T_\alpha)$ arising in the definition of $X_i$ is $\kappa$-small. Using the fact that the cardinality of $S$ is less than $\kappa$, we may reduce to showing that for each $\alpha$ in $S$ the set $T_\alpha$ is $\kappa$-small. We claim that in fact the space $\Hom_{\Pro(\ccal)}(X_{i-1}, Z_\alpha)$ is $\kappa$-small for all $\alpha$. This follows from the fact that $X_{i-1}$ is $\kappa$-cocompact, since all Hom spaces in $\ccal$ are $\kappa$-small.
  \end{proof}

We finish with the following basic result on the structure of categories of weakly projective objects:

\begin{proposition}\label{prop accessibility of wproj}
Let $\ccal$ be a regular category. Then the category $(\Pro(\ccal)_{\wproj})^\op$ is accessible and admits small products and filtered colimits.
\end{proposition}
\begin{proof}
The fact that $(\Pro(\ccal)_{\wproj})^\op$ admits small products and filtered colimits follows from remark \ref{remark wproj closed under}. It remains to prove that it is accessible. Let $\kappa$ be an uncountable regular cardinal such that $\ccal$ is $\kappa$-small. We will show that $(\Pro(\ccal)_{\wproj})^\op$ is $\kappa$-accessible. Since $ \Pro(\ccal)_{\wproj} $ is closed under cofiltered limits inside $\Pro(\ccal)$, it will suffice to show that every weakly projective object $P$ in $\Pro(\ccal)$ may be written as the cofiltered limit of a diagram of $\kappa$-cocompact weakly projective objects.

Let $\dcal$ be the full subcategory of $\Pro(\ccal)_{P/}$ on the objects $P \rightarrow X$ such that $X$ is $\kappa$-cocompact, and let $\dcal_0$ be the full subcategory of $\dcal$ on those objects $P \rightarrow X$ such that $X$ is weakly projective. Then $P$ is the limit of the canonical functor $\dcal \rightarrow \Pro(\ccal)$. To finish the proof it will suffice to show that the inclusion $\dcal_0 \rightarrow \dcal$ is initial and that $\dcal_0$ is cofiltered.  

 Let $P \rightarrow X_j$ be a diagram of objects of $\dcal$ indexed by a finite category $\Jcal$. Applying proposition \ref{prop effective small object argument}   we deduce the existence of a weak effective epimorphism $Q\rightarrow \lim X_j$ where $Q$ is  a $\kappa$-cocompact weakly projective object of $\Pro(\ccal)$.  Since $P$ is weakly projective the canonical map $P \rightarrow \lim X_j$ lifts to a morphism $P \rightarrow Q$, which defines an object of $\dcal_0$. It follows that $\dcal$ and $\dcal_0$ are cofiltered categories and the inclusion $\dcal_0 \rightarrow \dcal$ is initial, as desired.
\end{proof}

%%%%%%%%%%%%%%%%%%%%%%%%%%%%%%%%%%%%%%%%%%%%%%%%%%%%%%%%%%%%%%%%%%%%%%%%
%%%%%%%%%%%%%%%%%%%%%%%%%%%%%%%%%%%%%%%%%%%%%%%%%%%%%%%%%%%%%%%%%%%%%%%%
%%%%%%%%%%%%%%%%%%%%%%%%%%%%%%%%%%%%%%%%%%%%%%%%%%%%%%%%%%%%%%%%%%%%%%%%
%%%%%%%%%%%%%%%%%%%%%%%%%%%%%%%%%%%%%%%%%%%%%%%%%%%%%%%%%%%%%%%%%%%%%%%%
%%%%%%%%%%%%%%%%%%%%%%%%%%%%%%%%%%%%%%%%%%%%%%%%%%%%%%%%%%%%%%%%%%%%%%%%
%%%%%%%%%%%%%%%%%%%%%%%%%%%%%%%%%%%%%%%%%%%%%%%%%%%%%%%%%%%%%%%%%%%%%%%%

\subsection{The proof of theorem \ifx\usehyperref\undefined \ref{theorem barr embedding} \else \ref*{theorem barr embedding} \fi}\label{subsection proof embedding}

We now turn to the proof of theorem \ref{theorem barr embedding}.

\begin{notation}
Let $\ecal$ be an accessible category with small products. We denote by $\Fun^\Uppi(\ecal, \Spc)^\acc$ the full subcategory of $\Fun(\ecal, \Spc)$ on the functors which are accessible and preserve small products.
\end{notation}

\begin{lemma}\label{lemma acessible product pres functors are closure}
Let $\ecal$ be an accessible category with small products.  Then:
\begin{enumerate}[\normalfont (1)]
\item  $\Fun^\Uppi(\ecal, \Spc)^\acc$ is closed under finite limits and geometric realizations of Kan semisimplicial objects inside $\Fun(\ecal, \Spc)$.
\item $\Fun^\Uppi(\ecal, \Spc)^\acc$ is the smallest full subcategory of $\Fun(\ecal, \Spc)$ containing the corepresentable functors and closed under geometric realizations of Kan semisimplicial objects. 
\end{enumerate}
\end{lemma}
\begin{proof}
To prove (1) it suffices to show that accessible (resp. small product preserving) functors $\ecal \rightarrow \Spc$ are closed under finite limits and geometric realizations of Kan semisimplicial objects inside $\Fun(\ecal, \Spc)$. The case of accessibility follows from the fact that $\kappa$-filtered colimits in $\Spc$ commute with colimits and finite limits. We now address the case of small products. The closure under finite limits is clear, so we only need to show closure  under geometric realizations of Kan semisimplicial objects. To prove this it is enough to show that small products in $\Spc$ commute with geometric realizations of Kan semisimplicial objects, which follows from the correspondence between hypercovers and Kan semisimplicial objects.

We now prove (2). It follows from (1) that $\Fun^\Uppi(\ecal, \Spc)^\acc$ is a hypercomplete exact category and its inclusion into $\Fun(\ecal, \Spc)$ is regular. Consequently, it is enough to show that every object in $\Fun^\Uppi(\ecal, \Spc)^\acc$ has a hypercover by corepresentable functors. To do so it suffices to prove that every object $X$ in $\Fun^\Uppi(\ecal, \Spc)^\acc$ receives an effective epimorphism from a corepresentable functor. Since $X$ is accessible  it is the left Kan extension of its restriction to a small subcategory of $\ecal$, which implies that it may be written as a small colimit of corepresentable functors in $\Fun(\ecal, \Spc)$. It follows that there exists a small family of objects $e_\alpha$ in $\ecal^\op$ and an effective epimorphism $p: \bigoplus e_\alpha \rightarrow X$ in $\Fun(\ecal, \Spc)$ (where here the direct sum is computed in the functor category). Since $X$ preserves small products, the map $p$ factors through a map $q: e \rightarrow X$, where $e$ is the direct sum of the objects $e_\alpha$ in $\ecal^\op$. Our claim now follows from the fact that $q$ is an effective epimorphism.
\end{proof}

\begin{lemma}\label{lemma hypercover by wproj}
Let $\ccal$ be a bounded regular category. Then every object of  $\Pro(\ccal)^+$ admits a hypercover by objects of $\Pro(\ccal)_{\wproj}$.
\end{lemma}
 \begin{proof}
It suffices to show that every object $X$ in $\Pro(\ccal)^+$  receives an effective epimorphism from an object in $\Pro(\ccal)_{\wproj}$. Since $X$ is a regular sheaf on $\Pro(\ccal)$ we have that $X$ receives an effective epimorphism $X' \rightarrow X$ where $X'$ belongs to $\Pro(\ccal)$. Replacing $X$ by $X'$ we may now reduce to the case when $X$ belongs to $\Pro(\ccal)$. This follows from proposition \ref{prop effective small object argument}.
 \end{proof}
 
\begin{lemma}\label{lemma identify proplus with funpi}
Let $\ccal$ be a bounded regular category.
\begin{enumerate}[\normalfont (1)]
\item The category $\Pro(\ccal)^+$ is locally small.
\item Let $G: \Pro(\ccal)^+ \rightarrow \Fun((\Pro(\ccal)_{\wproj})^\op, \Spc)$ be the functor that sends each object $X$ to the restriction to $(\Pro(\ccal)_{\wproj})^\op$ of $\Hom_{ \Pro(\ccal)^+}(-, X)$. Then $G$ is fully faithful.
\item The image of $G$ consists of the functors $ (\Pro(\ccal)_{\wproj})^\op \rightarrow \Spc$ which are accessible and preserve small products.
\end{enumerate}
\end{lemma}
\begin{proof}
We begin with a proof of (1). Let $X, Y$ be objects in $\Pro(\ccal)^+$. We wish to show that $\Hom_{\Pro(\ccal)^+}(X, Y)$ is small. By lemma \ref{lemma hypercover by wproj} we may pick hypercovers $P_\bullet$ and $Q_\bullet$    of $X$ and $Y$ by objects in $\Pro(\ccal)_{\wproj}$. We have $\Hom_{\Pro(\ccal)^+}(X, Y) = \Tot \Hom_{\Pro(\ccal)^+}(P_\bullet, Y)$. Replacing $X$ by $P_n$ we may now assume that $X$ belongs to $\Pro(\ccal)_{\wproj}$. In this case we have that $\Hom_{\Pro(\ccal)^+}(X, Q_\bullet)$ is a hypercover of $\Hom_{\Pro(\ccal)^+}(X, Y)$. Replacing $Y$ by $Q_n$ we may now assume that $X$ and $Y$ both belong to $\Pro(\ccal)_{\wproj}$, in which case the desired claim follows  from  the fact that $\Pro(\ccal)$ is locally small.

We now prove (2). We wish to show that for every pair of objects $X$ and $Y$ in $\Pro(\ccal)^+$ the canonical map 
\[
\Hom_{\Pro(\ccal)^+}(X, Y) \rightarrow \Hom_{ \Fun((\Pro(\ccal)_{\wproj})^\op, \Spc)}(G(X), G(Y)) 
\]
is an isomorphism. Observe that $G$ is a regular functor between hypercomplete exact categories. It follows that both sides commute with geometric realizations of Kan semisimplicial objects in the $X$ variable. Applying lemma \ref{lemma hypercover by wproj} we may reduce to the case when $X$ belongs to $\Pro(\ccal)_{\wproj}$. In this case $G(X)$ is a representable sheaf on $ \Fun(\Pro(\ccal)_{\wproj})^\op, \Spc)$. It follows that $X$ and $G(X)$ are both projective, so we have that both sides preserve geometric realizations of Kan semisimplicial objects in the $Y$ variable. Another application of lemma \ref{lemma hypercover by wproj} allows us to further reduce to the case when $X$ and $Y$ belong to $\Pro(\ccal)_{\wproj}$, in which case the desired claim follows from the fully faithfulness of the Yoneda embedding for $\Pro(\ccal)_{\wproj}$.

It remains to establish (3). It follows from (2) together with lemma \ref{lemma hypercover by wproj} that the image of $G$ is the closure of the representable presheaves under geometric realizations of Kan semisimplicial objects. Part (3) now follows from an application of lemma \ref{lemma acessible product pres functors are closure}.
\end{proof}

\begin{lemma}\label{lemma in ccal iff 1}
Let $\ccal$ be a bounded exact category and let $X$ be an object of $\Pro(\ccal)^+$. The following are equivalent:
\begin{enumerate}[\normalfont (a)]
\item $X$ belongs to $\ccal$.
\item $X$ is truncated, and the restriction to $\Pro(\ccal)^\op$ of the functor $\Hom_{\Pro(\ccal)^+}(-, X)$ preserves small filtered colimits.
\end{enumerate}
\end{lemma}
\begin{proof}
The fact that (a) implies (b) is clear. Since $X$ is truncated and the terminal object of $\Pro(\ccal)^+$ belongs to $\ccal$ we have that $X$ admits a $k$-truncated morphism $X \rightarrow Y$ into an object $Y$ in $\ccal$ for some $k \geq -2$. We will show that in this scenario the object $X$ belongs to $\ccal$, by induction on $k$. The case $k = -2$ is clear, so assume that $k > -2$ and that the assertion is known for $k-1$. Since $X$ is a regular sheaf on $\Pro(\ccal)$, we may find an effective epimorphism $p: X' \rightarrow X$ where $X'$ belongs to $\Pro(\ccal)$. Write $X'$ as the cofiltered limit of a diagram of objects $X'_\alpha$ in $\ccal$. Then our assumptions on $X$ guarantee that $p$ factors through $X'_\alpha$ for some $\alpha$. Replacing $X'$ with $X'_\alpha$ we may now assume that $X'$ belongs to $\ccal$. We have that $X$ is the geometric realization of the \v{C}ech nerve of $p$, so it suffices to prove that this \v{C}ech nerve belongs to $\ccal$. To do so it is enough to prove that $X' \times_X X'$ belongs to $\ccal$. This follows from our inductive hypothesis, since the canonical map $X' \times_X X' \rightarrow X' \times_Y X'$ is $(k-1)$-truncated.
\end{proof}

\begin{lemma}\label{lemma weakly projective cover}
Let $\ccal$ be a regular category. Let $\Ical$ be a cofiltered poset, and let $X: \Ical \rightarrow \Pro(\ccal)$ be a functor. Then there exists a functor $P: \Ical \rightarrow \Pro(\ccal)$ and a natural transformation $\eta: P \rightarrow X$ with the following properties:
\begin{enumerate}[\normalfont (a)]
\item $P(i)$ is weakly projective for all $i$ in $\Ical$.
\item The map $\eta(i): P(i) \rightarrow X(i)$ is a weak effective epimorphism for all $i$ in $\Ical$.
\end{enumerate}
\end{lemma}
\begin{proof}
Let $S$ be the collection of arrows in $\Fun(\Ical, \Pro(\ccal))$ whose images under the evaluation functors  $\ev_i: \Fun(\Ical, \Pro(\ccal)) \rightarrow \Pro(\ccal)$ are effective epimorphisms in $\ccal$. Applying the small object argument in the presentable category $\Fun(\Ical, \Pro(\ccal))^\op$ we obtain a natural transformation $\eta: P \rightarrow X$ with the following properties:
\begin{enumerate}[\normalfont (a')]
\item $P$ is weakly left orthogonal to all maps in $S$.
\item $\eta$ is a transfinite precomposition of pullbacks of maps in $S$.
\end{enumerate}
We claim that $\eta$ has the desired properties. The fact that (b) is satisfied follows directly from (b') and the definitions, since the evaluation functors preserve small limits. To prove (a) it will suffice to show that the functors $\ev_i$ map objects which are weakly left orthogonal to $S$ to weakly projective objects. To do so it is enough to prove that for every $i$ in $\Ical$ the functor $R_i: \Pro(\ccal) \rightarrow \Fun(\Ical, \Pro(\ccal))$ which is right adjoint  $\ev_i$ sends effective epimorphisms in $\ccal$ to maps in $S$. In other words, we wish to show that $\ev_j \circ R_i$ maps effective epimorphisms in $\ccal$ to effective epimorphisms in $\ccal$, for all pairs $i, j$. Indeed, $\ev_j \circ R_i$ is the terminal endofunctor of $\Pro(\ccal)$ if there is no map $j \rightarrow i$, and the identity of $\Pro(\ccal)$ otherwise.
\end{proof}
  
\begin{lemma}\label{lemma characterize C}
Let $\ccal$ be a bounded exact category and let $X$ be an object of $\Pro(\ccal)^+$. The following are equivalent:
\begin{enumerate}[\normalfont (a)]
\item $X$ belongs to $\ccal$.
\item $X$ is truncated, and the restriction to $(\Pro(\ccal)_{\wproj})^\op$ of the functor $\Hom_{\Pro(\ccal)^+}(-, X)$ preserves small filtered colimits.
\end{enumerate}
\end{lemma}
\begin{proof}
By lemma \ref{lemma in ccal iff 1}, we reduce to showing that if $X$ satisfies (b) then the restriction to $\Pro(\ccal)^\op$ of  $\Hom_{\Pro(\ccal)^+}(-, X)$ preserves small filtered colimits. Let $Y_\alpha$ be a diagram of objects of $\Pro(\ccal)$ indexed by a cofiltered poset $\Ical$, and let $Y$ be its limit. We wish to show that the canonical map $\colim \Hom_{\Pro(\ccal)^+}(Y_\alpha, X) \rightarrow  \Hom_{\Pro(\ccal)^+}(Y , X)$ is an isomorphism.

 Applying lemma \ref{lemma weakly projective cover} we may inductively construct a semisimplicial hypercover $P_{\alpha, \bullet}$ of $Y_{\alpha}$ in $\Fun(\Ical, \Pro(\ccal)^+)$ such that $P_{\alpha, n}$ is a weakly projective object of $\Pro(\ccal)$ for all $\alpha$ and all $n$. The fact that $\Pro(\ccal)^+$ is hypercomplete implies that $|P_{\alpha, \bullet}| = Y_\alpha$. Let $P_{\bullet}$ be the limit of $P_{\alpha, \bullet}$ and note that $Y = |P_\bullet|$, by proposition \ref{prop commute tot and geometric realiz}.  We now have
 \[
 \Hom_{\Pro(\ccal)^+}(Y , X) = \Tot  \Hom_{\Pro(\ccal)^+}(P_{\bullet} , X) = \Tot \colim \Hom_{\Pro(\ccal)^+}(P_{\alpha, \bullet} , X).
 \]
 Since $X$ is truncated there exists $k$ such that $\Hom_{\Pro(\ccal)^+}(P_{\alpha, n} , X)$ belongs to $\Spc_{\leq k}$ for all $\alpha$ and all $n$. Using the fact that filtered colimits and totalizations commute in $\Spc_{\leq k}$  we deduce
 \[
  \Hom_{\Pro(\ccal)^+}(Y , X) =   \colim \Tot \Hom_{\Pro(\ccal)^+}(P_{\alpha, \bullet} , X) = \colim \Hom_{\Pro(\ccal)^+}(Y_\alpha , X) 
\]
as desired.
\end{proof}
 
 \begin{proof}[Proof of theorem \ref{theorem barr embedding}]
 Let $\Fun^\lex(\ccal, \Spc)$ be the full subcategory of $\Fun(\ccal, \Spc)$ on the left exact functors. Then we have an equivalence $\Pro(\ccal)^\op = \Fun^\lex(\ccal, \Spc)$ given by sending each object $X$ in $\Pro(\ccal)^\op$ to the restriction to $\ccal$ of  $\Hom_{\Pro(\ccal)}(X, -)$. This restricts to an equivalence $(\Pro(\ccal)_{\wproj})^\op = \Fun^\reg(\ccal, \Spc)$. Part (1) of the theorem is now a direct consequence of proposition \ref{prop accessibility of wproj}.    We may identify the functor $\Upsilon_\ccal$ with the composite functor
 \[
 \ccal \hookrightarrow \Pro(\ccal)^+ \xrightarrow{G} \Fun( (\Pro(\ccal)_{\wproj})^\op, \Spc)
 \]
 where $G$ sends each object $X$ in $\Pro(\ccal)^+$ to the restriction to $ (\Pro(\ccal)_{\wproj})^\op$ of the functor $\Hom_{\Pro(\ccal)^+}(-, X)$. Part (2) of the theorem now follows from part (2) of lemma \ref{lemma identify proplus with funpi}, while part (3) of the theorem follows from a combination of lemma \ref{lemma characterize C} together with part (3) of lemma  \ref{lemma identify proplus with funpi}.
 \end{proof}

\ifx\inmain\undefined
\bibliographystyle{myamsalpha2}
\bibliography{References}
\fi

%%%%%%%%%%%%%%%%%%%%%%%%%%%%%%%%%%%%%%%%%%%%%%%%%%%%%%%%%%%%%%%%%%%%%%%%
%%%%%%%%%%%%%%%%%%%%%%%%%%%%%%%%%%%%%%%%%%%%%%%%%%%%%%%%%%%%%%%%%%%%%%%%
%%%%%%%%%%%%%%%%%%%%%%%%%%%%%%%%%%%%%%%%%%%%%%%%%%%%%%%%%%%%%%%%%%%%%%%%
%%%%%%%%%%%%%%%%%%%%%%%%%%%%%%%%%%%%%%%%%%%%%%%%%%%%%%%%%%%%%%%%%%%%%%%%
%%%%%%%%%%%%%%%%%%%%%%%%%%%%%%%%%%%%%%%%%%%%%%%%%%%%%%%%%%%%%%%%%%%%%%%%
%%%%%%%%%%%%%%%%%%%%%%%%%%%%%%%%%%%%%%%%%%%%%%%%%%%%%%%%%%%%%%%%%%%%%%%%

\section{Additive exact categories}\label{section exact}

The goal of this section  is to study the specialization of the theory of exact categories and exact completions to the additive context. A central role in this section will be played by the theory of (finitely complete) prestable categories from \cite{SAG} appendix C, which we regard as $\infty$-categorical analogues of abelian categories.

We begin in \ref{subsection prestables} by discussing the connection between exactness and prestability. Our main result on this topic is theorem \ref{teo prestable iff exact and additive}, which states that a finitely complete category is prestable if and only if it is exact and additive, and identifies those finitely complete prestable categories which are hypercomplete in the sense of definition \ref{definition hypercomplete exact}. As an application of this result together with the $\infty$-categorical version of Barr's embedding theorem (theorem \ref{theorem barr embedding}) we prove an extension of Freyd-Mitchell's embedding theorem to the setting of finitely complete prestable categories, see corollary \ref{coro freyd mitchell}.

In \ref{subsection abelian n} we explore a family of statements which interpolate between the $(1,1)$-categorical description of exact additive categories (theorem \ref{theorem classical}) and the $(\infty,1)$-categorical description. We introduce for each $n \geq 1$ a class of additive categories which we call abelian $(n,1)$-categories, whose axioms directly generalize the axioms of abelian $(1,1)$-categories. We then generalize theorem \ref{theorem classical} to all values of $n$, by showing that an additive $(n,1)$-category is abelian if and only if it is exact (theorem \ref{teo abelian n}). As a consequence, we deduce a generalization of the Freyd-Mitchell embedding theorem to the setting of abelian $(n,1)$-categories.

Finally, in \ref{subsection deriving} we discuss the procedure  of deriving abelian $(n,1)$-categories. Our main result on this topic is theorem \ref{theorem deriving}, which constructs for each abelian $(n,1)$-category $\ccal$ a pair of finitely complete prestable categories which we regard as the connective bounded and unbounded derived categories of $\ccal$. We obtain these by specializing the theory of exact completions and hypercompletions, and in particular we obtain a universal property for these derived categories.

%%%%%%%%%%%%%%%%%%%%%%%%%%%%%%%%%%%%%%%%%%%%%%%%%%%%%%%%%%%%%%%%%%%%%%%%
%%%%%%%%%%%%%%%%%%%%%%%%%%%%%%%%%%%%%%%%%%%%%%%%%%%%%%%%%%%%%%%%%%%%%%%%
%%%%%%%%%%%%%%%%%%%%%%%%%%%%%%%%%%%%%%%%%%%%%%%%%%%%%%%%%%%%%%%%%%%%%%%%
%%%%%%%%%%%%%%%%%%%%%%%%%%%%%%%%%%%%%%%%%%%%%%%%%%%%%%%%%%%%%%%%%%%%%%%%
%%%%%%%%%%%%%%%%%%%%%%%%%%%%%%%%%%%%%%%%%%%%%%%%%%%%%%%%%%%%%%%%%%%%%%%%
%%%%%%%%%%%%%%%%%%%%%%%%%%%%%%%%%%%%%%%%%%%%%%%%%%%%%%%%%%%%%%%%%%%%%%%%

\subsection{Prestable categories}\label{subsection prestables}

We begin by recalling the definition of prestable category from \cite{SAG} appendix C.

\begin{definition}\label{def prestable}
We say that a category $\Ccal$ is prestable if:
\begin{enumerate}[\normalfont(1)]
\item $\Ccal$ is pointed (that is, it admits a zero object) and admits finite colimits.
\item The suspension functor $\Sigma: \Ccal \rightarrow \Ccal$ is fully faithful.
\item Every morphism in $\ccal$ of the form $f: Y \rightarrow \Sigma Z$  admits a fiber, and furthermore the fiber sequence $\operatorname{fib}(f) \rightarrow Y \rightarrow \Sigma Z$ is also a cofiber sequence.
\end{enumerate}
\end{definition}

\begin{remark}\label{remark prestable t structure}
It is shown in \cite{SAG} proposition C.1.2.9 that a finitely complete category $\Ccal$ is prestable if and only if there is a stable category $\Dcal$ equipped with a t-structure, and an equivalence $\Ccal = \Dcal_{\geq 0}$. This provides an ample source of prestable categories.
\end{remark}

\begin{example}\label{example lmod is prestable}
Let $A$ be a connective ring spectrum. Then the category $\LMod_A^\cn$ of connective left $A$-modules is prestable.
\end{example}

\begin{example}
Let $\ccal$ be a stable category. Then $\ccal$ is prestable.
\end{example}

\begin{remark}
While there exist many prestable categories that do not admit finite limits, our focus in this paper will be on finitely complete prestable categories, which we regard as $\infty$-categorical analogues of abelian categories. A  non necessarily finitely complete category with finite colimits $\ccal$ is prestable if and only if the presentable category $\Ind(\ccal)$ is prestable.  In light of this, general prestable categories may be thought of as analogues the notion of ind-abelian category from \cite{Sch}. 
\end{remark}

The following is \cite{SAG} definition C.1.2.12:

\begin{definition}
Let $\Ccal$ be a finitely complete prestable category. We say that $\ccal$ is separated if whenever $X$ is an object of $\ccal$ such that $\tau_{\leq k}(X) = 0$ for all $k \geq 0$ then $X = 0$.
\end{definition}

We are now ready to formulate our main result connecting prestability and exactness:

\begin{theorem}\label{teo prestable iff exact and additive}
Let $\Ccal$ be a finitely complete category. 
\begin{enumerate}[\normalfont (1)]
\item The category $\Ccal$ is prestable if and only if it is exact and additive. 
\item The category $\Ccal$ is stable if and only if it is exact, additive, and every morphism in $\Ccal$ is an effective epimorphism.
\item Assume that $\ccal$ is prestable. Then $\ccal$ is  hypercomplete  if and only if it is separated and admits geometric realizations of simplicial objects.
\item Assume that $\ccal$ is prestable, and let $F: \ccal \rightarrow \dcal$ be a functor into another finitely complete prestable category $\dcal$. Then $F$ is regular if and only if it is exact (i.e., preserves finite limits and colimits).
\end{enumerate}
\end{theorem}

Before going into the proof, we record the following consequence, which is a prestable version of the Freyd-Mitchell embedding theorem.

 \begin{corollary}\label{coro freyd mitchell}
 Let $\ccal$ be a finitely complete prestable category. Assume that every object of $\ccal$ is the limit of its Postnikov tower. Then there exists a connective commutative ring spectrum $A$ and a fully faithful exact embedding $\ccal \rightarrow \LMod_A^\cn$.
 \end{corollary}
 \begin{proof}
 Combining theorem  \ref{teo prestable iff exact and additive} and corollary \ref{coro upsilon when postnikov are limits}, we may find a (small) category $\Ical$ and a fully faithful regular functor $F: \ccal \rightarrow \Pcal(\Ical)$. Consider the category $\Sp^\cn \otimes \Pcal(\Ical)$ of grouplike $E_\infty$-monoids in $\Pcal(\Ical)$ and let $U :\Sp^\cn \otimes \Pcal(\Ical) \rightarrow \Pcal(\Ical)$ be the forgetful functor. Since $F$ preserves finite products and $\ccal$ is additive we have an induced fully faithful functor  $G: \ccal \rightarrow \Sp^\cn \otimes \Pcal(\Ical)$ such that $F = U \circ G$. Since $U$ creates finite limits and geometric realizations we have that $\Sp^\cn \otimes \Pcal(\Ical)$ is an exact category and $G$ is a regular functor. 
 
 Observe that $\Sp^\cn \otimes \Pcal(\Ical)$ is a separated Grothendieck prestable category. Let   $\lbrace X_\alpha \rbrace$ be a set of representatives for the isomorphism classes of objects of $\Ical$, and let $P = \bigoplus_\alpha \SS \otimes X_\alpha$ be the direct sum of the free grouplike $E_\infty$-monoids on the objects $X_\alpha$. Then $P$ is a projective generator of $\Sp^\cn \otimes \Pcal(\Ical)$. Let $A$ be the opposite of the connective ring spectrum of endomorphisms of $P$. Then by \cite{SAG} theorem C.2.1.6 the canonical functor 
 \[
 \Hom^\enh_{\Sp^\cn \otimes \Pcal(\Ical)}(P, -): \Sp^\cn \otimes \Pcal(\Ical) \rightarrow \LMod_{A}^\cn
 \]
 is fully faithful, and since $P$ is projective the above is also regular. The desired embedding is now obtained as the composition of $ \Hom^\enh_{\Sp^\cn \otimes \Pcal(\Ical)}(P, -)$ and $G$, with exactness following from part (4) of theorem \ref{teo prestable iff exact and additive}.
 \end{proof}

The remainder of this section is devoted to the proof of theorem \ref{teo prestable iff exact and additive}. The basic observation is that in an additive context, morphisms and groupoid objects are related by a restricted version of the Dold-Kan correspondence.

\begin{notation}\label{notation bar}
Let $\Ccal$ be a category with finite coproducts and let $Y \leftarrow X \rightarrow Y'$ be a span in $\Ccal$. We denote by $\Bar^{\amalg}_X(Y, Y')_\bullet$ the relative Bar construction of $Y$ and $Y'$ over $X$, with respect to the cocartesian symmetric monoidal structure. In other words, $\Bar^{\amalg}_X(Y, Y')_\bullet$ is the simplicial diagram with entries $\Bar^{\amalg}_X(Y, Y')_n = Y \oplus X^{{\amalg} n} \oplus Y'$.

Assume now given a commutative square
\[
\begin{tikzcd}
X \arrow{d}{} \arrow{r}{} & Y \arrow{d}{} \\ Y' \arrow{r}{} & Z
\end{tikzcd}
\]
in $\Ccal$. Then there is an induced augmentation $\Bar^\amalg_X(Y, Y')_\bullet \rightarrow Z$. We denote by $\Bar^\amalg_X(Y, Y' ; Z)_\bullet$ the corresponding augmented simplicial diagram.
\end{notation}

\begin{remark}\label{remark lke}
Let $\Ccal$ be a category with finite coproducts. Then a span $Y \leftarrow X \rightarrow Y'$ in $\Ccal$ corresponds to a functor  $\Lambda^2_0 \rightarrow \Ccal$. The simplicial diagram $\Bar^\amalg_X(Y, Y')_\bullet$ is the left Kan extension of this functor along the functor $\Lambda^2_0 \rightarrow \Delta^\op$ 
that picks out the span $[0] \xleftarrow{1} [1] \xrightarrow{0} [0]$. Similarly, given a commutative square in $\Ccal$ as in notation \ref{notation bar}, the augmented simplicial diagram $\Bar^\amalg_X(Y, Y';Z)_\bullet$ is its left Kan extension along the functor $[1] \times [1] \rightarrow \Delta^\op_+$ that picks out the commutative square
\[
\begin{tikzcd}
{[1]} \arrow{d}{0} \arrow{r}{1} & {[0]} \arrow{d}{} \\ {[0]} \arrow{r}{} & {[-1]}.
\end{tikzcd}
\]
In particular, the given commutative square is a pushout if and only if the associated augmented simplicial diagram $\Bar^\amalg_X(Y, Y' ; Z)_\bullet$ is a colimit diagram.
\end{remark}

\begin{notation}
Let $\Ccal$ be an additive category. We denote by 
\[
\dk : \Ccal^{[1]} \rightarrow \Ccal^{\Delta^\op}
\]
the functor that sends an arrow $X \rightarrow Y$ to the simplicial object $\Bar^\amalg_X(Y, 0)_\bullet$. Let $\Ccal^\seq$ be the full subcategory of $\Ccal^{[1] \times [1]}$ on those commutative squares
\[
\begin{tikzcd}
X \arrow{d}{} \arrow{r}{} & Y \arrow{d}{} \\
Y' \arrow{r}{} & Z
\end{tikzcd}
\]
 with the property that $Y' = 0$. In other words, $\Ccal^\seq$ is the category of pairs of composable arrows $X \rightarrow Y \rightarrow Z$ equipped with a nullhomotopy for the composition.  We denote by
 \[
 \dk^+ : \Ccal^\seq \rightarrow \Ccal^{\Delta^{\op}_+}
 \]
 the functor that sends a sequence $X \rightarrow Y \rightarrow Z$ to the augmented simplicial object $\Bar^\amalg_X(Y, 0; Z)_\bullet$.
\end{notation}

\begin{remark}\label{remark pullback}
Let $\Ccal$ be an additive category. Then we have a commutative square of categories
\[
\begin{tikzcd}
\Ccal^\seq \arrow{r}{\dk^+}\arrow{d}{} & \Ccal^{\Delta_+^\op} \arrow{d}{} \\
\Ccal^{[1]} \arrow{r}{\dk} & \Ccal^{\Delta^\op}
\end{tikzcd}
\]
where the left vertical arrow sends a sequence $X \rightarrow Y \rightarrow Z$ to the map $X \rightarrow Y$, and the right vertical arrow sends an augmented simplicial diagram to the underlying simplicial diagram.
\end{remark}

\begin{definition}
We say that an additive category $\Ccal$ admits complements if for every map $s: U \rightarrow X$ admitting a retraction there exists a map $s': U' \rightarrow X$ such that $X = U \oplus U'$. In this case we call $s'$ a complement for $s$.
\end{definition}

\begin{remark}
Let $\Ccal$ be an additive category and let $s: U \rightarrow X$ be a map admitting a retraction $r$. If $r$ admits a fiber then we have $X = U \oplus \fib(r)$. Conversely, if $s': U' \rightarrow X$ is a complement for $s$, then $r$ admits a fiber, and there exists an isomorphism between $\fib(r)$ and $U'$.
\end{remark}

\begin{lemma}\label{lemma sobre DK}
Let $\Ccal$ be an additive category admitting complements. Then:
\begin{enumerate}[\normalfont (1)]
\item The functor $\dk: \Ccal^{[1]} \rightarrow \Ccal^{\Delta^\op}$ is fully faithful.
\item A simplicial object $X_\bullet$ belongs to the image of $\dk$ if and only if it is a groupoid.
\item The functor $\dk^+: \Ccal^{\seq} \rightarrow \Ccal^{\Delta^\op_+}$ is fully faithful.
\item An augmented simplicial object belongs to the image of $\dk^+$ if and only if the underlying simplicial object belongs to the image of $\dk$.
\item Let $\alpha$ be an element of $\Ccal^\seq$. Then $\alpha$ is a cofiber sequence if and only if $\dk^+(\alpha)$ is a colimit diagram.
\item Let $\alpha$ be an element of $\Ccal^\seq$. Then $\alpha$ is a fiber sequence if and only if $\dk^+(\alpha)$ is a \v{C}ech nerve.
\end{enumerate}
\end{lemma}
\begin{proof}
We first prove item (1). We may write $\dk$ as a composition
\[
\Ccal^{[1]} \xrightarrow{\iota} \Ccal^{\Lambda^2_0} \xrightarrow{\Bar^\amalg} \Ccal^{\Delta^\op}
\]
where $\iota$ sends an arrow $X \rightarrow Y$ to the span $Y \leftarrow X \rightarrow 0$, and $\Bar^\amalg$ sends a span $Y \leftarrow X \rightarrow Y'$ to $\Bar^\amalg_X(Y, Y')_\bullet$. The functor $\Bar^\amalg$ has a right adjoint which sends a simplicial object $S_\bullet$ to the span $S_0 \xleftarrow{d_0} S_1 \xrightarrow{d_1} S_0$. Since $\Ccal$ admits complements and $d_1: S_1 \rightarrow S_0$ admits a section, there exists a fiber for $d_1$. It follows that $\dk$ has a right adjoint which sends a simplicial object $S_\bullet$ to the composite map $\fib(d_1) \rightarrow S_1 \xrightarrow{d_0} S_0$. To show that $\dk$ is fully faithful we will show that the unit of this adjunction is invertible. This reduces to the assertion that if $X \rightarrow Y$ is a morphism and $S_\bullet$ is its image under $\dk$, then the induced sequence $X \rightarrow S_1 \xrightarrow{d_1} S_0$ is a fiber sequence, which is clear from the definitions.

We now prove item (2). We first show that the image of an arrow $f: X \rightarrow Y$ under $\dk$ is a groupoid. Since $\ccal$ is additive, we may endow $X$ and $Y$ with $E_\infty$-monoid structures, and make $Y$ into an $X$-module via restriction of scalars along $f$. Similarly, $0$ becomes an $X$-module via the projection $X \rightarrow 0$. We may identify the simplicial object $\Bar^\amalg_X(Y, 0)_\bullet$ with the relative Bar construction $\Bar^\times_X(Y, 0)_\bullet$ computed with respect to the cartesian symmetric monoidal structure on $\ccal$. The projection $Y \rightarrow 0$ induces a map of simplicial objects $\rho: \Bar^\times_X(Y, 0)_\bullet \rightarrow \Bar^\times_X(0, 0)_\bullet$. We note that $\Bar^\times_X(0, 0)_\bullet$ is a monoid object in $\ccal$. This is in fact a group object, since the $E_\infty$-monoid structure on $X$ is grouplike. Consequently, to show that $\Bar^\times_X(Y, 0)_\bullet$ is a groupoid object it will suffice to show that for every $n \geq 0$ the commutative square
\[
\begin{tikzcd}
 \Bar^\times_X(Y, 0)_n \arrow{r}{} \arrow{d}{} & \Bar^\times_X(Y, 0)_0 \arrow{d}{}\\
 \Bar^\times_X(0, 0)_n \arrow{r}{} & \Bar^\times_X(0, 0)_0 
 \end{tikzcd}
\]
induced from $\rho$ and the map $0: [0] \rightarrow [n]$ is a pullback square. Unwinding the definitions, the above recovers the commutative square
\[
\begin{tikzcd}
Y \times X^{n-1} \arrow{d}{} \arrow{r}{} & Y \arrow{d}{} \\
X^{n-1} \arrow{r}{} & 0
\end{tikzcd}
\]
where the top and left maps are the projections. This is a pullback square, as desired.

Assume now given a groupoid $S_\bullet$ in $\Ccal$ and let $\alpha: X \rightarrow Y$ be the image of $S_\bullet$ under the right adjoint to $\dk$, so that $Y = S_0$, $X = \fib(d_1)$, and $\alpha$ is the restriction of $d_0$. We claim that the morphism $\dk(\alpha) \rightarrow S_\bullet$ is an equivalence. Since both of these are groupoids we may reduce to showing that this map is an equivalence on $0$-simplices and $1$-simplices. Indeed, on $0$-simplices this recovers the isomorphism $Y \rightarrow S_0$, and on $1$-simplices it recovers the canonical morphism 
\[
X \oplus Y = \fib(d_1) \oplus S_0 \rightarrow S_1
\]
which is an isomorphism since the map $\fib(d_1) \rightarrow S_1$ is a complement for $s_0 : S_0 \rightarrow S_1$.

The proofs of (3) and (4) are analogous to (1) and (2), and are left to the reader.

The fact that (5) holds was already observed in remark \ref{remark lke}. We now prove item (6). Replacing $\Ccal$ by its free completion under sifted colimits if necessary, we may reduce to the case when $\Ccal$ has all finite limits. We have a commutative square
\[
\begin{tikzcd}
\Ccal^\seq \arrow{d}{} \arrow{r}{\dk^+} &  \Ccal^{\Delta^\op_+} \arrow{d}{} \\
\Ccal^{[1]} \arrow{r}{\id} & \Ccal^{[1]}
\end{tikzcd}
\]
where the left vertical arrow sends a sequence $X \rightarrow Y \rightarrow Z$ to $Y \rightarrow Z$, and the right vertical arrow sends an augmented simplicial object $S_\bullet$ to $S_0 \rightarrow S_{-1}$. The vertical arrows restrict to equivalences on the full subcategories of fiber sequences and \v{C}ech nerves, respectively. It therefore suffices to show that if $\dk^+(\alpha)$ is a \v{C}ech nerve, then $\alpha$ is a fiber sequence. Since $\dk^+$ is fully faithful and has a right adjoint, we may reduce to showing that its right adjoint maps \v{C}ech nerves to fiber sequences. This follows from the commutativity of the above square  by passing to right adjoints of all arrows. 
\end{proof}

\begin{proof}[Proof of theorem \ref{teo prestable iff exact and additive}]
We begin with a proof of (1). Assume first that $\Ccal$ is prestable. The fact that $\Ccal$ is additive is \cite{SAG} example C.1.5.6. We now show that $\Ccal$ is exact. Consider first the case when $\Ccal$ is Grothendieck prestable (in other words, $\ccal$ is presentable and has left exact filtered colimits, see \cite{SAG} definition C.1.4.2). Then using \cite{SAG} theorem C.2.4.1 we see that $\Ccal$ is a left exact localization of the category $\LMod_A^\cn$ of connective modules over a connective ring spectrum $A$. Since exactness is preserved by passage to left exact localizations, we may reduce to the case $\Ccal = \LMod_A^\cn$. This follows from the fact that $\Spc$ is exact (since it is a topos) together with the fact that the forgetful functor $\LMod_A^\cn \rightarrow \Spc$ creates finite limits and geometric realizations.

We now address the case when $\Ccal$ is a general prestable category. Since $\Ind(\Ccal)$ is Grothendieck prestable, we already know it to be exact, so it suffices to show that $\Ccal$ is closed under finite limits and geometric realizations of groupoid objects inside $\Ind(\Ccal)$. The fact that $\Ccal$ is closed under finite limits is clear. By lemma \ref{lemma sobre DK}, the assertion that $\Ccal$ is closed under geometric realizations of groupoids is equivalent to the fact that $\Ccal$ is closed under cofibers in $\Ind(\Ccal)$. This is clear since cofibers are finite colimits.

Assume now that $\Ccal$ is exact and additive. By lemma \ref{lemma sobre DK} we have that $\Ccal$ admits cofibers, and hence $\Ccal$ admits finite colimits. We now verify conditions (2) and (3) from definition \ref{def prestable}:
\begin{enumerate}[\normalfont (1)]
\setcounter{enumi}{1}
\item Let $X$ be an object of $\Ccal$ and let $X_\bullet$ be the groupoid corresponding to the map $X \rightarrow 0$ under $\dk$. Then by lemma \ref{lemma sobre DK} the unit map $X \rightarrow \Omega \Sigma X$ agrees with the value on $[1]$ of the canonical map of groupoids from $X_\bullet$ to the \v{C}ech nerve of the morphism $0 \rightarrow |X_\bullet|$. This is an equivalence since $\Ccal$ is exact.
\item Let $f: Y \rightarrow \Sigma Z$ be a morphism in $\Ccal$. Then the sequence $\operatorname{fib}(f) \rightarrow Y \rightarrow \Sigma Z$ corresponds under $\dk^+$ to the augmented \v{C}ech nerve of $f$. To show that this sequence is a cofiber sequence we must show that $f$ is an effective epimorphism. In fact, we claim that the zero map $0 \rightarrow \Sigma Z $ is an effective epimorphism. Indeed, as observed in the previous item, $\Sigma Z $ is the realization of a groupoid acting on $0$.
\end{enumerate}

We now address (2). In light of (1), it suffices to show that if $\Ccal$ is a finitely complete prestable category, then $\Ccal$ is stable if and only if every morphism in $\Ccal$ is an effective epimorphism. It follows from proposition \ref{proposition 0 truncated are exact} that a morphism $f: X \rightarrow Y$  in $\Ccal$ is an effective epimorphism if and only if it induces an effective epimorphism  $\tau_{\leq 0}(X) \rightarrow \tau_{\leq 0}(Y)$. If every morphism in $\Ccal$ is an effective epimorphism then for every $X$ in $\Ccal_{\leq 0}$ we have that $0 \rightarrow X$ is an effective epimorphism, and therefore $X = 0$. Hence $\Ccal_{\leq 0} = 0$, which means that every object of $\Ccal$ is $\infty$-connective, and therefore $\Ccal$ is stable. Conversely, if $\Ccal$ is stable then $\Ccal_{\leq 0} = 0$, which implies that every morphism in $\Ccal_{\leq 0}$ is an effective epimorphism and hence every morphism in $\Ccal$ is an effective epimorphism, as desired.

We now prove (3). Assume first that $\ccal$ is hypercomplete. By corollary \ref{coro equivalence hypercomplete} we have that every $\infty$-connective object of $\ccal$ is zero, and hence $\ccal$ is separated. To prove that $\ccal$ admits geometric realizations of simplicial objects it suffices to show that if $X_\bullet$ is a simplicial object of $\ccal$ then (the semisimplicial object underlying) $X_\bullet$ satisfies the Kan condition. In other words, we have to show that for every $n \geq 1$ and every $0 \leq i \leq n$ the map $X([n]) \rightarrow X(\Lambda^n_i)$ is an effective epimorphism. We claim that this map admits a section. Let 
\[
\Hom^{\enh}_\ccal(X(\Lambda^n_i), -): \ccal \rightarrow \Sp^\cn
\]
 be the functor corepresented by the grouplike $E_\infty$-comonoid $X(\Lambda^n_i)$. The image of the simplicial object $X_\bullet$ under $\Hom^{\enh}_\ccal(X(\Lambda^n_i), -)$ defines a simplicial object $Y_\bullet$ in $\Sp^\cn$. To prove our claim it will suffice to show that the map $Y([n]) \rightarrow Y(\Lambda^n_i)$ is an effective epimorphism in $\Sp^\cn$. In other words, we have reduce to showing that the simplicial $E_\infty$-group $Y_\bullet$ satisfies the Kan condition. This follows from \cite{DAGVIII} corollary 4.2.7.

Assume now that $\ccal$ is separated and admits geometric realizations of simplicial objects. Let $f: X \rightarrow Y$ be an $\infty$-connective morphism in $\ccal$. Since $\cofib(f)$ is $\infty$-connective and $\ccal$ is separated we have that $\cofib(f) = 0 $. Since $f$ is the fiber of the map $Y \rightarrow \cofib(f)$ we deduce that $f$ is an isomorphism.  By  corollary \ref{coro equivalence hypercomplete} to show that $\ccal$ is hypercomplete it will suffice to show that $\ccal$ admits geometric realizations of semisimplicial objects. Applying \cite{SAG} proposition C.6.6.9 we may find a   Grothendieck prestable category $\Dcal$ such that $\ccal$ is equivalent to the full subcategory of $\dcal$ on the almost compact objects. We may thus reduce to showing that almost compact objects of $\dcal$ are closed under geometric realizations of semisimplicial objects. To do so it is enough to show that for each $k \geq -1$ the full subcategory $(\dcal_{\leq k})^\omega$ of $\dcal_{\leq k}$ on the compact objects is closed under geometric realizations of semisimplicial objects. Since restriction along the inclusion $\Delta^{\op}_{\text{s}, < k+2 } \rightarrow \Delta^{\op}_{\text{s}}$ preserves colimits valued in $(k+1)$-categories, we may further reduce to showing that  $(\dcal_{\leq k})^\omega$ is closed under $\Delta^{\op}_{\text{s}, <k+2 } $-indexed colimits inside $\dcal_{\leq k}$, which follows from the fact that $\Delta^{\op}_{\text{s}, < k+2 } $ is a finite category.

It remains to establish (4). Assume first that $F$ is exact. We need to show that it preserves effective epimorphisms. Let $f: X \rightarrow Y$ be an effective epimorphism in $\Ccal$. Then the fiber sequence $\fib(f) \rightarrow X \rightarrow Y$ is also a cofiber sequence. Therefore $F(\fib(f)) \rightarrow F(X) \rightarrow F(Y)$ is a cofiber sequence, which implies that $F(f): F(X) \rightarrow F(Y)$ is an effective epimorphism, as desired.

Conversely, assume that $F$ is regular. By definition, $F$ is left exact, so to show exactness it suffices to show that $F$ preserves cofiber sequences. Let $X \rightarrow Y \rightarrow Z$ be a cofiber sequence. This is in particular a fiber sequence, so its image $F(X) \rightarrow F(Y) \rightarrow F(Z)$ is a fiber sequence in $\Dcal$. Furthermore, the map $Y \rightarrow Z$ is an effective epimorphism, and therefore $F(Y) \rightarrow F(Z)$ is an effective epimorphism as well. This implies that $F(X) \rightarrow F(Y) \rightarrow F(Z)$ is also a cofiber sequence, as desired.
\end{proof}
 
%%%%%%%%%%%%%%%%%%%%%%%%%%%%%%%%%%%%%%%%%%%%%%%%%%%%%%%%%%%%%%%%%%%%%%%%
%%%%%%%%%%%%%%%%%%%%%%%%%%%%%%%%%%%%%%%%%%%%%%%%%%%%%%%%%%%%%%%%%%%%%%%%
%%%%%%%%%%%%%%%%%%%%%%%%%%%%%%%%%%%%%%%%%%%%%%%%%%%%%%%%%%%%%%%%%%%%%%%%
%%%%%%%%%%%%%%%%%%%%%%%%%%%%%%%%%%%%%%%%%%%%%%%%%%%%%%%%%%%%%%%%%%%%%%%%
%%%%%%%%%%%%%%%%%%%%%%%%%%%%%%%%%%%%%%%%%%%%%%%%%%%%%%%%%%%%%%%%%%%%%%%%
%%%%%%%%%%%%%%%%%%%%%%%%%%%%%%%%%%%%%%%%%%%%%%%%%%%%%%%%%%%%%%%%%%%%%%%%

\subsection{Abelian \texorpdfstring{$(n,1)$}{(n,1)}-categories}\label{subsection abelian n}

Our next goal is to introduce a theory of abelian $(n,1)$-categories, which interpolates between the classical theory of abelian categories and the theory of finitely complete prestable categories. 

\begin{definition}
Let $\Ccal$ be a category and let $k \geq -2$ be an integer. We say that a morphism $f: X \rightarrow Y$ in $\Ccal$ is $k$-cotruncated if it defines a $k$-truncated morphism in $\Ccal^\op$. In other words, $f$ is $k$-cotruncated if for every object $Z$ in $\ccal$ the morphism $\Hom_\ccal(Y, Z) \rightarrow \Hom_\ccal(X, Z)$ of precomposition with $f$ is a $k$-truncated map of spaces.
\end{definition}

\begin{remark}\label{remark description k truncated}
Let $\Ccal$ be an additive category and let $k \geq -1$. Let $f: X \rightarrow Y$ be a morphism in $\Ccal$ admitting a fiber. Then $f$ is $k$-truncated if and only if $\operatorname{fib}(f)$ is a $k$-truncated object of $\Ccal$.

Let $n \geq 1$ and assume given a stable category with t-structure $\Dcal$ such that $\Ccal$ is the full subcategory of $\Dcal$ on those objects which are connective and $(n-1)$-truncated. Then a morphism $f: X \rightarrow Y$ in $\Ccal$ is $(n-2)$-truncated (resp. $(n-2)$-cotruncated) if and only if $f$ induces a monomorphism on $H_{n-1}$ (resp. an epimorphism on $H_0$). 
\end{remark}

\begin{remark}\label{remark fibers or cofibers are }
Let $n \geq 1$ and let $\Ccal$ be an additive $(n,1)$-category. Let $f: X \rightarrow Y$ be a morphism in $\Ccal$. If $f$ admits a cofiber then the map $Y \rightarrow \operatorname{cofib}(f)$ is $(n-2)$-cotruncated. Similarly, if $f$ admits a fiber then the map $\fib(f) \rightarrow X$ is $(n-2)$-truncated.
\end{remark}

\begin{definition}\label{definition abelian n}
Let $n \geq 1$ be an integer and let $\Ccal$ be an $(n,1)$-category. We say that $\Ccal$ is $n$-abelian (or an abelian $(n,1)$-category) if the following conditions are satisfied:
\begin{enumerate}[\normalfont (1)]
\item $\Ccal$ is additive and admits finite limits and colimits
\item Let $f: X \rightarrow Y$ be an $(n-2)$-cotruncated morphism in $\Ccal$. Then the fiber sequence $\operatorname{fib}(f) \rightarrow X \rightarrow Y$ is also a cofiber sequence.
\item  Let $f: X \rightarrow Y$ be an $(n-2)$-truncated morphism in $\Ccal$. Then the cofiber sequence $X \rightarrow Y \rightarrow \operatorname{cofib}(f)$ is also a fiber sequence.
\end{enumerate}
\end{definition}

\begin{example}
In the case $n = 1$, definition \ref{definition abelian n} recovers the classical notion of abelian category.
\end{example}

\begin{example}\label{example of abelian n}
Let $\Dcal$ be a finitely complete prestable category. Let $n \geq 1$ and set $\Ccal = \Dcal_{\leq n-1}$. It follows from remark \ref{remark description k truncated} that  the sequences appearing in items (2) and (3) of definition \ref{definition abelian n} are preserved under the inclusion of $\ccal$ into the Spanier-Whitehead category $\operatorname{SW}(\Dcal)$ (see \cite{SAG} section C.1). We conclude that $\Ccal$ is an abelian $(n,1)$-category.  We will see below that every abelian $(n,1)$-category arises in this way for some choice of $\Dcal$.
\end{example}

\begin{example}\label{example closed under finite limits and colimits is ab}
Let $n \geq 1$. Let $\dcal$ be an abelian $(n,1)$-category and let $\ccal$ be a full subcategory of $\dcal$ which is closed under finite limits and colimits. Then $\ccal$ is an abelian $(n,1)$-category. Combining this with examples \ref{example lmod is prestable} and \ref{example of abelian n} we deduce that if $A$ is an $(n-1)$-truncated connective ring spectrum then every full subcategory of $(\LMod_A^\cn)_{\leq n-1}$ which is closed under finite limits and colimits is an abelian $(n,1)$-category. We will see below that every abelian $(n,1)$-category arises in this way for some choice of $A$. 
\end{example}

\begin{remark}
One may contemplate extending definition \ref{definition abelian n} in the limit $n \to \infty$. Declaring every morphism to be both $\infty$-truncated and $\infty$-cotruncated yields the notion of stable category. There are however two other notions which are relevant, namely finitely complete prestable categories and opposites of finitely complete prestable categories. These notions no longer have the symmetry under passage to opposites that is present in definition \ref{definition abelian n}. It follows however from theorems \ref{teo prestable iff exact and additive} and \ref{teo abelian n} that the axioms for abelian $(n,1)$-categories may be reformulated in such a way that in the limit $n \to \infty$ they recover the notion of finitely complete prestable category. 
\end{remark}

We are now ready to formulate the main theorem of this section:

\begin{theorem}\label{teo abelian n}
Let $n \geq 1$.
\begin{enumerate}[\normalfont (1)]
\item An $(n,1)$-category $\Ccal$ is $n$-abelian if and only if it is additive and $n$-exact.
\item Let $F: \ccal \rightarrow \dcal$ be a functor between abelian $(n,1)$-categories. Then $F$ is regular if and only if it is exact (i.e., preserves finite limits and colimits). 
\end{enumerate}  
\end{theorem}

The proof of theorem \ref{teo abelian n} needs some preliminary lemmas.

\begin{lemma}\label{lemma Aex is additive}
Let $\Ccal$ be a regular additive category. Then $\Ccal^\ex$ is additive.
\end{lemma}
\begin{proof}
We first show that $\Ccal^\ex$ is pointed. Let $0$ be the zero object of $\Ccal$. Since the inclusion $\Ccal \rightarrow \Ccal^\ex$ preserves finite limits, we have that $0$ is a final object of $\Ccal^\ex$. We claim that $0$ is also initial in $\Ccal^\ex$. To prove this, it suffices to show that the functor $\Hom_{\Sh(\Ccal)}(0, -): \Sh(\Ccal) \rightarrow \Spc$ preserves geometric realizations. We claim in fact that it preserves all colimits. Indeed, this follows from the fact that the functor $\Hom_{\Ccal}(0, -): \Ccal \rightarrow \Spc$ is regular.

We now show that $\Ccal^\ex$ admits biproducts. Let $X, Y$ be objects of $\Ccal^\ex$. Then $X \times Y$ belongs to $\Ccal^\ex$, and we have maps $(\id, 0): X \rightarrow X \times Y \leftarrow Y : (0, \id)$. We must show that this is a coproduct diagram in $\Ccal^\ex$. Since products commute with geometric realizations in $\Sh(\Ccal)$, we have that this diagram commutes with taking geometric realizations of groupoids in the $X$ and $Y$ variables. Hence we may reduce to the case when $X$ and $Y$ belong to $\Ccal$.

Say that an object $Z$ in $\Ccal^\ex$ is good if the induced map 
\[
\psi_{X, Y, Z} : \Hom_{\Ccal^\ex}(X \times Y, Z) \rightarrow \Hom_{\Ccal^\ex}(X, Z) \times \Hom_{\Ccal^\ex}(Y,Z)
\]
 is an equivalence for every $X, Y$ in $\Ccal$. We have to show that every object in $\Ccal^\ex$ is good. Since $\Ccal$ admits biproducts, we see that every $Z$ in $\Ccal$ is good. Hence we reduce to showing that if $Z_\bullet$ is a groupoid object which is termwise good, then its realization $|Z_\bullet|$ is good.

Observe that for every $X, Y$ in $\Ccal$ we have a commutative square of spaces
\[
\begin{tikzcd}[column sep = large]
{| \Hom_{\Ccal^\ex}(X \times Y, Z_{\bullet}) | } \arrow{r}{|\psi_{X, Y, Z_\bullet}|} \arrow{d}{}  & {|\Hom_{\Ccal^\ex}(X, Z_\bullet)| \times |\Hom_{\Ccal^\ex}(Y, Z_\bullet)|} \arrow{d}{} \\
 \Hom_{\Ccal^\ex}(X \times Y, |Z_{\bullet}|) \arrow{r}{\psi_{X, Y, |Z_\bullet|}} &   \Hom_{\Ccal^\ex}(X, |Z_{\bullet}|) \times \Hom_{\Ccal^\ex}(Y, |Z_{\bullet}|) .  
\end{tikzcd}
\]
We have that $\Hom_{\Ccal^\ex}(X \times Y, Z_{\bullet})$ is the \v{C}ech nerve of the map 
\[
\Hom_{\Ccal^\ex}(X \times Y, Z_{0}) \rightarrow  \Hom_{\Ccal^\ex}(X \times Y, |Z_{\bullet}|) 
\]
 and therefore  the left vertical map in the above square is a monomorphism whose image consists of those maps $X \times Y \rightarrow |Z_\bullet|$ which factor through $Z_0$. Similarly, the right vertical map is a monomorphism whose image consists of those pairs of maps $X \rightarrow |Z_\bullet|$ and $Y \rightarrow |Z_\bullet|$ which factor through $Z_0$.

The fact that $Z_\bullet$ is termwise good implies that the upper horizontal arrow in the above diagram is an equivalence. We therefore conclude that $\psi_{X, Y, |Z_\bullet|}$ restricts to an equivalence between the space of maps $X \times Y \rightarrow |Z_\bullet|$ which factor through $Z_0$ and the space of pairs of maps $X \rightarrow |Z_\bullet|$ and $Y \rightarrow |Z_\bullet|$ which factor through $Z_0$.

Assume now given a pair of effective epimorphisms $X_0 \rightarrow X$ and $Y_0 \rightarrow Y$, with $X_0, Y_0$ in $\Ccal$. Let $X_\bullet$ and $Y_\bullet$ be the \v{C}ech nerves of the projections $X_0 \rightarrow X$ and $Y_0 \rightarrow Y$, respectively. We have a commutative square
\[
\begin{tikzcd}[column sep = 2.5cm]
\Hom_{\Ccal^\ex}(X \times Y, |Z_\bullet|) \arrow{r}{\psi_{X, Y, |Z_\bullet|}} \arrow{d}{} & \Hom_{\Ccal^\ex}(X , |Z_\bullet|)  \times \Hom_{\Ccal^\ex}(Y, |Z_\bullet|)  \arrow{d}{} \\
\lim_{\Delta} \Hom_{\Ccal^\ex}(X_n \times Y_n, |Z_\bullet|) \arrow{r}{\lim_\Delta \psi_{X_n, Y_n, |Z_\bullet|}}  & \lim_{\Delta} \Hom_{\Ccal^\ex}(X_n , |Z_\bullet|)  \times \lim_\Delta \Hom_{\Ccal^\ex}(Y_n, |Z_\bullet|) 
\end{tikzcd}
\]
where the vertical arrows are equivalences. It follows from this that $\psi_{X, Y, |Z_\bullet|}$ restricts to an equivalence between the space of maps $X \times Y \rightarrow |Z_\bullet|$ whose restriction to $X_0 \times Y_0$ factors through $Z_0$, and the space of pairs of maps $X \rightarrow |Z_\bullet|$ and $Y\rightarrow |Z_\bullet|$ whose restriction to $X_0$ and $Y_0$ factor through $Z_0$.

Hence to show that $\psi_{X, Y, |Z_\bullet|}$ is an equivalence we may reduce to proving the following two assertions:
\begin{enumerate}[(a)]
\item For every map $f: X \times Y \rightarrow |Z_\bullet|$ we may choose $X_0, Y_0$ such that the restriction of $f$ to $X_0 \times Y_0$ factors through $Z_0$.
\item For every pair of maps $g_X: X \rightarrow |Z_\bullet|$ and $g_Y : Y \rightarrow |Z_\bullet|$ we may choose $X_0, Y_0$ such that the restriction of $g_X$ and $g_Y$ to $X_0$ and $Y_0$ factor through $Z_0$.
\end{enumerate}

Item (b) is a direct consequence of the fact that the projection $Z_0 \rightarrow |Z_\bullet|$ is an effective epimorphism. We now  address (a). Let $f: X \times Y \rightarrow |Z_\bullet|$ be a morphism. Since the projection $Z_0 \rightarrow |Z_\bullet|$ is an effective epimorphism, we may pick an effective epimorphism $U \rightarrow X \times Y$ with $U$ in $\Ccal$, such that the restriction of $f$ to $U$ factors through $Z_0$. We may now take $X_0 = U \times_{X \times Y} X$ and $Y_0 = U \times_{X \times Y} Y$.

It now follows that $\Ccal^\ex$ is semiadditive, so the forgetful functor from the category of commutative comonoids in $\Ccal^\ex$ to $\Ccal^\ex$ is an equivalence. To finish the proof, it remains to show that every object of $\Ccal^\ex$ is in fact a grouplike commutative comonoid in $\Ccal^\ex$. Since $\Ccal^\ex$ is generated by $\Ccal$ under colimits and the property that a comonoid be grouplike is closed under colimits, we may reduce to showing that every object of $\Ccal$ is a grouplike commutative comonoid in $\Ccal^\ex$. Since $\Ccal$ is additive  every object of $\Ccal$ is a grouplike commutative comonoid in $\Ccal$. We may thus reduce to showing that the inclusion $\Ccal \rightarrow \Ccal^\ex$ preserves finite coproducts. This is a direct consequence of the fact that $\Ccal$ and $\Ccal^\ex$ are semiadditive and the inclusion $\Ccal \rightarrow \Ccal^\ex$ preserves products.
\end{proof}

\begin{lemma}\label{base change of cotruncated}
Let $n \geq 1$ and let $\Ccal$ be an abelian $(n,1)$-category. Then the class of $(n-2)$-cotruncated morphisms in $\Ccal$ is stable under base change.
\end{lemma}
\begin{proof}
Assume given a cartesian square
\[
\begin{tikzcd}
X' \arrow{d}{f'} \arrow{r}{g'} & X \arrow{d}{f} \\ Y' \arrow{r}{g} & Y
\end{tikzcd} \rlap{\hspace{1cm}$(\ast)$}
\] 
where $f$ is $(n-2)$-cotruncated. We need to show that $f'$ is $(n-2)$-cotruncated. Observe that the map $f+g : X \oplus Y' \rightarrow Y$ is $(n-2)$-cotruncated. It follows that the fiber sequence 
\[
X' \xrightarrow{(g', -f')} X \oplus Y' \xrightarrow{f+g} Y
\]
 is also a cofiber sequence, and hence $(\ast)$ is a bicartesian square

Let $U$ be the cofiber of the map $\operatorname{fib}(f') \rightarrow X'$ and let $f'': X' \rightarrow U$ be the induced map. By remark \ref{remark fibers or cofibers are } we have that $f''$ is $(n-2)$-cotruncated.  We claim that the square
\[
\begin{tikzcd}
X' \arrow{d}{f''} \arrow{r}{g'} & X \arrow{d}{f} \\ U \arrow{r}{g} & Y
\end{tikzcd}   \rlap{\hspace{1cm}$(\ast\ast)$}
\]
is a pushout square. Since both $f$ and $f''$ are $(n-2)$-cotruncated, it will suffice to show that the induced map $\fib(f'') \rightarrow \fib(f)$ is an isomorphism.  By remark \ref{remark fibers or cofibers are } we have that the map $\fib(f') \rightarrow X'$ is $(n-2)$-truncated, and therefore we have $\fib(f'') = \fib(f')$. We may thus reduce to showing that the map $\fib(f') \rightarrow \fib(f)$ is an isomorphism, which follows from the fact that $(\ast)$ is cartesian.

Let $p: U \rightarrow Y'$ be the canonical map. Note that there is a morphism of commutative squares from $(\ast\ast)$ to $(\ast)$. Its cofiber is the square
\[
\begin{tikzcd}
0 \arrow{d}{} \arrow{r}{} & 0 \arrow{d}{} \\ \cofib(p) \arrow{r}{} & 0.
\end{tikzcd}
\]
Since both $(\ast)$ and $(\ast \ast)$ are pushout squares, the above is also a pushout square, from which it follows that $\cofib(p) = 0$. Applying remark \ref{remark fibers or cofibers are } once more we conclude that $p$ is $(-1)$-cotruncated, and in particular $(n-2)$-cotruncated. The lemma now follows from the fact that $f' = p f''$.
\end{proof}

\begin{proof}[Proof of theorem \ref{teo abelian n}]
We begin with a proof of (1). Assume first that $\Ccal$ is $n$-abelian. We first show that $\Ccal$ is regular. Let $f: X \rightarrow Y$ be a morphism in $\Ccal$. We use the equivalence of lemma \ref{lemma sobre DK}. The \v{C}ech nerve of $f$ is the image under $\dk$ of the map $g: \fib(f) \rightarrow X$, and its geometric realization exists and is given by $\cofib(g)$. Assume now given a base change $f': X' \rightarrow Y'$ of $f$, and let $g': \fib(f') \rightarrow X'$ be the induced map. We have a commutative diagram
\[\begin{tikzcd}
\fib(f') \arrow{r}{g'} \arrow{d}{=} & X' \arrow{r}{} \arrow{d}{} & \cofib(g') \arrow{d}{} \arrow{r}{} & Y' \arrow{d}{} \\
\fib(f) \arrow{r}{g} & X \arrow{r}{} & \cofib(g) \arrow{r}{} &  Y.
\end{tikzcd}
\]
Our task is to show that the rightmost square is cartesian. The comparison map $\mu: \cofib(g') \rightarrow \cofib(g) \times_Y Y'$ sits in a morphism of sequences
\[
\begin{tikzcd}
\fib(f') \arrow{r}{g'} \arrow{d}{ \id } & X' \arrow{d}{\id} \arrow{r}{} & \cofib(g')\arrow{d}{\mu} \\
\fib(f') \arrow{r}{g'} \arrow{r}{} & X' \arrow{r}{} & \cofib(g) \times_Y Y'.
\end{tikzcd}
\]
Here the top row is both a fiber and a cofiber sequence, and the bottom row is a fiber sequence. Applying lemma \ref{base change of cotruncated} we see that the map $X' \rightarrow \cofib(g) \times_Y Y'$ is $(n-2)$-cotruncated, and therefore the bottom row is also a cofiber sequence. The fact that $\mu$ is an isomorphism now follows from the fact that the left and middle vertical arrows in the above diagram are isomorphisms.

We now show that $\Ccal$ is $n$-exact. Let $X_\bullet$ be an $n$-efficient groupoid object in $\Ccal$. Once again we use the equivalence of lemma \ref{lemma sobre DK}. The groupoid $X_\bullet$ is the image of a morphism $f: U \rightarrow X_0$ under the map $\dk: \Ccal^{[1]} \rightarrow \Ccal^{\Delta^\op}$. We have $X_1 = U \oplus X_0$ and the map $X_1 \rightarrow X_0 \times X_0$ is given by the matrix
\[
\begin{pmatrix}
f & \id_{X_0} \\
0 & \id_{X_0}
\end{pmatrix}.
\]
The map $f$ is a base change of the above, and hence it is $(n-2)$-truncated. Therefore $f$ is the fiber of its cofiber, which implies that $X_\bullet$ is effective.

Assume now that $\Ccal$ is additive and $n$-exact. By proposition \ref{prop ex completion of ex n cat}, we have $\Ccal = (\Ccal^{\ex})_{\leq n-1}$. Combining theorem \ref{teo prestable iff exact and additive} with lemma \ref{lemma Aex is additive} we see that $\Ccal^\ex$ is a finitely complete prestable category. The fact that $\Ccal$ is $n$-abelian now follows from example \ref{example of abelian n}. 

We now prove (2). Assume first that $F$ is exact. Then $F$ is in particular left exact, so to prove regularity we only need to show that $F$ preserves geometric realizations of \v{C}ech nerves. This follows from lemma \ref{lemma sobre DK}, using the fact that $F$ is additive and preserves cofiber sequences. Assume now that $F$ is regular. Then  $F$ is in particular left exact, so we only need to show that $F$ is right exact. Let $F^\ex: \ccal^\ex \rightarrow \dcal^\ex$ be the induced functor, and note that we have a commutative square of categories
\[
\begin{tikzcd}
\ccal^\ex \arrow{r}{F^\ex}\arrow{d}{\tau_{\leq n-1}} & \dcal^\ex \arrow{d}{\tau_{\leq n-1}} \\
\ccal \arrow{r}{F} & \dcal.
\end{tikzcd}
\]
The vertical arrows are localization functors, so to prove that $F$ is right exact it suffices to show that $F^\ex$ is right exact. This follows from theorem \ref{teo prestable iff exact and additive}, by virtue of lemma \ref{lemma Aex is additive}.
\end{proof}

%%%%%%%%%%%%%%%%%%%%%%%%%%%%%%%%%%%%%%%%%%%%%%%%%%%%%%%%%%%%%%%%%%%%%%%%
%%%%%%%%%%%%%%%%%%%%%%%%%%%%%%%%%%%%%%%%%%%%%%%%%%%%%%%%%%%%%%%%%%%%%%%%
%%%%%%%%%%%%%%%%%%%%%%%%%%%%%%%%%%%%%%%%%%%%%%%%%%%%%%%%%%%%%%%%%%%%%%%%
%%%%%%%%%%%%%%%%%%%%%%%%%%%%%%%%%%%%%%%%%%%%%%%%%%%%%%%%%%%%%%%%%%%%%%%%
%%%%%%%%%%%%%%%%%%%%%%%%%%%%%%%%%%%%%%%%%%%%%%%%%%%%%%%%%%%%%%%%%%%%%%%%
%%%%%%%%%%%%%%%%%%%%%%%%%%%%%%%%%%%%%%%%%%%%%%%%%%%%%%%%%%%%%%%%%%%%%%%%

\subsection{Deriving abelian \texorpdfstring{$(n,1)$}{(n,1)}-categories}\label{subsection deriving}

We now specialize the theory of exact completions and hypercompletions to the additive context to obtain a theory of connective derived categories for abelian $(n,1)$-categories.

\begin{notation}\label{notation nab}
Let $\Cat_{\normalfont \text{pst}}$ be the subcategory of $\Cat$ on the finitely complete prestable categories and exact functors. We denote by $\Cat^{\normalfont \text{b}}_{\normalfont \text{pst}}$ the full subcategory of $\Cat_{\normalfont \text{pst}}$ on those finitely complete prestable categories $\ccal$ which are bounded, and by $\Cat_{\normalfont \text{pst}, \hyp}$ the full subcategory of $\Cat_{\normalfont \text{pst}}$ on those finitely complete prestable categories which are separated and admit geometric realizations. 

For each $n \geq 1$ we denote by $(n, 1)\kr\Cat_{\normalfont \text{ab}}$ the subcategory of $\Cat$ on the abelian $(n,1)$-categories and exact functors. We note that the assignment $\ccal \mapsto \ccal_{\leq n-1}$ assembles into a functor  $ \Cat_{\normalfont \text{pst}} \rightarrow  (n, 1)\kr\Cat_{\normalfont \text{ab}}$.
\end{notation}

\begin{theorem}\label{theorem deriving}
Let $n \geq 1$.
\begin{enumerate}[\normalfont (1)]
\item  The functor $(-)_{\leq n-1} : \Cat^{\normalfont \text{b}}_{\normalfont \text{pst}} \rightarrow  (n, 1)\kr\Cat_{\normalfont \text{ab}}$ admits a fully faithful left adjoint 
\[
\der^{\bounded}(-)_{\geq 0}:  (n, 1)\kr\Cat_{\normalfont \text{ab}} \rightarrow \Cat^{\normalfont \text{b}}_{\normalfont \text{pst}}
\]
that sends each abelian $(n,1)$-category $\ccal$ to $\ccal^\ex$. Furthermore, an object belongs to the image of $\der^{\bounded}(-)_{\geq 0}$ if and only if it is $(n-1)$-complicial.
\item The functor $(-)_{\leq n-1} : \Cat_{\normalfont \text{pst}, \hyp} \rightarrow  (n, 1)\kr\Cat_{\normalfont \text{ab}}$ admits a fully faithful left adjoint 
\[
\der(-)_{\geq 0}:  (n, 1)\kr\Cat_{\normalfont \text{ab}} \rightarrow \Cat_{\normalfont \text{pst}, \hyp}
\]
that sends each abelian $(n,1)$-category $\ccal$ to $\ccal^\hyp$. Furthermore, an object belongs to the image of $\der(-)_{\geq 0}$ if and only if it is $(n-1)$-complicial.
\end{enumerate}  
\end{theorem}

\begin{remark}
Let $n \geq 1$ and let $\ccal$ be an abelian $(n,1)$-category. Then one may define stable categories
\[
\der^\bounded(\ccal) = \colim ( \der^{\bounded}(\ccal)_{\geq 0} \xrightarrow{\Sigma} \der^{\bounded}(\ccal)_{\geq 0} \xrightarrow{\Sigma} \ldots )
\]
and 
\[
\der^-(\ccal) = \colim ( \der(\ccal)_{\geq 0} \xrightarrow{\Sigma} \der(\ccal)_{\geq 0} \xrightarrow{\Sigma} \ldots ).
\]
In the language of \cite{SAG} C.1.1, these are the Spanier-Whitehead categories of $\der^\bounded(\ccal)_{\geq 0}$ and $\der(\ccal)_{\geq 0}$. It follows from \cite{SAG} proposition C.1.2.9 that $\der^\bounded(\ccal)$ and $\der^-(\ccal)$ admit t-structures whose connective halves are given by $\der^\bounded(\ccal)_{\geq 0}$ and $\der(\ccal)_{\geq 0}$.

Assume now that $n = 1$. Then it follows from theorem \ref{theorem deriving} that $\der^\bounded(\ccal)$ is the unique stable category equipped with a t-structure and an equivalence  $\der^\bounded(\ccal)^\heartsuit = \ccal$ satisfying the following conditions:
\begin{itemize}
\item Every object in $\der^\bounded(\ccal)$ is $k$-connective and $k$-truncated for some $k$.
\item Let $X$ be a connective object in $\der^\bounded(\ccal)$. Then there exists an object $X'$ in $\ccal$ and a morphism $X' \rightarrow X$ inducing an epimorphism on $H_0$.
\end{itemize}

Similarly, $\der^-(\ccal)$ is the unique stable category equipped with a t-structure and an equivalence  $\der^-(\ccal)^\heartsuit = \ccal$ satisfying the following conditions:
\begin{itemize}
\item Every object in $\der^-(\ccal)$ is $k$-connective for some $k$.
\item Every $\infty$-connective object of $\der^-(\ccal)$ is $0$ (in other words, the t-structure on $\der^-(\ccal)$ is left separated).
\item The full subcategory of $\der^-(\ccal)$ on the connective objects admits geometric realizations of simplicial objects.
\item Let $X$ be a connective object in $\der^-(\ccal)$. Then there exists an object $X'$ in $\ccal$ and a morphism $X' \rightarrow X$ inducing an epimorphism on $H_0$.
\end{itemize}

It follows from the above description that  $\der^\bounded(\ccal)$ (resp. $\der^-(\ccal)$) is the stable enhancement of the  classical bounded (resp. bounded above) derived category of $\ccal$. We refer to \cite{AntieauEnhancement} for a discussion of the uniqueness of such enhancements.
\end{remark}

Before going into the proof of theorem \ref{theorem deriving} we record some consequences.

\begin{corollary}
Let $n \geq 1$ and let $\ccal$ be an abelian $(n,1)$-category. Then there exists an $(n-1)$-truncated connective ring spectrum $A$ and a fully faithful exact embedding $\ccal \rightarrow (\LMod_A^\cn)_{\leq n-1}$.
\end{corollary}
\begin{proof}
Applying corollary \ref{coro freyd mitchell} we obtain a connective ring spectrum $B$ and a fully faithful exact functor $\der(\ccal)_{\geq 0} \rightarrow \LMod_B^\cn$. This restricts to a fully faithful exact functor $\ccal \rightarrow (\LMod_B^\cn)_{\leq n-1}$. The corollary now follows by setting $A = \tau_{\leq n-1}(B)$.
\end{proof}

\begin{corollary}\label{coro abelian iff ind abelian}
Let $n \geq 1$. A finitely complete category $\ccal$ is an abelian $(n,1)$-category if and only if $\Ind(\Ccal)$ is an abelian $(n,1)$-category.
\end{corollary}
\begin{proof}
Assume first that $\Ccal$ is an abelian $(n,1)$-category. Then  $\Ind(\der(\ccal)_{\geq 0})$ is a Grothendieck prestable category, and we have
\[
 \Ind(\der(\ccal)_{\geq 0})_{\leq n-1} = \Ind((\der(\ccal)_{\geq 0})_{\leq n-1}) = \Ind(\ccal).
\]
The fact that $\Ind(\ccal)$ is an abelian $(n,1)$-category now follows from example \ref{example of abelian n}.

Assume now that $\Ind(\Ccal)$ is an abelian $(n,1)$-category. Then $\ccal$ is an $(n,1)$-category as well, and since $\ccal$ is finitely complete we deduce that $\ccal$ is idempotent complete. It follows that $\ccal$ agrees with the full subcategory of $\Ind(\ccal)$ on the compact objects. Since  $\Ind(\ccal)$ admits finite colimits and $\ccal$ admits finite limits we have that $\ccal$ is closed under finite limits and colimits inside $\Ind(\ccal)$. The result now follows from example \ref{example closed under finite limits and colimits is ab}.
\end{proof}

\begin{example}
Let $n \geq 1$ and let $A$ be an $(n-1)$-truncated connective ring spectrum. Then the category $\LMod_A^\cn$ of connective  left $A$-modules is a separated Grothendieck prestable category. Since $\LMod_A^\cn$ is generated under colimits by the $(n-1)$-truncated object $A$, we see that $\LMod_A^\cn$ is $(n-1)$-complicial. It follows from theorem \ref{theorem deriving} that we have equivalences
\[
\LMod_A^\cn = \der((\LMod_A^\cn)_{\leq n-1})_{\geq 0} = ((\LMod_A^\cn)_{\leq n-1})^\hyp. 
\]

Assume now that $A$ is left coherent (in other words, $\pi_0(A)$ is a left coherent ring and $\pi_k(A)$ is a finitely presented left $\pi_0(A)$-module for all $k$) and let $\LMod_A^{\cn, \text{aperf}}$ be the full subcategory of $\LMod_A^\cn$ on the almost perfect connective left $A$-modules. Then it follows from \cite{HA} proposition 7.2.4.18 that $\LMod_A^{\cn, \text{aperf}}$ is a finitely complete prestable category. This is in fact separated and admits geometric realizations, so it defines an object in $\Cat_{\normalfont \text{pst}, \hyp}$. An application of part (5) of \cite{HA} proposition 7.2.4.11 shows that $\LMod_A^{\cn, \text{aperf}}$ is $(n-1)$-complicial. It follows from theorem \ref{theorem deriving} that we have equivalences
\[
\LMod_A^{\cn, \text{aperf}} =  \der((\LMod_A^{\cn,\text{aperf}})_{\leq n-1})_{\geq 0} = ((\LMod_A^{\cn,\text{aperf}})_{\leq n-1})^\hyp. 
\]
\end{example}

We now turn to the proof of theorem \ref{theorem deriving}.
 
\begin{lemma}\label{lemma Chyp is additive}
Let $\ccal$ be an additive category. Then $\ccal^\hyp$ is additive.
\end{lemma}
\begin{proof}
Since the projection $\ccal^{\hyp} \rightarrow \widehat{\Ccal^\hyp}$ is conservative and preserves finite products and coproducts, it is enough to show that  $\widehat{\Ccal^\hyp}$ is additive. This is the limit of the categories $(\ccal^\hyp)_{\leq k}$ along the truncation functors. Since these functors preserve finite products and coproducts we may further reduce to showing that $(\ccal^\hyp)_{\leq k}$ is additive for all $k \geq 0$. Applying proposition \ref{prop regular in sheaf topos} we have
\[
(\ccal^\hyp)_{\leq k} = \Sh(\ccal)^\reg \cap \Sh(\ccal)_{\leq k} = (\ccal^\ex)_{\leq k}.
\]
The category $(\ccal^\ex)_{\leq k}$ is a localization of $\ccal^\ex$, and the localization functor preserves finite products. The fact that $(\ccal^\hyp)_{\leq k} $ is additive now follows from lemma \ref{lemma Aex is additive}.
\end{proof}

\begin{proof}[Proof of theorem \ref{theorem deriving}]
In light of theorems \ref{teo prestable iff exact and additive} and \ref{teo abelian n} and corollaries \ref{corollary fully faithful exact on n1} and 
\ref{corollary fully faithful hypercomplete on n1}, the theorem will follow if we are able to show that if $\ccal$ is an abelian $(n,1)$-category then $\ccal^\ex$ and $\ccal^\hyp$ are additive. This follows from lemmas \ref{lemma Aex is additive} and \ref{lemma Chyp is additive}. 
\end{proof}

\ifx\inmain\undefined
\bibliographystyle{myamsalpha2}
\bibliography{References}
\fi

\bibliographystyle{myamsalpha2}
\bibliography{References}
  
\end{document}